\documentclass{article}
\usepackage[a4paper]{geometry}
\usepackage{graphicx,amsmath,amsfonts,amsthm} 
\usepackage{xcolor} 
\usepackage{dsfont}
\usepackage{bm} 
\usepackage{enumitem} 
\usepackage{comment}
\usepackage{todonotes}
\theoremstyle{definition}
\newtheorem{definition}{Definition}

\newtheorem{proposition}{Proposition}
\newtheorem{lemma}{Lemma}
\newtheorem{assumption}{Assumption}
\newtheorem{remark}{Remark}
\newtheorem{theorem}{Theorem}

\newcommand{\bbbP}{{\bar{\mathbb{P}}}}

\newcommand{\Rr}{{\mathbb{R}}}

\newcommand{\Ee}{{\mathbb{E}}}
\newcommand{\E}{{\mathbb{E}}}

\newcommand{\rmx}{{\mathrm{x}}}

\newcommand{\bdelta}{{\bm{\delta}}}

\newcommand{\R}{\mathbb R}

\title{Fast and slow mean-field games}
\author{Roxana Dumitrescu\\ CREST, ENSAE, Institut Polytechnique de Paris \and  Julian Gutierrez Pineda\\ Division of engineering, New York University Abu Dhabi \and  Peter Tankov\footnote{Corresponding author, email: \texttt{peter.tankov@ensae.fr}} \\ CREST, ENSAE, Institut Polytechnique de Paris}
\date{}

\begin{document}

\maketitle


\begin{abstract}
\textcolor{black}{We propose a framework for constructing approximate Nash equilibria in mean-field games (MFG) with common noise based on a two-time-scale structure. In our model, the common noise is carried by a fast variable evolving under ergodic dynamics, while the slow variable is either optimally controlled or stopped. The fast variable enters the dynamics of the slow one through the drift coefficient. The key idea is to construct an approximation to the solution of the full MFG with common noise using an ``effective'' MFG without common noise, where the coefficients are averaged with respect to the stationary measure of the fast-scale process. We construct an explicit $\varepsilon$-MFG equilibrium for the full MFG from the equilibrium for the effective MFG with randomized control and stopping. To this end, we obtain new results on existence of MFG equilibria with randomized stopping. We rely on strong convergence results for two-scale diffusions under structural assumptions on the MFG, and show that the time-scale separation parameter controls the error in the Nash equilibrium condition.}
\end{abstract}



\section{Introduction}

\textcolor{black}{We introduce a framework for constructing approximate Nash equilibria in mean-field games with common noise without solving the full common-noise MFG. In contrast to the master equation approach, which characterizes equilibria through infinite-dimensional partial differential equations, our method reduces the original problem to an effective mean-field game without common noise. An equilibrium of the effective problem is then used to construct an approximate equilibrium for the original two-scale system. In particular, we establish the existence of strong randomized $\varepsilon$-equilibria when the time-scale separation parameter is sufficiently small.}

\textcolor{black}{Mean-field games with common noise model systems in which a large population of agents interacts while being jointly affected by a common source of uncertainty; see, among others, \cite{bassou2024mean,burzoni2023mean,carmona2015mean,dumitrescu2024energy,tangpi2024optimal}. The common noise may represent aggregate economic shocks, environmental conditions, or systemic fluctuations. Its presence, however, substantially complicates the analysis. The distribution of the representative agent becomes a random conditional law, while the value function and the equilibrium measure flow are intrinsically infinite-dimensional objects. Consequently, standard compactness and fixed-point arguments do not generally apply directly.}

\textcolor{black}{Many applications also involve dynamics evolving on distinct time scales. Long-term investment, capacity, or strategic decisions may interact with rapidly fluctuating operational, financial, or environmental factors. This motivates the study of multiscale MFGs with common noise, in which agents are simultaneously exposed to shared randomness and to a separation between slow and fast dynamics. The analytical difficulty is therefore twofold. Common noise makes the equilibrium distribution random, whereas time-scale separation requires controlling the interaction between the slow and fast variables. }

\textcolor{black}{In this paper, we focus on two-scale mean-field games with common noise, where the common noise enters only through the fast ergodic component. This is motivated by economic applications, where the fast scale models external factors, such as weather, which affect all firms at once, whereas the slow scale describes internal processes of the firm such as investments and productivity increases, which are less correlated among different companies. } \textcolor{black}{Stochastic averaging can then be used to eliminate the fast variable and to derive the effective dynamics for the slow component. This leads to an MFG without common noise, for which equilibria can be constructed using existing methods (e.g., the continuous-time linear programming approach).}

\textcolor{black}{We consider both optimal stopping and optimal control problems. In each setting, the representative agent has two state variables evolving on different time scales. The slow component $X^\delta$ describes the long-term evolution of the system, whereas the fast component $Y^\delta$ captures short-term fluctuations. The parameter $\delta>0$ quantifies the separation between the two time scales. The dynamics are given by
\[
\left\{
\begin{aligned}
    dX^\delta_t
    &=
    \widetilde b_t\,dt
    +
    \sigma(t,X^\delta_t)\,dW^{\mathrm{x}}_t,\\[0.2cm]
    dY^\delta_t
    &=
    \frac{1}{\delta}B(Y^\delta_t)\,dt
    +
    \frac{1}{\sqrt{\delta}}\Sigma(Y^\delta_t)\,dW^{\mathrm{y}}_t
    +
    \frac{1}{\sqrt{\delta}}\Sigma(Y^\delta_t)\,dW^{\mathrm{c}}_t.
\end{aligned}
\right.
\]
Here, $W^{\mathrm{x}}$ and $W^{\mathrm{y}}$ are idiosyncratic Brownian motions, while $W^{\mathrm{c}}$ is the common noise. The two scales are coupled through the drift of the slow component.}

\textcolor{black}{In the optimal stopping formulation,
\[
\widetilde b_t=b(t,X^\delta_t,Y^\delta_t),
\]
and the representative agent chooses a stopping time $\tau$ to optimize
\[
\mathbb{E}\left[
    \int_0^\tau
    f(t,X^\delta_t,Y^\delta_t,m^\delta_t)\,dt
    +
    g(\tau,X^\delta_\tau,\mu^\delta)
\right].
\]
Here, $m^\delta_t$ denotes the conditional distribution, given the common noise, of the states of the agents who remain active at time $t$, while $\mu^\delta$ is the joint conditional distribution of stopping times and stopped states.
In the optimal control formulation,
\[
\widetilde b_t=b(t,X^\delta_t,Y^\delta_t,\alpha_t),
\]
and the representative agent chooses a control process
$(\alpha_t)_{0\leq t\leq T}$ to optimize
\[
\mathbb{E}\left[
    \int_0^T
    f(t,X^\delta_t,Y^\delta_t,\alpha_t,m^\delta_t)\,dt
    +
    g(X^\delta_T,m^\delta_T)
\right],
\]
where $m^\delta_t$ denotes the conditional distribution of
$(X^\delta_t,Y^\delta_t)$ given the common noise. Thus, in both formulations, the equilibrium measure flow is random and adapted to the filtration generated by the common noise.}

\textcolor{black}{The effective problem is obtained by averaging the coefficients and objective functions with respect to the invariant distribution of the fast ergodic process. We first construct an equilibrium for this averaged MFG and then use the corresponding randomized strategy to define a candidate equilibrium for the original two-scale problem. Strong stochastic averaging estimates allow us to control the discrepancy between the original and effective state processes and, consequently, the loss of optimality under unilateral deviations.}

\textcolor{black}{Our analysis relies on the stochastic averaging frameworks of \cite{rockner2021strong} and \cite{rockner2021averaging}, which provide strong convergence estimates of orders $1/2$ and $1/6$, respectively, under suitable regularity and ergodicity assumptions. Strong convergence is particularly important in the game-theoretic setting, since convergence in distribution alone does not generally provide sufficient control of the payoff under unilateral deviations. These estimates allow us to translate the equilibrium property of the effective problem into an approximate Nash equilibrium condition for the original multiscale game.}

\textcolor{black}{A key structural assumption is that the volatility coefficient of the slow component $X^\delta$ does not depend on the fast variable $Y^\delta$. Example 4.1 in \cite{Liu2010} shows that this independence is, in general, necessary for strong convergence of the slow process. We also assume that the dynamics of the fast component are independent of the slow variable. This condition guarantees the regularity, and in particular the Lipschitz continuity, of the averaged coefficients required to compare the original and effective systems.}

\textcolor{black}{To establish the existence of equilibria for the effective MFG, we build on the continuous-time linear programming relaxation developed in \cite{1Bouveret,4LPApproach,5FP,dumitrescu2024energy}. The LP formulation replaces direct optimization over controls or stopping times with optimization over occupation measures satisfying linear constraints. It provides a robust existence theory, including in settings where equilibria in pure strategies may fail to exist; see \cite{bertucci2018optimal}, Proposition 1.3. However, LP solutions are formulated in terms of measure-valued objects and must be connected to admissible randomized strategies before they can be used to approximate the original stochastic game.}

\textcolor{black}{For this purpose, we introduce a notion of \textit{strong randomized equilibrium} adapted to the multiscale setting. By a strong randomized equilibrium, we mean a mean-field equilibrium expressed in terms of randomized strategies, for which the corresponding equilibrium measure flow is adapted to the filtration generated by the common noise. In the stopping problem, the representative agent uses a randomized stopping strategy, whereas in the control problem the agent uses a randomized control strategy. This formulation is related to the probabilistic equilibrium notions developed in \cite{carmona2018probabilistic,carmona2016mean,carmona2017mean}, while being specifically tailored to the approximation procedure considered in this paper. In particular, our argument does not require an independent existence result for an exact equilibrium of the full common-noise MFG.}

\textcolor{black}{The main contributions of the paper are the following:
\begin{itemize}
    \item We derive an effective MFG without common noise from the original two-scale problem. The reduction relies on the ergodicity of the fast component and on the fact that the common noise acts through this component. It provides a tractable alternative to solving the full common-noise MFG or the associated infinite-dimensional master equation.
\item We prove the existence of \textit{strong randomized $\varepsilon$-equilibria} for the original two-scale MFG. More precisely, for every $\varepsilon>0$, we show that an equilibrium of the effective MFG induces a strong randomized $\varepsilon$-equilibrium of the original problem whenever $\delta$ is sufficiently small. The Nash error is controlled quantitatively through the strong stochastic averaging estimates of \cite{rockner2021averaging,rockner2021strong} and converges to zero as $\delta\to0$. To the best of our knowledge, this provides a new quantitative construction of approximate equilibria for MFGs combining common noise and time-scale separation.
\item We establish a connection between randomized equilibria and the continuous-time LP formulation of MFGs. In particular, we show how solutions of the LP problem can be represented by randomized stopping or control strategies and interpreted as equilibria of the effective MFG. This representation provides the existence mechanism required for the subsequent approximation of the two-scale problem.
\item We develop two complementary methods for handling the dependence of the coefficients and payoffs on the fast-scale variable and on the associated occupation measures. When the fast component does not interact with the occupation measures, the required estimates are obtained through a time-discretization argument combined with the Poisson equation for the fast process. When such interactions are present, we combine the discretization procedure with the linear structure of the dynamics and the exponential ergodicity of the fast component.
\end{itemize}
Combining continuous-time linear programming with stochastic averaging therefore yields a constructive route from an equilibrium of the effective MFG to a strong randomized approximate equilibrium of the original common-noise problem. This strategy underlies the main existence results, Theorems \ref{thm:Existence 2S epsilon MFG eq} and \ref{relaxed1}, and provides quantitative control of the Nash error as the separation between the slow and fast time scales increases.}

As a potential application of our results, one could consider an extension of the electricity market industry dynamics model of \cite{3aid} allowing for common noise. In this model, two types of agents are present: renewable producers facing entry decisions, and conventional producers aiming to exit the market at the right time, and interacting through the market price of electricity. Entry and exit decisions are driven by long-term factors such as technology and labor costs, evolving on a \emph{slow scale}. In contrast, the revenues of renewable plants depend on short-term fluctuations in generation and demand, which evolve on a \emph{fast scale}. Because producers within the same market zone are exposed to common demand and correlated renewable output, modeling common noise is essential.

The remainder of the paper is organized as follows. Section \ref{sec:Notation} introduces the notation. Section \ref{Two-scale Mean-field game of optimal stopping} develops the optimal stopping framework, with assumptions in Section \ref{sec:OS Notation and Assump}, and presents the main result of this paper in the optimal stopping case (Theorem \ref{thm:Existence 2S epsilon MFG eq}). Section \ref{Sec:OS existence effective} proves the existence of randomized equilibria (Theorem \ref{thm:exist_effective}) and randomized approximate equilibria (Theorem \ref{thm:Existence 2S epsilon MFG eq randomized}). Section \ref{Sec:Control} turns to controlled diffusions, introducing the framework in Section \ref{Sec:Control intro} and proving Theorems \ref{relaxed1} and \ref{pro:epsilon Nash strong--c} in Section \ref{Sec:Control existence e-MFG}, with further results on the existence of equilibria for the effective problem in Section \ref{Sec:Control existence effective}. Technical material is collected in the Appendix.

\subsection*{Review of relevant literature}The combination of time-scale separation and common noise has not been previously addressed in the MFG literature. Related works on singular perturbed MFGs without noise include \cite{mendico2023singular,  cannarsa2025rate}, which consider MFGs with control of acceleration. 

For MFGs with common noise and controlled dynamics, \cite{ahuja2018asymptotic} studies small-noise asymptotics, showing that an approximate Nash equilibrium can be obtained by perturbing the control in the zero-noise limit (e.g., Theorem 5). We adopt from this work the notion of approximate equilibrium (see their Definition 2 and our Definitions \ref{def:Strong Nash2s--e--strict stop}, \ref{def:2S e--MFG eq with randomized stop}, and \ref{def:Strong Nash2s--strict c}). 
From the master equation perspective, \cite{meynard2024study} analyzes MFGs where in addition to an exogenous common noise, an endogenous common noise is introduced through an additional variable. The analysis relies on the notion of Lipschitz solutions to the master equation \cite{bertucci2024lipschitz}. In contrast, our construction circumvents solving the master equation by relying on an approximation via an MFG without common noise. 

Other approaches to studying MFGs with common noise and controlled dynamics include viscosity solutions frameworks in \cite{zhou2024viscosity} and strong solutions formulations without idiosyncratic noise in \cite{cardaliaguet2022first}. MFGs of optimal stopping with common noise have been less explored. \cite{he2023limit} studies stopped non--Markovian McKean--Vlasov state dynamics under common noise and shows propagation of chaos results. Some stylized models admit tractable analysis under reduced settings, such as finite state spaces or discrete time \cite{nutz2018mean, bassou2024mean, belak2021continuous, escribe2024mean}. 

Relaxation methods for optimal stopping MFGs without common noise have been studied through viscosity solutions on flows of joint marginal distributions \cite{talbi2023viscosity, talbi2024finite, talbi2023dynamic}, and through  state space enlargement in \cite{cosso2025mean}. Randomized equilibria with common noise generated by countable partitions are analyzed in \cite{ferrari2025existence}. Controlled martingale problem frameworks \cite{lacker2015mean} and MFGs of controls with interaction in both states and controls \cite{djete2023mean} provide existence and convergence results without common noise. In contrast to the previous approaches, the LP framework is directly amenable to numerical implementation, a feature crucial for our approach, which yields estimates that can be readily applied in computations. 

In the context of two-scale controlled diffusions, we refer to \cite{alvarez2002viscosity, alvarez2007multiscale, Bardi_2011, Bardi_2010}. Homogenization techniques studied in \cite{papanicolau1978asymptotic, ferreira2020two, cesaroni2016homogenization, cacace2018ergodicproblemmeanfield, lions2019homogenizationbackwardforwardmeanfieldgames} address space-scale separation of periodic structures, which do not apply to our ergodic time-scale separation setting. 

Strong convergence results for diffusions have been obtained through averaging techniques. The approach in \cite{rockner2021averaging}, based on the Poisson equation method (\cite{rockner2021diffusion, crisan2022poisson, brehier2020orders, rockner2019strong}), achieves strong convergence of order $1/2$. In contrast, \cite{rockner2021strong} employs the classical discretization method (\cite{Liu2010, cerrai2009khasminskii, has1966stochastic, li2008averaging, veretennikov1991averaging, wang2012average}), yielding strong convergence of order $1/3$ for two–scale McKean–Vlasov equations; under additional assumptions, this method also recovers order $1/2$.

Averaging techniques have also proved effective in other contexts, including jump–diffusions \cite{givon2007strong}, asymptotic expansions \cite{fouque2003multiscale, fouque2003singular, fouque2011multiscale, de2021averaging}, and PDE–based approaches \cite{PardouxVeretennikov2003, Khasminskii2004}. Furthermore, \cite{crisan2022poisson} extended Poisson equation methods to establish a uniform-in-time averaging theorem for SDE systems with superlinearly growing coefficients.

\subsection*{Notation}\label{sec:Notation} 

We fix a time horizon $T>0$. The slow scale process $X^\delta$ and the fast scale process $Y^\delta$ take values in $\mathbb R$. Finally, in the case of mean-field games of control, let $A$ be a convex compact subset of $\mathbb R$, where the control takes values.

We consider a probability space $(\Omega, \mathcal{F}, \mathbb{P})$ that supports independent scalar Brownian motions $W^{\text{x}}$, $W^{\text{y}}$, and $W^{\text{c}}$, as well as independent random variables $X_0 \sim m_0^X$ and $Y_0 \sim m_0^{Y}$, where $m_0^X$ and $m_0^Y$ are probability distributions on $\mathbb R$. We write
\[
\mathbb{F}^{\bm{W}} = (\mathcal{F}_t^{\bm{W}})_{0 \leq t \leq T}, \quad \mathbb{F}^{\text{x}}=(\mathcal{F}_t^{\text{x}})_{0 \leq t \leq T}, \quad \mathbb{F}^{\text{c}} = (\mathcal{F}_t^{\text{c}})_{0 \leq t \leq T}
\]
for the complete natural filtrations generated respectively by
\[
(X_0, Y_0, W_t^{\text{x}}, W_t^{\text{y}},W_t^{\text{c}})_{0 \leq t \leq T}, \quad (X_0,W^{\text{x}}_t)_{0\leq t \leq T}, \quad (Y_0, W_t^{\text{c}})_{0 \leq t \leq T}.
\]
For notational convenience, we include $Y_0$ in the common-noise filtration. Since $Y$ evolves on the fast time scale and is ergodic under our assumptions, the effect of its initial condition becomes negligible in the averaging limit $\delta \to 0$.

For any measurable function $\varphi$ and any finite measure $\nu$, we write:
\[
\langle \varphi , \nu \rangle := \int \varphi(y)\,\nu(dy).
\]
If $m$ is a measure on a product space $\mathbb R \times \mathbb R$, we denote by $m^{(\text{x})}$ its marginal on $\mathbb R$, that is, 
\[
 m^{(\text{x})}(dx) = \int_{\Rr} m(dx,dy).
\]
Similarly, if $m$ is a measure on a product space $\mathbb R \times \mathbb R \times A$, we denote by $m^{(\text{x},\text{y})}$ its marginal on $\mathbb R \times \mathbb R$, and by $m^{(\text{x})}$ its marginal on $\mathbb R$, that is, 
\[
m^{(\text{x},\text{y})}(dx,dy) = \int_A m(dx,dy,da), \qquad m^{(\text{x})}(dx) = \int_{\Rr \times A} m(dx,dy,da).
\]
Finally, if $\mu$ is a measure on a product space  $[0,T]\times\mathbb R \times \mathbb R $, we denote by $\mu^{(\text{t},\text{x})}$ its marginal on $[0,T]\times \mathbb R$:
\[
 \mu^{(\text{t},\text{x})}(dt,dx) = \int_{\Rr} \mu(dt,dx,dy).
\]

For a Borel set $\mathcal C$, we denote by $\mathcal P(\mathcal C)$ the set of probability measures on $(\mathcal C,\mathcal B(\mathcal C))$ and by  $\mathcal P_{\text{sub}}(\mathcal C)$ the set of subprobability measures (positive measures with total mass less or equal to $1$) on $(\mathcal C,\mathcal B(\mathcal C))$. When the set $\mathcal C$ is a subset of a normed space, we denote by $\mathcal P_q(\mathcal C)$ the set of probability measures on  $(\mathcal C,\mathcal B(\mathcal C))$ with finite $q$-moment. 

Further, we let $\mathcal{V}(\mathcal C)$ denote the set of measurable flows $(m_t)_{0 \le t \le T}$ of finite positive Borel measures on $\mathcal{C}$ such that, for every Borel set $B \subseteq \mathcal{C}$, the map $t \mapsto m_t(B)$ is measurable and
\[
\int_0^T m_t(\mathcal{C})\, dt < \infty.
\]
Flows are identified $dt$-almost everywhere on $[0,T]$.

We denote by $\mathbb V(\mathcal C)$ the set of $\mathbb F^{\text{c}}$-progressively measurable processes $(m_t)_{0 \le t \le T}$ with values in $\mathcal P_{\text{sub}}(\mathcal C)$ such that
\[
\int_0^T m_t(\mathcal{C})\, dt < \infty \quad \text{a.s.}
\]

Finally, we denote by $\mathbb M(\mathcal C)$ the set of $\mathcal F_T^{\text{c}}$-measurable random variables with values in $\mathcal P([0,T]\times \mathcal C)$. For $\mu \in \mathbb M(\mathcal C)$, we denote by $\mu^c_t$ its conditional expectation with respect to $\mathbb F^c$:
\begin{align}
\mu^c_t(B):= \mathbb E[\mu(B)|\mathcal F^c_t],\quad B\in \mathcal B([0,T]\times \mathcal C). \label{muproj}
\end{align}

\section{Two-scale Mean-field game of optimal stopping}
\label{Two-scale Mean-field game of optimal stopping}

In this section, we discuss the case of two-scale MFG of optimal stopping, and show that one can obtain the existence of an $\varepsilon$-equilibrium for such games from an equilibrium for the \textit{effective} MFG problem. 

\subsection{Model and assumptions}\label{sec:OS Notation and Assump}

We consider the following two-scale stochastic system in {$\mathbb R\times \mathbb{R}$}:
\begin{align}
\label{eq:dynamics 2S}
\begin{cases}
dX^\delta_t = b(t,X^\delta_t,Y^\delta_t) dt + \sigma (t,X^\delta_t) dW^{\text{x}}_t, \quad X^\delta_0 = X_0 \sim m_0^X,
\\
dY^\delta_t = \frac{1}{\delta} B(Y^\delta_t) dt + \frac{1}{\sqrt{\delta}}\Sigma (Y^\delta_t) dW^{\text{y}}_t + \frac{1}{\sqrt{\delta}}\hat{\Sigma} (Y^\delta_t) dW^{\text{c}}_t , \quad Y_0 \sim m^Y_0,
\end{cases} \quad 0 \leq t \leq T,
\end{align}
where $m_0^X \in \mathcal P(\mathbb R)$ and $m_0^Y \in \mathcal P(\mathbb R)$.

By making the natural time change $t \mapsto t \delta$, the process $(\Tilde{Y}^\delta_t):=(Y^\delta_{t \delta})$ has the same distribution as the solution of
\begin{align}
\label{eq:Fast system}
dY_t =  B(Y_t) dt + \Sigma (Y_t) dW_t + \hat{\Sigma} (Y_t) d\hat{W}_t,
\end{align}
where $(W_t,\hat{W}_t) = (\delta^{-1/2}W^{\text{y}}_{t \delta},\delta^{-1/2}W^{\text{c}}_{t \delta})$. We define the generators
\[
\mathcal{L}_Y(y) := B(y) \partial_y + \tfrac{1}{2}( \Sigma^2(y) + \hat{\Sigma}^2(y) ) \partial_{yy}, \quad \mathcal{L}_X(t,x,y) = b(t,x,y)\partial_x + \tfrac{1}{2} \sigma^2 (t,x) \partial_{xx}.
\]
Under appropriate assumptions on $B$, $\Sigma$ and $\hat{\Sigma}$, the process \eqref{eq:Fast system} admits a unique invariant measure $\bar{\mu}$.

In the sequel, we shall work under one of the two alternative regimes, which correspond to two different techniques for obtaining convergence rates in the main proof: the Poisson equation regime and the explicit Ornstein-Uhlenbeck (OU) regime. These regimes comprise assumptions on the process coefficients (part 1) and the reward functionals (part 2). 
\begin{assumption}[Poisson equation regime, part 1]
\label{hyp:Dynamics_Poisson_X_Y}
\leavevmode
\begin{enumerate}[label=(\roman*)] 
    \item \label{hyp:RocknerAG2} There exists a constant $\lambda>0$ such that for all $x\in \mathbb R$, and $t\in[0,T]$, $\tfrac{1}{\lambda} \leq \sigma^2(t,x)\leq \lambda$

    \item \label{hyp:slow coeff-stopping}
    $b(\cdot,\cdot,y) \in C^{1/2,1}_b([0,T]\times \mathbb R)$ for all $y\in \Rr$, $\sigma \in C^{1/2,1}_b([0,T]\times \mathbb R)$ and there exists $C>0$ such that
    \begin{align*}
    & |b(t,x,y)-b(t,x',y')| \leq C(|x-x'|+|y-y'|), \quad (t,x,x',y,y') \in [0,T]\times \mathbb R^2 \times \Rr^2,
    \end{align*}
    and $m_0^X \in \mathcal{P}_q(\mathbb R)$ for some $q \geq 2$.
    
    \item \label{hyp:RocknerAs2} There exists a constant $\lambda>0$ such that for all $y\in \Rr$, $\tfrac{1}{\lambda} \leq \Sigma^2(y) + \hat{\Sigma}^2(y)\leq \lambda$;

    \item \label{hyp:RocknerAb2} $\lim_{|y| \to \infty} yB(y) = -\infty$;
  
    \item \label{hyp:fast coeff} $m^Y_0 \in \mathcal{P}_{q'}(\Rr)$ for some $q'>2$, and $B,\Sigma, \hat{\Sigma} \in C^1_b(\Rr)$.
\end{enumerate}
\end{assumption}

\begin{assumption}[Explicit OU regime, part 1]
\label{hyp:Dynamics_Discret_X_Y}

\leavevmode
\begin{enumerate}[label=(\roman*)] 
    \item \label{hyp:Dynamics_Discret_X_Y A1} 
    There exist a constant $C>0$ such that for all $(t_1,t_2,x_1,x_2,y_1,y_2) \in [0,T]^2 \times \mathbb R^2 \times \Rr^2$, 
    \begin{align*}
    &|b(t_1, x_1, y_1) - b(t_2, x_2, y_2)| + \|\sigma(t_1, x_1) - \sigma(t_2, x_2)\| \leq C \left(|t_1 - t_2| + |x_1 - x_2| + |y_1 - y_2| \right),
    \end{align*}
    $B(y) = \kappa(\theta - y)$, $\Sigma(y) = \sigma_1$, and $\hat{\Sigma}(y) = \sigma_2$ for some positive constants $\kappa, \theta , \sigma_1$, and $\sigma_2$. 

    \item \label{hyp:Dynamics_Discret_X_Y IC}
    {$m^X_0\in \mathcal{P}_{q}(\mathbb R)$, $m^Y_0 \in \mathcal{P}_{q'}(\Rr)$ for some $q,q'\geq 2$.}

    \item \label{hyp:Dynamics_Discret_X_Y A2} $\partial_t b$, $\partial_x b$, and $\partial_y b$ exist on $[0, T] \times \mathbb R \times \mathbb{R}$, and there exist $C>0$ and $\gamma^1 \in (0,1]$ such that
    \begin{align*}
    & \sup_{t \in [0,T],\,x \in \mathbb R} \Big( 
    |\partial_t b(t,x,y_1)-\partial_t b(t,x,y_2)| 
    + \|\partial_x b(t,x,y_1)-\partial_x b(t,x,y_2)\| 
    \\
    & \hskip2cm + \|\partial_y b(t,x,y_1)-\partial_y b(t,x,y_2)\| \Big) \le C |y_1-y_2|^{\gamma_1}.
    \end{align*}
\end{enumerate}    
\end{assumption}
Under Assumption \ref{hyp:Dynamics_Poisson_X_Y}, the dynamics \eqref{eq:Fast system} admit a unique invariant probability measure $\bar{\mu}$ (\cite{rockner2021averaging}, Assumptions $(\text{\bf A}_b)$ and $(\text{\bf A}_\sigma)$, Section 3.2 in \cite{Liu2010}, Theorem 4.1 and Corollary 4.4 in \cite{khasminskii2011stochastic}, and Section 4 in \cite{Bardi_2010}). Under Assumption \ref{hyp:Dynamics_Discret_X_Y}, the invariant measure is explicitly given by $\bar{\mu} = \mathcal{N}\left(\theta, (2 \kappa)^{-1} (\sigma_1^2 + \sigma_2^2)\right)$.

Our assumptions under both the Poisson equation regime and the explicit OU regime guarantee the existence of moments of the invariant  measure. The proof of the following result can be found in the Appendix. 
\begin{proposition}
\label{pro:E invariant m estimate}
Suppose that Assumptions \ref{hyp:Dynamics_Poisson_X_Y} or \ref{hyp:Dynamics_Discret_X_Y} hold true. Then, the unique invariant probability measure $\bar{\mu}$ of \eqref{eq:Fast system} satisfies $\int_{\Rr}|y|^{p'}\bar{\mu}(dy)<\infty$ for any $p'\geq 1$.
\end{proposition}


Under the setting of Assumption \ref{hyp:Dynamics_Discret_X_Y}, we extend the fast-scale dynamics to a two-dimensional setting together with its invariant distribution by considering the projection $\hat{Y}_t:= \Ee [ Y_t | \mathcal{F}^{\text{c}}_t ]$. The pair $(Y,\hat{Y})$ is then an Ornstein–Uhlenbeck process with bivariate normal invariant measure $\bar{\mu}'$ given by $\mathcal{N}(\theta I_2, \Sigma)$, where $\Sigma = \frac{1}{2 \kappa} \left( \begin{smallmatrix} \sigma_1^2 + \sigma_2^2 & \sigma_2^2 \\ \sigma_2^2 & \sigma_2^2 \end{smallmatrix} \right)$. {Moreover, we define $\tilde Y_t:=Y_t - \hat Y_t$ and observe that $\tilde Y$ is independent from $\hat Y$ and satisfies the dynamics
$$ 
d\tilde Y_t  = -\kappa \tilde Y_t\,dt + \sigma_1 dW_t.
$$
Similarly to the above, we define rescaled processes $\tilde Y^\delta$ and $\hat Y^\delta$.}

Throughout, for any function $f$ depending on the fast variable(s), we denote by $\bar{f}$ its average with respect to the corresponding invariant measure: 
\begin{align}
\label{eq:averaging}    
\bar{f}(\cdot) := \int_{\Rr} f(\cdot,y)\,\bar{\mu}(dy) \quad \text{or} \quad \bar{f}(\cdot) := \int_{\Rr} f(\cdot,y,y')\,\bar{\mu}'(dy,dy').
\end{align}
When $f$ depends only on $y$, these two definitions agree, so no ambiguity arises.

\begin{remark} 
Strong existence and uniqueness for \eqref{eq:dynamics 2S} follow directly from the global Lipschitz continuity and regularity conditions in Assumptions \ref{hyp:Dynamics_Poisson_X_Y}–\ref{hyp:slow coeff-stopping}, \ref{hyp:fast coeff}. Although the uniform ellipticity of $\sigma$ and $(\Sigma, \hat{\Sigma})$ in Assumption \ref{hyp:Dynamics_Poisson_X_Y}–\ref{hyp:RocknerAG2},\ref{hyp:RocknerAs2} is not strictly required for strong existence and uniqueness for \eqref{eq:dynamics 2S}, we impose it because it will later be essential for establishing parabolic regularity in the associated Poisson equations.
\end{remark}

We introduce the following notation
\begin{align*}
& J^\delta(\tau,m,\mu):=\mathbb{E}\left[\int_0^{\tau }  f(t,X_t^\delta,Y_t^\delta,m_t)dt+ g(\tau ,X_{\tau }^\delta,\mu^c_\tau)\right],
\end{align*}
where $(m_t)_{0\leq t\leq T}\in \mathbb V(\mathbb R \times \mathbb R)$, $\mu \in \mathbb M(\mathbb R \times \mathbb R)$, and we recall that $\mu^c_\tau$ is the projection of the measure $\mu$ on the filtration of the common noise, defined in \eqref{muproj}.

We impose one of the following two alternative sets of assumptions on the cost functions $f$ and $g$. 
\begin{assumption}[Poisson equation regime, part 2]
\label{hyp:Cost Poisson OS}
The functions $f$ and $g$ have the representation
\begin{align*}
& f(t,x,y,m) = F^1(t,x,y) F^2\left(t,x,\int_{\mathbb R}\hat{f}_1(x')m^{(\text{x})}(dx')\right)+F^3(t,x)\int_{\mathbb R \times \mathbb{R}}\hat{f}_2(x',y')m(dx',dy'), 
\\
& g(t,x,\mu) = G\left(t,x,\int_{[0,T]\times \mathbb R} \hat{g}(s,x')\mu^{(\text{t},\text{x})}(ds,dx')\right),
\end{align*}
where $F^1:[0,T]\times \mathbb R \times \mathbb{R} \to \mathbb{R}$, $F^2,G:[0,T]\times \mathbb R \times \mathbb{R} \to \mathbb{R}$, $F^3:[0,T]\times \mathbb R  \to \mathbb{R}$ are continuous and bounded, $F^2\in C_b^{1,2,1}([0,T]\times \mathbb R\times \Rr)$, $F^3$ is $\gamma^3$-H\"{o}lder continuous uniformly on $x$, with $\gamma^3 \geq 1/2$, $\hat{g}:[0,T] \times \mathbb R \to \mathbb{R}$, {$\hat{f}_1\in C^2_b(\mathbb R) $} and $\hat{f}_2:\mathbb R \times \Rr \to \Rr$ are continuous and bounded, and there exits $C>0$ such that for all $(t,x,x',y, y', \beta,\beta')\in [0,T]\times \mathbb R^2 \times \Rr^4$
\begin{align*}
&|F^1(t,x,y)-F^1(t,x',y')| + |\hat{f}_2(x,y)-\hat{f}_2(x',y')| \leq C \left(|x - x'|+|y - y'|\right), 
\\
& | F^2(t,x,\beta) - F^2(t,x',\beta')| + | G(t,x,\beta) - G(t,x',\beta')| \leq C\left(|x - x'| + |\beta - \beta'| \right),
\\
& | F^3(t,x) - F^3(t,x')| + |\hat{g}(t,x)-\hat{g}(t,x')| \leq C |x - x'| ,
\end{align*}
and $F^1(\cdot,\cdot,y) \in C^{1/2,1}_b([0,T]\times \mathbb R)$, $\hat{f}_2(\cdot,y) \in C^{1}_b(\mathbb R)$, for all $y\in \Rr$.
\end{assumption}
\begin{assumption}[Explicit OU regime, part 2]
\label{hyp:Cost discretization OS}
The functions $f$ and $g$ have the representation
\begin{align*}
  f(t,x,y,m) &= F\left(t,x,y,\int_{\mathbb R\times \mathbb{R}} \hat{f}(x',y')\,m(dx',dy')\right),\\
   g(t,x,\mu) &= G\left(t,x,\int_{[0,T]\times \mathbb R} \hat{g}(s,x')\mu^{(\text{t},\text{x})}(ds,dx')\right),
\end{align*}
where $\hat g$ and $G$ are as in Assumption \ref{hyp:Cost Poisson OS},   $F:[0,T]\times\mathbb R\times \Rr^2 \to \Rr$ is bounded and there exists $C>0$ such that for all $(t, t',x,x',y, y',\beta, \beta') \in [0,T]^2 \times \mathbb R^2 \times \Rr^4$ we have
\begin{align*}
& |F(t,x,y,z)-F(t',x',y',z')| \leq C \left( |t-t'| + |x-x'| + |y-y'| + |z-z'|\right),
\\
& |\hat{f}(x,y) - \hat{f}(x',y')| \leq C \left(|x-x'| + |y-y'|\right).
\end{align*}

\end{assumption}

\subsection{Two-scale $\varepsilon$-MFG equilibrium with \textit{strict stopping}} \label{sec:OS e-MFG strict}
In this section, we first introduce the notions of \textit{two-scale $\varepsilon$-MFG equilibrium with strict stopping}, \textit{two-scale MFG equilibrium with strict stopping}, and \textit{effective MFG equilibrium with strict stopping}.

\begin{definition}[Two-scale $\varepsilon$-MFG and MFG equilibrium with \textit{strict} stopping]\label{def:Strong Nash2s--e--strict stop}
Let $\varepsilon\geq 0$ and $\delta>0$. We say that $(\tau^{\varepsilon,\delta},m^{\varepsilon,\delta}, \mu^{\varepsilon,\delta})$ is a \textit{two-scale $\varepsilon$-MFG equilibrium with strict stopping} if:
\begin{itemize}
\item[(i)] {$\tau^{\varepsilon,\delta} \in \mathcal{T}(\mathbb{F}^{\textbf{W}})$ and} for all $\tau^{\prime} \in \mathcal{T}(\mathbb{F}^{\textbf{W}})$,
\begin{align*}
& J^\delta\!\left(\tau^{\varepsilon, \delta}, m^{\varepsilon,\delta}, \mu^{\varepsilon,\delta}\right) \geq J^\delta\!\left(\tau', m^{\varepsilon,\delta}, \mu^{\varepsilon,\delta}\right) - \varepsilon
\end{align*}
\item[(ii)] $(m^{\varepsilon,\delta}_t)_{0\leq t\leq T}\in \mathbb V(\mathbb R \times \mathbb R)$ and $\mu^{\varepsilon,\delta} \in \mathbb M(\mathbb R \times \mathbb R)$, such that 
\begin{align*}
& m^{\varepsilon,\delta}_t(B) = \Ee\left[ \textbf{1}_{B}(X^\delta_t,Y^\delta_t)\mathbf 1_{(t,T]}(\tau^{\varepsilon,\delta}) \bigg| \mathcal{F}^{\text{c}}_t\right]\ \text{a.s., } \quad B \in \mathcal{B}(\mathbb R \times \Rr),\,\, \quad t \in [0,T], \notag \\
& \mu^{\varepsilon, \delta}(C\times B )=\mathbb{E}\left[\textbf{1}_{B}(X_{\tau^{\varepsilon, \delta}}^{\delta},Y_{\tau^{\varepsilon, \delta}}^{\delta})\textbf{1}_{C} (\tau^{\varepsilon,\delta})\bigg| \mathcal{F}^{\text{c}}_T\right] \ \text{ a.s., } B \in \mathcal{B}(\mathbb R \times \Rr),\,\, C \in \mathcal{B}([0,T]).
\end{align*}
\end{itemize}
When $\varepsilon=0$, we drop the superscript and say that $(\tau^\delta,m^\delta,\mu^\delta)$ is a \textit{two-scale MFG equilibrium with strict stopping}. 
\end{definition}

Next, we introduce the notion of \textit{effective MFG equilibrium with strict stopping}. 
On the same probability space $(\Omega, \mathcal{F}, \mathbb{P})$, let $X^0$ be the strong solution of
\begin{align}
\label{eq:dynamics effective}
dX^0_t = \bar{b}(t,X^0_t) dt + \sigma (t,X^0_t) dW^{\text{x}}_t, \quad X^0_0 = X_0 \sim m_0^X, \quad 0\le t \leq T,
\end{align}
and denote by $\bar{\mathcal{L}}$ the generator associated with \eqref{eq:dynamics effective}.

Because the invariant measure $\bar{\mu}$ of \eqref{eq:Fast system} is independent of the slow scale, {and in view of Proposition \ref{pro:E invariant m estimate}}, either under both the Poisson equation regime and the explicit OU regime , we obtain uniform Lipschitz continuity of $\bar{b}$ on $x$; that is,
\begin{equation}
\label{eq:Lipschitz eff os}
|\bar{b}(t,x)-\bar{b}(t,x')| \leq \int_{\Rr} |b(t,x,y)-b(t,x',y)|\bar{\mu}(dy) \leq C|x-x'|, \quad (t,x,x') \in [0,T] \times \mathbb R.
\end{equation}
Similarly, by Proposition \ref{pro:E invariant m estimate}, there exists a constant $C>0$ such that $|\bar{b}(t,x)| \leq C(1+|x|)$ for all $(t,x) \in [0,T] \times \mathbb R$. Combined with \eqref{eq:Lipschitz eff os} and the properties of $\sigma$ {($C^{1/2;1}_b$ under Poisson equation regime or Lipschitz continuity in both variables under explicit OU regime, this yields existence and uniqueness of a strong solution to \eqref{eq:dynamics effective}.   

To study the existence of the approximate MFG equilibrium for the two-scale problem (cf. Definition \ref{def:Strong Nash2s--e--strict stop}), we introduce the effective MFG problem. Under Poisson equatio regime (Assumption \ref{hyp:Cost Poisson OS}), and for $m\in\mathcal V(\mathbb R)$, we set
\begin{align}\label{payoff1}
& \bar{f}(t,x,m):=\bar{F}^1(t,x)F^2\left(t,x,\int_{\mathbb R}\hat{f}_1(x')m(dx'))\right) + F^3(t,x)\int_{\mathbb R }\bar{\hat{f}}_2(x')m(dx').
\end{align}
Under explicit OU regime (Assumption \ref{hyp:Cost discretization OS}), the cost functional is instead
\begin{align}\label{payoff2}
\bar{f}(t,x,m) := \int_{\mathbb R^2 } F\left(t,x,y, \int_{\mathbb R}\hat{F}(x',y')m(dx')\right) \bar{\mu}'(dy, dy'),
\end{align}
where we denote
$$
\hat F_\delta(x,y) = \mathbb E[\hat f(x,Y^\delta_t)|\hat Y^\delta_t = y]\quad \text{and}\quad \hat F(x,y) = \lim_{\delta\to 0} \hat F_\delta(x,y).
$$ 
Let
\begin{align*}
& J^{0}\!\left(\tau,m,\mu\right):=\mathbb{E}\left[\int_0^{\tau }\bar{f}(t,X_t^0,m_t)dt+ g(\tau ,X_{\tau }^0,\mu)\right].
\end{align*}

\begin{definition}[Effective MFG equilibrium with \textit{strict} stopping]
\label{def:Effective MFG equilibrium with strict stopping}
We say that $(\tau,m,\mu)$ is an \textit{effective MFG equilibrium with strict stopping} if:
\begin{itemize}
\item[(i)] {$\tau \in \mathcal{T}(\mathbb{F}^{\text{x}})$} and for all $\tau' \in \mathcal{T}(\mathbb{F}^{\text{x}})$,
\begin{align*}
& J^{0}\!\left(\tau,m,\mu\right) \geq J^{0}\!\left(\tau',m,\mu\right)
\end{align*}
\item[(ii)] $(m_t)_{0 \leq t \leq T}\in \mathcal V(\mathbb R)$  and $\mu \in \mathcal P([0,T]\times \mathbb R)$, such that:
\begin{align*}
& m_t(B) = \Ee\left[ \textbf{1}_{B}(X^0_t)\mathbf 1_{(t,T]}(\tau)\right], \quad B \in \mathcal{B}(\mathbb R),\quad t \in [0,T], \notag \\
& \mu(C \times B)=\mathbb{E}\left[\textbf{1}_{B}(X^0_{\tau})\textbf{1}_{C} (\tau  )\right], \quad B \in \mathcal{B}(\mathbb R),\, C \in \mathcal{B}([0,T]).
\end{align*}
\end{itemize}
\end{definition}
\begin{remark}\textcolor{black}{Existence of an \textit{effective} MFG equilibria with \textit{strict stopping} is obtained under additional asssumptions for which sufficient conditions are provided in \cite{5FP}.}
\end{remark}
\textcolor{black}{Before turning to the main result of construction of an approximate equilibrium for the two-scale MFG of optimal stopping in pure strategies}, we collect several auxiliary results whose proof is given in the Appendix. 
\begin{proposition}
\label{pro:Strong rate sup}  Let $0 < \delta <1/2$. Suppose that Assumption \ref{hyp:Dynamics_Poisson_X_Y} holds. Then, for all $T>0$, there exists a constant $C>0$ independent of $\delta$ such that 
\[
\Ee \left[ \sup\limits_{0\leq t \leq T} |X^\delta_t - X^0_t| \right] \leq C \delta^{1/2}.
\]
Alternatively, suppose that Assumption \ref{hyp:Dynamics_Discret_X_Y} holds. Then, for all $T>0$, there exists a constant $C>0$ independent of $\delta$ such that
\[
\Ee\left[\sup_{0\leq t\leq T}|X_t^\delta-X_t^0|\right]\leq C\delta^{1/3}.
\]
\end{proposition}

Finally, we recall an auxiliary identity, whose proof is deferred to the Appendix.
\begin{proposition}[\textit{Independent enlargement of filtrations}]
\label{pro:sup equality} 
Let $(m_t)_{0 \leq t \leq T}\in \mathcal V(\mathbb R)$  and $\mu \in \mathcal P([0,T]\times \mathbb R)$. Then,
\begin{align*}
& \sup_{\tau \in \mathcal{T}(\mathbb{F}^{\text{x}})} J^0(\tau, m,\mu) = \sup_{\tau \in \mathcal{T}(\mathbb{F}^{\bm{W}})} J^0(\tau, m,\mu)
\end{align*}
\end{proposition}

We are now ready to present the main result of this section.
\begin{theorem}[Construction of a  strong two-scale $\varepsilon$-MFG equilibrium with \textit{strict stopping} from an effective MFG equilibrium]
\label{thm:Existence 2S epsilon MFG eq}
Suppose that either Poisson equation regime (Assumptions \ref{hyp:Dynamics_Poisson_X_Y} and \ref{hyp:Cost Poisson OS}) or Explicit OU regime (Assumptions \ref{hyp:Dynamics_Discret_X_Y} and \ref{hyp:Cost discretization OS}) holds. Let $(\tau^{\star},m^\star,\mu^\star)$ be an effective MFG equilibrium with strict stopping. For $0<\delta<1/2$, define  $( m^{\delta,\star},\mu^{\delta,\star})$ as follows:
\begin{align*}
& m^{\delta , \star}_t(B) := \Ee \left[ \textbf{1}_{B}(X^\delta_t,Y^\delta_t) \mathbf 1_{(t,T]}(\tau^{\star}) \bigg| \mathcal{F}^{\text{c}}_t \right], \quad B \in \mathcal{B}(\mathbb R \times \Rr), \quad t \in [0,T], \notag
\\
& \mu^{\delta , \star}(B) := \Ee\left[ \textbf{1}_{B}(\tau^{\star}, X^\delta_{\tau^\star}, Y^\delta_{\tau^\star}) \bigg| \mathcal{F}^{\text{c}}_T \right], \quad B \in \mathcal{B}([0,T]\times \mathbb R \times \Rr),
\end{align*}
where $(X^\delta, Y^\delta)$ is the strong solution of \eqref{eq:dynamics 2S}.

Then $(\tau^\star, m^{\delta,\star},\mu^{\delta,\star})$ is a two-scale $\varepsilon$-MFG equilibrium with strict stopping, with
\[
\varepsilon = 
\begin{cases}
C\,\delta^{1/6}, & \text{Poisson equation regime}, \\[4pt]
C\,\delta^{1/3}, & \text{Explicit OU regime},
\end{cases}
\]
for some $C>0$ independent of $\delta$. 
\end{theorem}
\begin{proof}
Let $\tau \in \mathcal{T}(\mathbb{F}^{\mathbf{W}})$. Since  $(\tau^\star,m^\star,\mu^\star)$ is an effective MFG equilibrium with strict stopping, Proposition \ref{pro:sup equality} applies, showing that
\begin{align}
\label{eq:Strong Second term bound}
J^0(\tau, m^\star, \mu^\star) - J^0(\tau^\star, m^\star, \mu^\star) \leq 0.
\end{align}
Therefore, 
\begin{align}\label{eq:Aux delta strong 1}
J^\delta(\tau,m^{\delta,\star},\mu^{\delta,\star})-J^\delta(\tau^\star,m^{\delta,\star}, \mu^{\delta, \star}) & \leq \Delta_\delta (\tau) + \Delta_\delta (\tau^\star) 
\end{align}
with $\Delta_\delta (\tau):=\left|J^\delta(\tau,m^{\delta,\star}, \mu^{\delta, \star})-J^{0}(\tau,m^{\star}, \mu^{\star})\right|$.

We next bound the remaining terms in \eqref{eq:Aux delta strong 1}, splitting the proof according to the two sets of assumptions. Throughout the following calculations, the constant $C$ may change from line to line but remains independent of $\delta$.

\paragraph{Assume Poisson equation regime (Assumptions \ref{hyp:Dynamics_Poisson_X_Y} and \ref{hyp:Cost Poisson OS}) holds.} We write 
\begin{align}
\label{eq:Strong first term main}
& \Delta_\delta(\tau)
= \left| \Ee \left[ \int_0^\tau \left( \tilde{A}_t + A_t \right) dt + B \right] \right|,
\end{align}
where
\begin{align*}
\tilde{A}_t & = F^1(t,X^\delta_t,Y^\delta_t) F^2\left(t,X^\delta_t, \langle \hat{f}_1, m^{\delta ,\star(\text{x})}_t\rangle\right) - \bar{F}^1(t,X^0_t) F^2\left(t,X^0_t,\langle \hat{f}_1, m^{\star}_t\rangle\right),
\\
A_t & = F^3(t,X^\delta_t) \langle \hat{f}_2, m^{\delta ,\star}_t\rangle - F^3(t,X^0_t) \langle \bar{\hat{f}}_2, m^{\star}_t\rangle,
\\
B& = G\left( \tau, X^\delta_\tau, \mathbb E\left[\langle \hat{g}, \mu^{\delta , \star(\text{t},\text{x})}\rangle|\mathcal F^c_\tau\right] \right) - G\left( \tau, X^0_\tau , \langle \hat{g}, \mu^{\star} \rangle \right).
\end{align*}
{\bf \textit{Step 1: Bound on $\tilde{A}_t$.}} Let us write
\begin{align}
\label{eq:tildeA expand}
\tilde{A}_t & = \tilde{A}^{1}_t  + \tilde{A}^{2}_t + \tilde{A}^{3}_t,
\end{align}
where
\begin{align}
\label{eq:tildeA1 expand}
\tilde{A}^{1}_t & = F^2\left(t,X^0_t, \langle \hat{f}_1, m^{\star}_t \rangle\right) \left( F^1(t,X^\delta_t,Y^\delta_t) - \bar{F}^1(t,X^\delta_t) \right),
\\
\tilde{A}^{2}_t & = F^2\left(t,X^0_t, \langle \hat{f}_1, m^{\star}_t \rangle\right) \left( \bar{F}^1(t,X^\delta_t) - \bar{F}^1(t,X^0_t) \right),\\
\tilde{A}^{3}_t & = F^1(t,X^\delta_t,Y^\delta_t) \left( F^2\!\left(t,X^\delta_t, \langle \hat{f}_1, m^{\delta, \star(\text{x})}_t \rangle\right)  - F^2\!\left(t,X^0_t, \langle \hat{f}_1, m^{\star}_t \rangle\right)\right).
\end{align}
The estimate for the integral involving $\tilde A^1$ follows from Lemma \ref{lemmaM}.
Since $F^2$ is bounded and $\bar{F}^1$ is Lipschitz, we estimate the integral involving $\tilde{A}^{2}_t$ using Proposition \ref{pro:Strong rate sup}.
For the integral involving $\tilde{A}^3_t$, the bound on $F^1$, the Lipschitz continuity of $F^2$ and $\hat{f}_1$, and Proposition \ref{pro:Strong rate sup} yield
\begin{align}
\label{eq:tildeA2 bound final}
& \left| \Ee \left[ \int_0^\tau \tilde{A}^{3}_t dt \right] \right|  \leq C \sqrt{\delta}.
\end{align}
Combining these estimates, we obtain 
\begin{align}
\label{eq:tildeA bound final}
\left| \Ee \left[ \int_0^\tau \tilde{A}_t dt \right] \right|  \leq C \sqrt{\delta}.
\end{align}
{\bf \textit{Step 2: Bound on $A_t$.}} Let us write
\begin{align}
\label{eq:A expand}
A_t & = A^{1}_t + A^{2}_t + A^3_t,
\end{align}
where
\begin{align*}A^{1}_t &= \langle \hat{f}_2, m^{\delta,\star}_t \rangle ( F^3(t,X^\delta_t) - F^3(t,X^0_t) ),\\
  A^{2}_t &= F^3(t,X^0_t)  ( \langle \bar{\hat{f}}_2, m^{\delta,\star}_t \rangle - \langle \bar{\hat{f}}_2, m^{\star}_t \rangle ),\\
 A^{3}_t & =  F^3(t,X^0_t) ( \langle \hat{f}_2, m^{\delta,\star}_t \rangle - \langle \bar{\hat{f}}_2, m^{\delta,\star}_t \rangle ).
\end{align*}                                                                                                                               
The boundedness of $\hat{f}_2$, the Lipschitz continuity of $F^3$, and Proposition \ref{pro:Strong rate sup} show that
\begin{align}
\label{eq:A1 bound 1}
\Ee\left[|A^{1}_t|\right] & \leq C \Ee\left[ \left| X^\delta_t - X^0_t \right| \right] \leq C \sqrt{\delta}.
\end{align}
Boundedness of $F^3$, Lipschitz continuity of $\bar{\hat{f}}_2$, and Proposition \ref{pro:Strong rate sup} yield
\begin{align}
\label{eq:A22 bound}
& \left| \Ee\! \left[ \int_0^\tau A^{2}_t dt \right] \right| \!\leq C \Ee\! \left[ \int_0^T \!\!\! \Ee[ | \bar{\hat{f}}_2(X^\delta_t) - \bar{\hat{f}}_2(X^0_t) | \big| \mathcal{F}^{\text{c}}_t ] dt \right] \!\leq C \int_0^T \!\!\!\Ee [ |X^\delta_t - X^0_t |] dt \leq C \sqrt{\delta}.
\end{align}
To bound the term $A^{3}_t$, divide $[0,T]$ into intervals of size $\Delta$ to be chosen later, and write $$
\theta^\Delta_t:=\sup\{n\Delta:\Delta n \leq t\}.
$$
Using the regularity of the coefficients in \eqref{eq:dynamics effective} provided by Assumption \ref{hyp:Dynamics_Poisson_X_Y}, we obtain the existence of a constant $C>0$ independent of $\delta$ such that
\[
\Ee [ | X^0_t - X^0_s |^2 ] \leq C |t-s|.
\]
Boundedness of $F^3$, H\"{o}lder continuity of the map $t \mapsto F^3(t,x)$ with exponent $ \gamma^3 \geq 1/2$, and Lipschitz continuity of the map $F^3$, combined with the previous bound, yield
\begin{align*}
\Ee \left[  \left|F^3(t,X^0_t)-F^3(\theta^\Delta_t,X^0_{\theta^\Delta_t}) \right| \right] &\leq \Ee \left[  \left|F^3(t,X^0_t)-F^3(t,X^0_{\theta^\Delta_t}) \right| \right] + \Ee \left[  \left|F^3(t,X^0_{\theta^\Delta_t})-F^3(\theta^\Delta_t,X^0_{\theta^\Delta_t}) \right| \right] \notag
\\
& \leq C \Ee \left[ \left|X^0_t-X^0_{\theta^\Delta_t}\right| \right] + C \left|t-\theta^\Delta_t\right|^{\gamma^3} \notag
\\
& \leq C\left(\Delta^{1/2} + \Delta^{\gamma^3} \right)  \leq C \Delta^{1/2}.
\end{align*}

Further, using the boundedness of $\hat f_2$ and $F^3$, we write: 
\begin{align*}
  &\left| \Ee \left[ \int_0^\tau A^{3}_t dt \right] \right|  \leq \Delta \mathbb E\left[\sup_{0\leq t\leq T} |A^3_t|\right] + \left| \Ee \left[ \int_0^{\theta^\Delta_\tau} A^{3}_t dt \right] \right|\\ &\leq C\Delta + \left| \Ee \left[ \int_0^{\theta^\Delta_\tau}  ( F^3(t,X^0_t)  - F^3(\theta^\Delta_t,X^0_{\theta^\Delta_t}) ) ( \langle \hat{f}_2, m^{\delta,\star}_t \rangle - \langle \bar{\hat{f}}_2, m^{\delta,\star}_t \rangle ) dt \right] \right| \\ &+ \left| \sum_{k=0}^{\lfloor T/\Delta \rfloor} \Ee \left[ \int_{k\Delta}^{(k+1)\Delta} \mathbf 1_{k\Delta\leq \theta^\Delta_\tau} F^3(k\Delta,X^0_{k\Delta})   ( \langle \hat{f}_2, m^{\delta,\star}_t \rangle - \langle \bar{\hat{f}}_2, m^{\delta,\star}_t \rangle ) dt \right] \right| \\ &\leq C\Delta^{1/2} + \left|\sum_{k=0}^{\lfloor T/\Delta \rfloor} \Ee \left[ \mathbf 1_{k\Delta\leq \theta^\Delta_\tau} F^3(k\Delta,X^0_{k\Delta}) \int_{k\Delta}^{(k+1)\Delta}  ( \langle \hat{f}_2, m^{\delta,\star}_t \rangle - \langle \bar{\hat{f}}_2, m^{\delta,\star}_t \rangle ) dt \right]\right| \\
                                                           &\leq C\Delta^{1/2} + \left|\sum_{k=0}^{\lfloor T/\Delta \rfloor} \Ee \left[ \mathbf 1_{k\Delta\leq \theta^\Delta_\tau} F^3(k\Delta,X^0_{k\Delta}) \int_{k\Delta}^{(k+1)\Delta}  \mathbb E[ (  \hat{f}_2(X^\delta_t,Y^\delta_t) - \bar{\hat{f}}_2(X^\delta_t)  ) \mathbf 1_{(t,T]}(\tau^\star)|\mathcal F^c_t] dt \right]\right|\\
  &\leq C\Delta^{1/2} + \left|\sum_{k=0}^{\lfloor T/\Delta \rfloor} \Ee \left[ \mathbf 1_{k\Delta\leq \theta^\Delta_\tau} F^3(k\Delta,X^0_{k\Delta}) \mathbb E\left[\int_{k\Delta}^{(k+1)\Delta}  (  \hat{f}_2(X^\delta_t,Y^\delta_t) - \bar{\hat{f}}_2(X^\delta_t)  ) \mathbf 1_{(t,T]}(\tau^\star)dt|\mathcal F^{\mathbf W}_{k\Delta} \right] \right]\right|\\
  &\leq C\Delta^{1/2} + C\sum_{k=0}^{\lfloor T/\Delta \rfloor} \Ee \left[\left| \mathbb E\left[\int_{k\Delta}^{(k+1)\Delta}  (  \hat{f}_2(X^\delta_t,Y^\delta_t) - \bar{\hat{f}}_2(X^\delta_t)  ) \mathbf 1_{(t,T]}(\tau^\star)dt|\mathcal F^{\mathbf W}_{k\Delta} \right]\right| \right]\\
  &\leq C\Delta^{1/2} + C\sum_{k=0}^{\lfloor T/\Delta \rfloor} \Ee \left[\left(\int_{k\Delta}^{(k+1)\Delta}  (  \hat{f}_2(X^\delta_t,Y^\delta_t) - \bar{\hat{f}}_2(X^\delta_t)  ) \mathbf 1_{(t,T]}(\tau^\star)dt\right)^2\right]^{\frac{1}{2}}
\end{align*}
By Lemma \ref{lem:supL2_H} we then conclude that
$$
\left| \Ee \left[ \int_0^\tau A^{3}_t dt \right] \right|  \leq C \Delta^{1/2} + \frac{C\sqrt{\delta}}{\Delta}. 
$$
The sharpest rate, is obtained by choosing $\Delta = \delta^\frac{1}{3}$. With this choice, we get
\begin{align}
\label{eq:A21 bound final}
\left| \Ee \left[ \int_0^\tau A^{3}_t dt \right] \right|  \leq C \delta^{\tfrac{1}{6}}
\end{align}
and, using \eqref{eq:A22 bound} and \eqref{eq:A1 bound 1}, 
\begin{align}
\label{eq:A bound final}
\left| \Ee \left[ \int_0^\tau A_t dt \right] \right|  \leq C \delta^{\tfrac{1}{6}}.
\end{align}
{\bf\textit{Step 3: Bound on \eqref{eq:Aux delta strong 1}.}} For the term $B$ in \eqref{eq:Strong first term main}, we use the Lipschitz continuity of $G$ and $\hat{g}$  in Assumption \ref{hyp:Cost Poisson OS} together with Proposition \ref{pro:Strong rate sup} to get
\begin{align}
 \left| \Ee [B] \right| &\leq C \Ee \left[| X^\delta_\tau - X^0_\tau | + | \mathbb E\left[\langle \hat{g}, \mu^{\delta, \star(\text{t},\text{x})} \!- \mu^{\star} \rangle|\mathcal F^c_\tau\right] | \right] \notag\\ &\leq\! C ( \sqrt{\delta} + \Ee \left[ |\hat{g}(\tau^{\star}, X^\delta_{\tau^{\star}}) - \hat{g}(\tau^{\star}, X^0_{\tau^{\star}}) | \right] ) \!\leq\! C \sqrt{\delta}. \label{eq:terminalcost B1 B2 bound}
\end{align}
From \eqref{eq:Strong first term main}, \eqref{eq:tildeA bound final}, \eqref{eq:A bound final}, and \eqref{eq:terminalcost B1 B2 bound}, we obtain 
\begin{align}
\label{eq:Strong First term bound}
 \Delta_\delta(\tau) \leq  C \delta^{\tfrac{1}{6}},
\end{align}
where $C>0$ is independent of $\delta$. Since  $\tau \in \mathcal{T}(\mathbb{F}^{\bm{W}})$ is arbitrary, we also get
\begin{align}
\label{eq:Strong Third term bound}
 \Delta_\delta(\tau^\star) \leq C \delta^{\tfrac{1}{6}},
\end{align}
Therefore, \eqref{eq:Aux delta strong 1}, \eqref{eq:Strong Second term bound}, \eqref{eq:Strong First term bound}, and \eqref{eq:Strong Third term bound} show that
\begin{align*}
& J^\delta(\tau, m^{\delta,\star}, \mu^{\delta,\star}) - J^{\delta}(\tau^\star, m^{\delta,\star}, \mu^{\delta,\star})\leq C \delta^{\tfrac{1}{6}}.
\end{align*}
Because $\tau \in \mathcal{T}(\mathbb{F}^{\bm{W}})$ is arbitrary, we conclude that $(m^{\delta , \star},\mu^{\delta , \star})$ is a two--scale $\varepsilon$--MFG equilibrium for $\varepsilon = C \delta^{1/6}$, where $C$ is independent of $\delta$.
\vskip0.3cm
\paragraph{Assume Explicit OU regime (Assumptions \ref{hyp:Dynamics_Discret_X_Y} and \ref{hyp:Cost discretization OS}) holds.} Similarly to the first part we write 
\begin{align}
\label{eq:costs2_simplified-00}
\Delta_\delta(\tau)= \left| \Ee \left[ \int_0^\tau  A_t  dt + B \right]\right|, 
\end{align}
where the term $B$ has exactly the same form as in the first part, and is estimated in the same way, using the second part of Proposition \ref{pro:Strong rate sup} in place of the first part. It remains therefore to estimate the term $A$: 
\begin{align*}
A_t & = F\!\left(t,X^\delta_t, Y^\delta_t, \Ee \left[ \hat{f}(X^\delta_t,Y^\delta_t) \mathbf 1_{(t,T]}(\tau^\star)\bigg| \mathcal{F}^{\text{c}}_t  \right] \right)
\\
& \quad - \int_{\Rr^2} F\!\left(t,X^0_t,y, \Ee \left[ \hat{F}(X^0_{t},y') \mathbf 1_{(t,T]}(\tau^\star) \right]\right) \bar{\mu}'(dy, dy').
\end{align*}
Writing
\begin{align*}
A^{1}_t &=  F(t,X^\delta_t, Y^\delta_t,\Ee[ \hat{f}(X^\delta_{t},Y^\delta_t) \mathbf 1_{(t,T]}(\tau^\star) | \mathcal{F}^{\text{c}}_t ] ) -F(t,X^0_t, Y^\delta_t, \Ee[ \hat{f}(X^\delta_{t},Y^\delta_t) \mathbf 1_{(t,T]}(\tau^\star) | \mathcal{F}^{\text{c}}_t ] ) , \notag
\\
A^{2}_t & = F(t,X^0_t, Y^\delta_t, \Ee[ \hat{f}(X^\delta_{t},Y^\delta_t) \mathbf 1_{(t,T]}(\tau^\star) | \mathcal{F}^{\text{c}}_t ] ) - F(t,X^0_t, Y^\delta_t,\Ee[ \hat{f}(X^0_{t},Y^\delta_t) \mathbf 1_{(t,T]}(\tau^\star) | \mathcal{F}^{\text{c}}_t ] ) , \notag
  \\
A^{3}_t &=  F(t,X^0_t, Y^\delta_t,  \Ee [ \hat{F}_t(X^0_t,\hat Y^\delta_t) \mathbf 1_{(t,T]}(\tau^\star) | \hat Y^\delta_t] ) -  F(t,X^0_t, Y^\delta_t,  \Ee [ \hat{F}(X^0_t,\hat Y^\delta_t) \mathbf 1_{(t,T]}(\tau^\star)| \hat Y^\delta_t ] ),\\
A^{4}_t &=  F(t,X^0_t, Y^\delta_t,  \Ee [ \hat{F}(X^0_t,\hat Y^\delta_t) \mathbf 1_{(t,T]}(\tau^\star)| \hat Y^\delta_t ] )  - \int_{\Rr^2} F(t, X^0_{t} , y, \Ee [ \hat{F}(X^0_t,y')  \mathbf 1_{(t,T]}(\tau^\star) ] ) \bar{\mu}'(dy,dy'),
\end{align*}
we will provide estimates for each term in the subsequent steps. 

{\bf \textit{Step 1: Immediate bounds}} Using the Lipschitz continuity and boundedness from Assumption \ref{hyp:Cost discretization OS}, together with Proposition \ref{pro:Strong rate sup}, we obtain, 
\begin{align}
\label{eq:costs2_simplified-0002}
& \Ee \left[ \left| \mathbf 1_{(t,T]}(\tau) A^{1}\right|\right], \Ee \left[ \left| \mathbf 1_{(t,T]}(\tau) A^{2}\right|\right]\leq C \delta^{\tfrac{1}{3}}.
\end{align}
To estimate the term $A^{3}_t$, recall that $\tilde Y_t = Y^\delta_t - \hat Y^\delta_t$ is independent from $\mathcal F^c_t$ and follows the centered Gaussian law with variance $v^1_t:=\frac{\sigma_1^2}{2\kappa}(1-e^{-\frac{2\kappa t}{\delta}})$. Thus, by the Lipschitz property of $F$ and $\hat f$, 
\begin{align}
\mathbf E[|\mathbf 1_{(t,T]}(\tau)A^{3}_t|]&\leq C\mathbb E[|\hat F_t(X^0_t,\hat Y^\delta_t) -\hat F(X^0_t,\hat Y^\delta_t) |]\notag\\ &\leq C\mathbb E\left[\int_{\mathbb R} e^{-\frac{x^2}{2}} |\hat f(X^0_t,x\sqrt{v^1_t} + \hat Y^\delta_t) -\hat f(X^0_t,x\sqrt{v^1_\infty} + \hat Y^\delta_t)  |\right]\notag\\&\leq C |\sqrt{v^1_t} - \sqrt{v^1_\infty}| \leq C \left(1-\sqrt{1-e^{-\frac{2\kappa t}{\delta}}}\right)\label{varest1}
\end{align}
Thus,
\begin{align}
\left| \Ee \left[ \int_0^\tau A^{3}_t dt \right] \right| \leq C \int_0^T \left(1-\sqrt{1-e^{-\frac{2\kappa t}{\delta}}}\right) dt \leq  C \delta \int_0^\infty \left(1-\sqrt{1-e^{-{2\kappa t}}}\right) dt = C\delta. \label{varest2}
\end{align}

{\bf \textit{Step 2: Bounds using the discretization procedure}} Next, we derive the estimate for $A^{4}_t$. Divide $[0,T]$ into intervals of size $\Delta$ to be chosen later. Using the notation introduced in the first part of the proof and the boundednes of $F$, we write:
\begin{align}
\label{eq:costs2_simplified-0021}
  \left| \Ee \left[ \int_0^\tau A^{4}_t dt \right] \right|  &\leq \Delta \mathbb E\left[\sup_{0\leq t\leq T} |A^{4}_t|\right] + \left| \Ee \left[ \int_0^{\theta^\Delta_\tau} A^{4}_t dt \right] \right|\notag\\ &\leq C\Delta + \sum_{k=0}^{\lfloor T/\Delta \rfloor-1} \left|\Ee\left[ \mathbf 1_{k\Delta \leq \theta^\Delta_\tau} \int_{k\Delta}^{(k+1)\Delta} A^{4}_t dt \right] \right|
\end{align}
From Assumption \ref{hyp:Dynamics_Discret_X_Y}, for $t \in [k \Delta, (k+1) \Delta)$, we easily deduce that $\mathbb E[|X^0_t - X^0_{k\Delta}|] \leq C\Delta^{\frac{1}{2}}$. Let us denote
$$
\tilde F(t,k\Delta,y,y')  = F(t,X^0_{k\Delta}, y,  \Ee [ \hat{F}(X^0_{t},y') \mathbf 1_{(t,T]}(\tau^\star) ] )
$$
and
$$
\xi(y,t) = \theta + (y -\theta)e^{-\frac{\kappa t}{\delta}}.
$$
It is easy to check that $\tilde F$ is Lipschitz in $y$ and $y'$, leading to the following estimate:
\begin{align*}
& \left| \Ee \left[ \tilde F(t,k\Delta, Y^\delta_t, \hat{Y}^\delta_t) - \int_{\Rr^2} \tilde F(t,k\Delta, y, y') \bar{\mu}'(dy,dy')|\mathcal F^{\mathbf W}_{k\Delta} \right]\right|\\
  & = \frac{1}{2\pi}\left|\int_{\mathbb R^2} e^{-\frac{x_1^2}{2} - \frac{x_2^2}{2}} \left( \tilde F(t,k\Delta, \xi(Y^\delta_{\Delta k},t-\Delta k)+ x_1\sqrt{v^1_{t-\Delta k}} + x_2 \sqrt{v^2_{t-\Delta k}}, \xi(Y^\delta_{\Delta k},t-\Delta k)+ x_2 \sqrt{v^2_{t-\Delta k}})\right.\right.\\
  & \qquad \qquad- \left. \left. \tilde F(t,k\Delta, \theta+ x_1\sqrt{v^1_{\infty}} + x_2 \sqrt{v^2_{\infty}}, \theta+ x_2 \sqrt{v^2_{\infty}})\right)dx_1\, dx_2 \right|\\
  &\leq C \int_{\mathbb R^2}   e^{-\frac{x_1^2}{2} - \frac{x_2^2}{2}} \left(|\xi(Y^\delta_{\Delta k},t-\Delta k) - \theta| + |x_1|(\sqrt{v^1_\infty}-\sqrt{v^1_{t-\Delta k}} ) + |x_2|(\sqrt{v^2_\infty}-\sqrt{v^2_{t-\Delta k}})  \right)\\
  &\leq  C \left(|Y^\delta_{\Delta k}-\theta| e^{-\frac{\kappa(t-\Delta k)}{\delta}}  + \sqrt{v^1_\infty}-\sqrt{v^1_{t-\Delta k}}  + \sqrt{v^2_\infty}-\sqrt{v^2_{t-\Delta k}} \right)
\end{align*}
Plugging this estimate into the integral term, we find
\begin{align}
  &\left|\Ee\left[ \mathbf 1_{k\Delta \leq \theta^\Delta_\tau} \int_{k\Delta}^{(k+1)\Delta} A^{4}_t dt \right] \right|  \notag\\
  & \leq C\Delta^{\frac{3}{2}} +  \left|\Ee \left[\mathbf 1_{k\Delta \leq \theta^\Delta_\tau} \int_{k\Delta}^{(k+1)\Delta} \tilde F(t,k\Delta, Y^\delta_t, \hat{Y}^\delta_t) - \int_{\Rr^2} \tilde F(t,k\Delta, y, y') \bar{\mu}'(dy,dy') dt \right]\right|,\notag\\
  & \leq C\Delta^{\frac{3}{2}} +  C\int_{0}^{\Delta} \left( e^{-\frac{\kappa t}{\delta}}\Ee|Y^\delta_{\Delta k}-\theta|+e^{-\frac{\kappa t}{\delta}}\Ee|\hat Y^\delta_{\Delta k}-\theta|  + \sqrt{v^1_\infty}-\sqrt{v^1_{t}}  + \sqrt{v^2_\infty}-\sqrt{v^2_{t}} \right) dt \notag\\
  & \leq C \Delta^{\frac{3}{2}} + C \delta (\Ee|Y^\delta_{\Delta k}-\theta|+\Ee|\hat Y^\delta_{\Delta k}-\theta| + 1),\label{varest_final}
\end{align}
where the integrals involving $v^1_t$ and $v^2_t$ are estimated as in (\ref{varest1}--\ref{varest2}).

By Assumption \ref{hyp:Dynamics_Discret_X_Y}--\ref{hyp:Dynamics_Discret_X_Y IC}, the initial distributions have finite second moments, which yields bounds on those of $Y^\delta_t$, and $\hat{Y}^\delta_t$ independent of $\delta$. By \eqref{eq:costs2_simplified-0021} and \eqref{varest_final} we therefore obtain
\begin{align}
\label{eq:costs2_simplified-00237}
& \left|\Ee\left[\int_0^\tau  A^{4}_t dt\right]\right| \leq C \Delta^{\frac{1}{2}} + C\frac{\delta}{\Delta}. 
\end{align}
The optimal rate $\delta^{\frac{1}{3}}$ is obtained by taking $\Delta \sim \delta^{\frac{2}{3}}$. 

{\bf \textit{Step 3: Bound on \eqref{eq:Aux delta strong 1}}}
By \eqref{eq:costs2_simplified-00}, \eqref{eq:costs2_simplified-0002},\eqref{varest2} and \eqref{eq:costs2_simplified-00237}, we obtain 
\begin{align}
\label{eq:costs2_simplified-002378}
& \Delta_\delta(\tau) \leq C \delta^{1/3}.
\end{align}
and similarly $\Delta_\delta(\tau^\star) \leq C \delta^{1/3}$. Together with \eqref{eq:Aux delta strong 1}, \eqref{eq:Strong Second term bound}, this shows that
\begin{align*}
& J^\delta(\tau, m^{\delta,\star}, \mu^{\delta,\star}) - J^{\delta}(\tau^\star, m^{\delta,\star}, \mu^{\delta,\star})\leq C \delta^{\tfrac{1}{3}}.
\end{align*}
Because $\tau \in \mathcal{T}(\mathbb{F}^{\bm{W}})$ is arbitrary, we conclude that $(m^{\delta , \star},\mu^{\delta , \star})$ is a two--scale $\varepsilon$--MFG equilibrium for $\varepsilon = C \delta^{1/3}$, where $C$ is independent of $\delta$.
\end{proof}
\begin{remark}
\label{rem:Poisson orders}
In Assumption \ref{hyp:Cost Poisson OS}, the dependence of $\hat{f}_2$ on the fast-scale variable $y$ reduces the approximation order in Theorem \ref{thm:Existence 2S epsilon MFG eq} to $1/6$. Under Poisson equation regiume, {\bf \textit{Step 1}} of the proof shows, through \eqref{eq:tildeA bound final}, that the cost contribution $F^1 F^2$ is bounded with order $1/2$. However, in {\bf \textit{Step 2}}, the discretization procedure requires a step-size $\Delta$, whose sharpest bound lowers the order to $1/6$. Consequently, if $\hat{f}_2$ in Assumption \ref{hyp:Cost Poisson OS} is independent of the fast scale, Theorem \ref{thm:Existence 2S epsilon MFG eq} recovers the higher order $1/2$.

In contrast, under Assumption \ref{hyp:Cost discretization OS}, the dependence of $\hat{f}$ on the fast-scale variable $y$ necessitates the discretization procedure. Under explicit OU regime, {\bf \textit{Step 1}} shows, via \eqref{eq:costs2_simplified-0002}, that the cost contributions without fast-scale dependence on the measure are bounded with order $1/3$, a rate inherited from Proposition \ref{pro:Strong rate sup}. In {\bf \textit{Step 2}}, involving the discretization procedure, the step can be chosen appropriately to recover the same order of strong convergence.
\end{remark}

\begin{remark}
\label{rem:Common noise and orders}    
In Poisson equation regime, the discretization in {\bf \textit{Step 2}} arises from the presence of common noise. We use Lemma \ref{lem:supL2_H}, based on the Poisson equation method, to estimate expectations of integrals. The presence of common noise, however, introduces a conditional expectation into these integrals, as reflected in the term. In the absence of common noise, these terms can instead be estimated directly by differentiating, which yields the sharper convergence rate of order $1/2$.
\end{remark}
\begin{remark}
\label{rem:Linear dynamics fast}
The Ornstein–Uhlenbeck dynamics for the fast process in Assumptions \ref{hyp:Dynamics_Discret_X_Y}--\ref{hyp:Dynamics_Discret_X_Y A1} satisfy the more general condition
\begin{align*}
&|B(y_1) - B(y_2)| + |\Sigma(y_1) - \Sigma(y_2)| + |\hat{\Sigma}(y_1) - \hat{\Sigma}(y_2)| \leq C |y_1 - y_2|, 
\\
&2\langle B(y_1) - B(y_2), y_1 - y_2 \rangle + 3|\Sigma(y_1) - \Sigma(y_2)|^2 + 3|\hat{\Sigma}(y_1) - \hat{\Sigma}(y_2)|^2 \leq -\beta |y_1 - y_2|^2,
\end{align*}
for some positive constants $C$ and $\beta$. Under this condition and Assumption \ref{hyp:Dynamics_Discret_X_Y}--\ref{hyp:Dynamics_Discret_X_Y A2}, \cite{rockner2021strong} shows—via the discretization method—that strong convergence holds for McKean–Vlasov dynamics with order $1/3$. In our framework with common noise, we establish the existence of approximate equilibria by restricting the fast variable to linear dynamics.
\end{remark}

\subsection{Two-scale $\varepsilon$-MFG equilibrium with \textit{randomized stopping}}

Since MFG equilibria with strict stopping do not necessarily exist \cite{bertucci2018optimal}, to ensure existence of equilibria, the class of stopping strategies must be enlarged to include randomized stopping times. In this section, we construct the $\varepsilon$-MFG equilibrium with randomized stopping for the two-scale problem assuming existence of an MFG equilibrium with randomized stopping for the effective problem; such existence results will be provided in Section \ref{Sec:OS existence effective}. 

We begin by introducing the concepts of \textit{two-scale MFG equilibrium with randomized stopping} and their \textit{$\varepsilon$}-approximate counterparts. As a preliminary step, we define the admissible strategy set in this framework, which is characterized by a specific class of randomization kernels. More precisely, let
\begin{align*}
\mathcal{K}^{\bm{W}} := \left\{
\begin{aligned}
& k:\Omega \times \mathcal{B}([0,T])  \to [0,1]: \, k \text{ is a regular probability kernel and }
\\
&\text{for all }0\leq t \leq T \text{ and } B\in \mathcal{B}([0,t]),\, k(\cdot,B) \text{ is } \mathcal{F}^{\bm{W}}_t\text{--measurable}
\end{aligned}
\right\}.
\end{align*}
For a given $k \in \mathcal{K}^{\bm{W}}$, define the associated cost functional:
\begin{align*}
& J^\delta(k,m,\mu):=\mathbb{E}\left[\int_0^T f(t,X_t^\delta,Y_t^\delta,m_t)k(\cdot,(t,T])dt+ \int_0^T g(t, X_t^\delta,\mu^c_t)k(\cdot,dt)\right],
\end{align*}

We now define a \textit{two-scale $\varepsilon$-MFG and MFG equilibrium with randomized stopping}.
\begin{definition}[Two-scale $\varepsilon$-MFG equilibrium with \textit{randomized} stopping]
\label{def:2S e--MFG eq with randomized stop}
Let $\epsilon\geq 0$ and $\delta > 0$. We say that $(k^{\varepsilon,\delta},m^{\varepsilon, \delta}, \mu^{\varepsilon, \delta})$ is a \textit{two-scale $\varepsilon$--MFG equilibrium with randomized stopping} if:
\begin{itemize}
\item[(i)] For all $k' \in \mathcal{K}^{\bm{W}}$,
\begin{align*}
& J^\delta(k^{\varepsilon,\delta}, m^{\varepsilon,\delta}, \mu^{\varepsilon,\delta}) \geq J^\delta(k', m^{\varepsilon,\delta}, \mu^{\varepsilon,\delta}) - \varepsilon,
\end{align*}
\item[(ii)] $(m^{\varepsilon,\delta}_t)_{0 \leq t \leq T}\in \mathbb V(\mathbb R\times \mathbb R)$  and $\mu^{\varepsilon,\delta} \in \mathbb M(\mathbb R\times \mathbb R)$ such that
\begin{align*}
& m^{\varepsilon, \delta}_t(B) = \Ee\left[ \textbf{1}_{B}(X^{\delta}_t, Y^\delta_t)k^{\varepsilon,\delta}(\cdot,(t,T])\bigg| \mathcal{F}^{\text{c}}_t\right] \text{a.s.} \quad B \in \mathcal{B}(\mathbb R \times \Rr),\,\, \quad t \in [0,T], \notag 
\\
& \mu^{\varepsilon, \delta} (B) = \Ee\left[ \int_0^T \textbf{1}_B(\theta,X^\delta_\theta,Y^\delta_\theta) k^{\varepsilon, \delta}(\cdot,d\theta) \bigg| \mathcal{F}^{\text{c}}_T \right]\,\text{a.s.}, \quad B \in \mathcal{B}([0,T]\times \mathbb R \times \Rr).
\end{align*}
\end{itemize}
When $\varepsilon=0$, we drop the superscript and say that $(k^\delta,m^\delta,\mu^\delta)$ is a \textit{two-scale MFG equilibrium with randomized stopping}. 
\end{definition}

\textcolor{black}{To study the existence of the approximate MFG equilibrium for the two-scale problem (cf. Definition \ref{def:2S e--MFG eq with randomized stop}), we introduce the concept  \textit{of an effective MFG equilibrium with randomized stopping}
We first define the following class of randomized stopping times for the \textit{effective} MFG problem. Let
\begin{align*} 
\mathcal{K}^{\text{x}} := \left\{
\begin{aligned}
& k:\Omega \times \mathcal{B}([0,T]) \to [0,1]: \, k \text{ is a regular probability kernel and }
\\
&\text{for all }0\leq t \leq T \text{ and } B\in \mathcal{B}([0,t]),\, k(\cdot,B) \text{ is } \mathcal{F}^{\text{x}}_t\text{--measurable}
\end{aligned}
\right\},
\end{align*}
and for $k\in \mathcal{K}^{\text{x}}$ define
\begin{align*}
& J^0(k,m,\mu):=\mathbb{E}\left[\int_0^T \bar{f}(t,X_t^0,m_t)k(\cdot,(t,T])dt+ \int_0^T g(t, X_t^0,\mu)k(\cdot,dt)\right],
\end{align*}
with $\bar{f}$ given by either \eqref{payoff1} or \eqref{payoff2}.
\begin{definition}[Effective MFG equilibrium with \textit{randomized} stopping]\label{def:exist_effective}
We say that $(k,m,\mu)$ is an \textit{effective MFG equilibrium with randomized stopping} if:
\begin{itemize}
\item[(i)] For all $k' \in \mathcal{K}^{\text{x}}$,
\begin{align*}
& J^0(k,m,\mu) \geq J^0(k',m,\mu),
\end{align*}
\item[(ii)] $(m_t)_{0 \leq t \leq T}\in \mathcal V(\mathbb R)$  and $\mu\in \mathcal P([0,T]\times \mathbb R)$ such that 
\begin{align*}
& m_t(B) = \Ee\left[ \textbf{1}_{B}(X^0_t)k(\cdot,(t,T])\right], \quad B \in \mathcal{B}(\mathbb R),\,\, \quad t \in [0,T], \notag \\
& \mu(B) = \Ee\left[ \int_0^T \textbf{1}_B(\theta,X^0_\theta) k(\cdot,d\theta) \right], \quad B \in \mathcal{B}([0,T]\times \mathbb R).
\end{align*}
\end{itemize}
\end{definition}}

We now present the main result of this Section, which consists of proving the existence of a \textit{two--scale $\varepsilon$-MFG equilibrium with randomized stopping} for the two-scale optimal stopping MFG with common noise.

\begin{theorem}[Construction  of a strong two-scale $\varepsilon$-MFG equilibrium with \textit{randomized stopping} from an effective MFG equilibrium]
\label{thm:Existence 2S epsilon MFG eq randomized}
Suppose either Poisson equation regime (Assumptions \ref{hyp:Dynamics_Poisson_X_Y} and \ref{hyp:Cost Poisson OS}), or Explicit OU regime (Assumptions \ref{hyp:Dynamics_Discret_X_Y}, \ref{hyp:Cost discretization OS}). 
Let $(k^{\star},m^\star,\mu^\star)$ be an effective MFG equilibrium with randomized stopping. 
For $0<\delta<1/2$, define $(k^{\star},m^{\delta,\star},\mu^{\delta,\star})$ as follows:
\begin{align*}
& m^{\delta , \star}_t(B) := \Ee \left[ \textbf{1}_{B}(X^\delta_t,Y^\delta_t)\, k^{\star}(\cdot,(t,T]) \bigg| \mathcal{F}^{\text{c}}_t \right], 
\quad B \in \mathcal{B}(\mathbb R \times \Rr), \quad t \in [0,T], \notag
\\
& \mu^{\delta , \star}(B) := \Ee\left[ \int_0^T \textbf{1}_{B}(\theta, X^\delta_\theta, Y^\delta_\theta)\, k^{\star}(\cdot,d\theta) \bigg| \mathcal{F}^{\text{c}}_T \right], 
\quad B \in \mathcal{B}([0,T]\times \mathbb R \times \Rr),
\end{align*}
where $(X^\delta, Y^\delta)$ is the strong solution of \eqref{eq:dynamics 2S}.

Then $(k^{\star}, m^{\delta,\star},\mu^{\delta,\star})$ is a two-scale $\varepsilon$-MFG equilibrium with randomized stopping, with
\[
\varepsilon = 
\begin{cases}
C\,\delta^{1/6}, & \text{Poisson equation regime}, \\[4pt]
C\,\delta^{1/3}, & \text{Explicit OU regime},
\end{cases}
\]
for some $C>0$ independent of $\delta$.
\end{theorem}

\begin{proof} 
 Let $(k^\star, m^\star, \mu^\star)$ be an effective MFG equilibrium with randomized stopping, which implies that:
\begin{align*}
m^\star_t(B) & = \Ee\left[ \textbf{1}_{B}(X^0_t) k^{\star}(\cdot,(t,T])\right], \quad B\in \mathcal{B}(\mathbb R), \, t\in[0,T], \notag
\\
\mu^{\star}(B) & = \Ee\left[ \int_0^T \textbf{1}_{B}(\theta, X^0_\theta) k^{\star}(\cdot,d\theta)\right], \quad B\in \mathcal{B}([0,T]\times \mathbb R).
\end{align*}
It is clear that  $k^{\star}\in \mathcal{K}^{\bm{W}}$.
Let $k \in \mathcal{K}^{\bm{W}}$. We need to estimate the following difference:
\begin{align*}
& J^\delta(k,m^{\delta,\star}, \mu^{\delta,\star})-J^\delta(k^\star, m^{\delta,\star}, \mu^{\delta,\star}).
\end{align*}
To do so, we rewrite it as follows:
\begin{align}
\label{ineq}
J^\delta(k,m^{\delta,\star}, \mu^{\delta,\star})-J^\delta(k^\star, m^{\delta,\star}, \mu^{\delta,\star})&=J^\delta(k,m^{\delta,\star}, \mu^{\delta,\star})-J^{0}(k,m^{\star}, \mu^{\star})\nonumber \\
&+ J^{0}(k,m^{\star}, \mu^{\star})-J^{0}(k^\star,m^\star,\mu^\star)\nonumber\\
&+J^{0}(k^\star,m^\star,\mu^\star)-J^\delta(k^\star,m^{\delta,\star},\mu^{\delta,\star}).
\end{align}
\textcolor{black}{Fix $\omega \in \Omega$. For each $t\in[0,T]$, define
\[
F_t(\omega):=k(\omega,[0,t]).
\]
Since $k$ is a regular probability kernel, for every $\omega\in\Omega$ the map $t\mapsto F_t(\omega)$ is increasing, right-continuous, and satisfies $F_T(\omega)=1$. Let $\bar{F}^{-1}_\omega:[0,1]\to [0,T]$ denote the right-continuous inverse of $t \mapsto F_\omega(t)$. Let $(\bar{\Omega},\sigma(U),\mathbb{Q})$ be a probability space supporting a uniformly distributed random variable $U$ taking values in $[0,1]$ and independent of $\mathbb{F}^{\bm{W}}$. Let $\bar{\mathcal{F}} = \mathcal{F} \otimes \sigma(U)$, $\bar{\mathbb{F}}= \{\mathcal{F}^{\bm{W}}_t \otimes \sigma(U),\, 0 \leq t \leq T\}$, and $\bar{\mathbb{P}}(d\omega , du) = \mathbb{P}(d\omega) \mathbb{Q}(du)$. On the enlarged space $(\Omega \times \bar{\Omega}, \bar{\mathcal{F}}, \bar{\mathbb{F}}, \bar{\mathbb{P}})$, define
\[
\tau(\omega,u): = \bar{F}^{-1}_\omega (U(u)).
\]
Then, 
\[
\{ \tau \leq t\} = \{(\omega,u) \in \Omega \times \bar{\Omega}:\, U(u) \leq F_\omega(t)\} = \{(\omega,u) \in \Omega \times \bar{\Omega}:\, U(u) \leq k(\omega,[0,t])\} \in \bar{\mathcal{F}}_t,
\]
which implies that $\tau$ is a $\bar{\mathbb{F}}$-stopping time. Observe that we have
\begin{align}\label{equal}
\bar{\mathbb{P}}(\tau \leq t | \mathcal{F}_T) = k(\cdot,[0,t]) \quad \text{a.s.}
\end{align}
}
We begin by showing that the second term satisfies 
\begin{align}
\label{eq:Strong middle term bound kernels}    
J^{0}(k,m^{\star}, \mu^{\star})-J^{0}(k^\star,m^\star,\mu^\star) \leq 0.
\end{align}
\textcolor{black}{In view of \eqref{equality}, we have
\begin{align}\label{tau}
J^{0}(k,m^{\star}, \mu^{\star})=\mathcal{J}^{0}(\tau,m^{\star}, \mu^{\star}),
\end{align}
with $\mathcal{J}^0(\tau,m^{\star}, \mu^{\star}):=\mathbb{E}\left[\int_0^\tau \bar{f}(t,X_t^0,m_t^\star)dt+g(\tau, X_\tau^0,\mu^\star) \right].$
Similarly, one can construct a $\bar{\mathbb{F}}$-stopping time $\tau^\star$ such that
\begin{align}\label{taustar}
J^{0}(k^\star,m^{\star}, \mu^{\star})=\mathcal{J}^{0}(\tau^\star,m^{\star}, \mu^{\star})=\sup_{\tau \in \mathcal{T}(\bar{\mathbb{F}})} \mathcal{J}         ^0(\tau, m^\star,\mu^\star).
\end{align}
Note that, by similar arguments as those used in the proof of Proposition \ref{pro:sup equality}, we get
\begin{align}\label{equalJ}
\sup_{\tau \in \mathcal{T}(\mathbb{F}^{\text{x}})} \mathcal{J}^0(\tau, m^\star,\mu^\star) = \sup_{\tau \in \mathcal{T}(\bar{\mathbb{F}})} \mathcal{J}         ^0(\tau, m^\star,\mu^\star).
\end{align}
Since $(k^\star,m^{\star},\mu^{\star})$ is an effective  Nash equilibrium and by using \eqref{tau}, \eqref{taustar} and \eqref{equalJ}, we derive \eqref{eq:Strong middle term bound kernels}.
For the remaining terms in \eqref{ineq}, we assume Poisson equation regime (Assumptions \ref{hyp:Dynamics_Poisson_X_Y} and \ref{hyp:Cost Poisson OS}) and illustrate the computation with the first term on the right-hand side of \eqref{ineq}, noting that the same arguments apply for the third one.
From \eqref{equal}, we derive that $\Ee [ \mathbf{1}_A(X^\delta_t,Y^\delta_t) k(\cdot,[0,t])| \mathcal{F}^\text{c}_T] = \bar{\Ee} [ \mathbf{1}_A(X^\delta_t,Y^\delta_t) \mathbf{1}_{[0,t]}(\tau)| \mathcal{F}^\text{c}_T]$ a.s. for any $A \in \mathcal{B}(\mathbb R\times \Rr)$. By the same arguments, we obtain}
\[
\Ee [ \phi(X^\delta_t,Y^\delta_t) k(\cdot,(t,T])| \mathcal{F}^\text{c}_T] = \bar{\Ee} [ \phi(X^\delta_t,Y^\delta_t) \mathbf 1_{(t,T]}(\tau)| \mathcal{F}^\text{c}_T] \quad \text{a.s}
\]
and $\Ee[\phi(X^\delta_t,Y^\delta_t)k(\cdot,(t,T])]=\bar{\Ee}[\phi(X^\delta_t,Y^\delta_t)\mathbf 1_{(t,T]}(\tau)]$
for any $\phi \in C_b(\mathbb R\times \Rr)$. Using the analogous identities for $k^\star$, define
\begin{align*}
& \bar{m}^{\delta , \star}_t(B) := \bar{\Ee} \left[ \textbf{1}_{B}(X^\delta_t,Y^\delta_t)\, \mathbf 1_{(t,T]}(\tau^{\star}) \bigg| \mathcal{F}^{\text{c}}_t \right], 
\quad B \in \mathcal{B}(\mathbb R \times \Rr), \quad t \in [0,T], \notag
\\
& \bar{\mu}^{\delta , \star}(B) := \bar{\Ee} \left[ \textbf{1}_{B}(\tau^\star, X^\delta_{\tau^\star}, Y^\delta_{\tau^\star})\, \bigg| \mathcal{F}^{\text{c}}_T \right], 
\quad B \in \mathcal{B}([0,T]\times \mathbb R \times \Rr).
\end{align*}
Then, we obtain the a.s. equalities,
\begin{align}
\label{eq:Enlarge2}
& \langle \hat{f}_1, m^{\delta ,\star(\text{x})}_t\rangle = \langle \hat{f}_1, \bar{m}^{\delta ,\star(\text{x})}_t\rangle , \quad \langle \hat{f}_2, m^{\delta ,\star}_t\rangle = \langle \hat{f}_2, \bar{m}^{\delta ,\star}_t\rangle, \quad \langle \hat{g}, \mu^{\delta , \star(\text{t},\text{x})}\rangle = \langle \hat{g}, \bar{\mu}^{\delta , \star(\text{t},\text{x})}\rangle.
\end{align}
and the identities
\begin{align}
\label{eq:Enlarge3}
& \langle \hat{f}_1, m^{\star}_t\rangle \!=\! \bar{\Ee}[\hat{f}_1(X^0_t)\mathbf 1_{(t,T]}(\tau^\star)] , \quad \langle \bar{\hat{f}}_2, m^{\star}_t\rangle \!=\! \bar{\Ee}[\bar{\hat{f}}_2(X^0_t)\mathbf 1_{(t,T]}(\tau^\star)], \quad \langle \hat{g}, \mu^{\star}\rangle \!=\! \bar{\Ee}[\hat{g}(\tau^\star,X^0_{\tau^\star})].
\end{align}
By \eqref{eq:Enlarge2} and \eqref{eq:Enlarge3}, we can write
\begin{align}
\label{eq:Enlarge1}
& J^\delta(k,m^{\delta,\star}, \mu^{\delta,\star})-J^0(k, m^{\star}, \mu^{\star}) \notag
\\
& = \Ee\left[ \!\int_0^T \!\!\! (f(t,X_t^\delta,Y_t^\delta,m^{\delta,\star}_t) - \bar{f}(t,X_t^0,m^{\star}_t))k (\cdot,(t,T]) dt + \!\!\int_0^T \!\!\!( g(t,X_t^\delta,\mathbb E[\mu^{\delta,\star}|\mathcal F^c_t]) - g(t,X_t^0,\mu^{\star}))k(\cdot,dt)) \right] \notag
\\
& = \bar{\Ee} \left[ \int_0^T \!\!\!\left(\! F^1(t,X^\delta_t,Y^\delta_t) F^2(t,X^\delta_t, \langle \hat{f}_1, \bar{m}^{\delta ,\star(\text{x})}_t\rangle) \!-\! \bar{F}^1(t,X^0_t) F^2(t,X^0_t, \bar{\Ee}[\hat{f}_1(X^0_t)\mathbf 1_{(t,T]}(\tau^\star)]) \!\right)\!\mathbf 1_{(t,T]}(\tau)dt \right] \notag
\\
& \quad + \bar{\Ee} \left[ \int_0^T \!\!\!\left(\! F^3(t,X^\delta_t) \langle \hat{f}_2, \bar{m}^{\delta ,\star}_t\rangle - F^3(t,X^0_t) \bar{\Ee}[\bar{\hat{f}}_2(X^0_t)\mathbf 1_{(t,T]}(\tau^\star)] \!\right)\!\mathbf 1_{(t,T]}(\tau)dt \right] \notag
\\
& \quad + \bar{\Ee}\left[ G( \tau, X^\delta_\tau, \mathbb E[\langle \hat{g}, \bar{\mu}^{\delta , \star(\text{t},\text{x})}\rangle|\mathcal F^c_\tau]) - G( \tau, X^0_\tau , \bar{\Ee}[\hat{g}(\tau^\star,X^0_{\tau^\star})]) \right]
\end{align}
and proceed with the estimate as in the proof of Theorem \ref{thm:Existence 2S epsilon MFG eq} (see \eqref{eq:Strong first term main}), obtaining
\begin{align}
\label{eq:Strong First term bound kernels}
& \left| J^\delta(k, m^{\delta, \star}, \mu^{\delta, \star}) - J^0(k, m^{\star}, \mu^{\star})\right|  \leq C \delta^{\tfrac{1}{6}},
\end{align}
where $C>0$ is independent of $\delta$. Similarly, for the third term on the right-hand side of \eqref{ineq}, 
\begin{align}
\label{eq:Strong Third term bound kernels}
& \left| J^0(k^\star, m^{\star}, \mu^{\star}) - J^\delta(k^\star, m^{\delta,\star}, \mu^{\delta,\star})\right| \leq C \delta^{\tfrac{1}{6}}.
\end{align}
By \eqref{ineq}, \eqref{eq:Strong middle term bound kernels}, \eqref{eq:Strong First term bound kernels}, and \eqref{eq:Strong Third term bound kernels}, we get
\begin{align*}
& J^\delta(k, m^{\delta,\star}, \mu^{\delta,\star}) - J^{\delta}(k^\star, m^{\delta,\star}, \mu^{\delta,\star})\leq C \delta^{\tfrac{1}{6}}.
\end{align*}
Because $k \in \mathcal{K}^{\bm{W}}$ is arbitrary, we conclude that $(k^\star,m^{\delta , \star},\mu^{\delta , \star})$ is a two--scale $\varepsilon$--MFG equilibrium with randomized stopping for $\varepsilon = C \delta^{1/6}$, where $C$ is independent of $\delta$.

The same enlargement procedure yields the result when {\bf(2)} Assumptions \ref{hyp:Dynamics_Discret_X_Y} and \ref{hyp:Cost discretization OS} hold.
\end{proof}

\section{Existence of an \textit{effective} MFG equilibrium with \textit{randomized stopping}} \label{Sec:OS existence effective}
In this section, we consider the general setting in which the existence of an equilibrium for the \textit{effective mean-field game (MFG) problem} can be established in the class of randomized strategies. We have seen in the preceding section that from such an equilibrium one can construct a $\varepsilon$- \textit{MFG equilibrium with randomized stopping for the two-scale MFG problem}.

Now, we proceed with the proof of the existence of an \textit{effective MFG equilibrium with randomized stopping}. This result is based on the linear programming approach for mean-field games developed in \cite{1Bouveret} and \cite{4LPApproach} and a new result which consists in a probabilistic representation of the admissible occupation measures which appear in the LP formulation in terms of \textit{randomized stopping times}.

Under explicit OU regime (Assumption \ref{hyp:Dynamics_Discret_X_Y}), this requires an additional condition, needed to invoke Theorem 2.24 from \cite{4LPApproach}.

\begin{assumption}
\label{hyp:EJP unbounded equality}
The coefficients $\bar b$ and $\sigma$ are bounded, with $\sigma$ bounded away from zero.
\end{assumption}

\noindent The main result of this section can be formulated as follows.

\begin{theorem}\label{thm:exist_effective}
Suppose that either Poisson equation regime (Assumptions \ref{hyp:Dynamics_Poisson_X_Y} and \ref{hyp:Cost Poisson OS}) holds or Explicit OU regime (Assumptions \ref{hyp:Dynamics_Discret_X_Y} and \ref{hyp:Cost discretization OS}) with additional Assumption  \ref{hyp:EJP unbounded equality} holds. Then there exists an \textit{effective} MFG equilibrium with \textit{randomized stopping}.
\end{theorem}
The proof of Theorem \ref{thm:exist_effective} is based on the results which will be presented below, and therefore it is postponed at the end of the Section.

\paragraph{Linear programming MFG equilibrium for the effective optimal stopping MFG problem.} For a given filtered probability space $(\Omega, \mathcal{F}, \mathbb F, \mathbb P)$, $\tau$ a $\mathbb F$-stopping time such that $\tau\leq T$ $\mathbb P$-a.s., $W$ a $\mathbb F$-Brownian motion and $(X_t)$ a $\mathbb F$-adapted process such that \eqref{eq:dynamics effective} holds, one can define the measures 
$$\mu^{\tau}:=\mathbb P \circ \left(\tau, X_{\tau}\right)^{-1},$$
and
$$m^{\tau}_t(B):= \mathbb E^{\mathbb P}\left[ \textbf{1}_B(X_t) \mathbf 1_{(t,T]}(\tau)\right],  \quad B\in \mathcal{B}(\mathbb R), \quad t\in [0, T].$$
By applying Itô's formula, one can show that $(m^{\tau}, \mu^\tau)\in \mathcal{R}$,
where $\mathcal{R}$ represents the set of pairs $(m,\mu)$, where $(m_t)_{0\leq t\leq T}$ is a flow of sub-probability measures on $\mathbb R$ and $\mu\in \mathcal P([0,T]\times \mathbb R)$, which satisfy the constraint 
\begin{align}
\label{eq:cons motivation measures}
\int_{[0,T] \times \mathbb R} \varphi(t,x) \mu(dt,dx) = \int_{\mathbb R}\varphi(0,x)m_0^X(dx) + \int_{[0,T]\times \mathbb R}\left((\partial_t + \bar{\mathcal{L}})\varphi\right)(t,x)m_t(dx)dt,
\end{align}
for all test functions $\varphi \in C^{1,2}_b([0,T]\times \mathbb R;\mathbb{R})$.

We now introduce the notion of linear programming MFG equilibrium for the effective MFG problem.
\begin{definition}[Linear programming MFG equilibrium] 
\label{def:LP Nash effective}
We say that $(m^{\star},\mu^{\star}) \in \mathcal{R}$ is \textit{a linear programming MFG equilibrium} if for all $(m,\mu) \in \mathcal{R}$:
\begin{align}
&\int_{0}^T\!\!\!\int_{\mathbb R} \bar{f}(t,x,m^\star_t)m^\star_t(dx)dt
+\int_{[0,T]\times\mathbb R} g(t,x,\mu^\star)\mu^\star(dt,dx) \notag\\
&\quad \geq
\int_{0}^T\!\!\!\int_{\mathbb R} \bar{f}(t,x,m^\star_t)m_t(dx)dt 
+\int_{[0,T]\times\mathbb R} g(t,x,\mu^\star)\mu(dt,dx).
\end{align}
\end{definition}

We recall the following existence result of a linear programming (LP) MFG equilibrium. Either {\bf(1)} Assumptions \ref{hyp:Dynamics_Poisson_X_Y} and \ref{hyp:Cost Poisson OS}, or {\bf(2)} Assumptions \ref{hyp:Dynamics_Discret_X_Y} and \ref{hyp:Cost discretization OS} ensure that Assumption 3.1 in \cite{4LPApproach} holds for the effective problem. Under this assumption, the authors establish the following result (Theorem 3.11).

\begin{theorem}[Existence of a linear programming MFG equilibrium for optimal stopping]\label{existLP} Suppose that either Poisson equation regime (Assumptions \ref{hyp:Dynamics_Poisson_X_Y} and \ref{hyp:Cost Poisson OS}) holds or Explicit OU regime (Assumptions \ref{hyp:Dynamics_Discret_X_Y} and \ref{hyp:Cost discretization OS}) with additional Assumption  \ref{hyp:EJP unbounded equality} holds. Then there exists a LP MFG equilibrium for the effective MFG problem.
\end{theorem}

\paragraph{MFG equilibrium for the effective MFG problem with \textit{randomized stopping}}
\label{sec:canon representation effective}
Introduce the space $\Omega_X:=C([0,T];\mathbb{R})$ with the canonical process $X_t(\omega):=\omega_t$, and consider the enlarged canonical space $\bar{\Omega}_X:=C([0,T];\mathbb{R}) \times [0,T]$. Define the canonical elements
$\bar{X}_t(\bar{\omega}):=\omega_t$ and 
 $\bar{\tau}(\bar{\omega}):=\theta,$
with $\bar{\omega}:=(\omega,\theta)$.
On $(\Omega_X, \mathcal{B}(\Omega_X)$ we define, for $0 \leq t \leq T$, the $\sigma$-algebras $\mathcal{F}_t^{X}:=\sigma\{X_s,\,\ s \leq t\}$ and on $\bar{\Omega}_X$ we introduce $\mathcal{F}^{\bar{X}}_t=\sigma\{\bar{X}_s,\,\, s \leq t\}$, $\mathcal{F}^{\bar{\tau}}_t=\sigma\{\{\bar{\tau} \leq s\},\,\, s \leq t\}$ and $\bar{\mathcal{F}}_t=\mathcal{F}^{\bar{X}}_t \vee \mathcal{F}^{\bar{\tau}}_t$. We introduce the associated canonical filtrations given by $\bar{\mathbb{F}} = (\bar{\mathcal{F}}_t)_{0\leq t \leq T}$, $\mathbb{F}^X = (\mathcal{F}^X_t)_{0\leq t \leq T}$, $\mathbb{F}^{\bar{X}} = (\mathcal{F}^{\bar{X}}_t)_{0\leq t \leq T}$, and $\mathbb{F}^{\bar{\tau}} = (\mathcal{F}^{\bar{\tau}}_t)_{0\leq t \leq T}$. It is clear that $\bar{\tau}$ is a $\bar{\mathbb{F}}$-stopping time.

For a probability measure $\bar{\mathbb P}$ on $ (\bar\Omega_X,\mathcal{B}(\bar\Omega_X))$, we define by $\mathbb F^{\bar X, \bar{\mathbb P}}$ the completion of the filtration $\mathbb F^{\bar X}$ with the null sets of $\bar{\mathbb P}$.

Let $\nu$ be the law of the strong solution $X^0$ of \eqref{eq:dynamics effective} on $(\Omega_X,\mathcal{B}(\Omega_X))$. We similarly define $\mathbb F^{X,\nu}$ as the completion of $\mathbb F^X$ with the null sets of $\nu$. We introduce the following set of probability measures:
\begin{definition}[Law of process with a randomized stopping time]
Let $\mathcal{P}_{\text{R}}$ be the set of probability measures $\bar{\mathbb{P}}$ on $(\bar{\Omega}_X, \mathcal{B}(\bar{\Omega}_X))$ such that
\begin{itemize}
\item[(i)] The marginal of $\bar{\mathbb{P}}$ on $\Omega_X$ is $\nu$ 
\item[(ii)] Under $\bar{\mathbb{P}}$, the filtration $\mathbb{F}^{\bar{X}}$ is immersed in $\bar{\mathbb{F}}$.
\end{itemize}
\end{definition}
 By the disintegration theorem (\cite{carmona2018probabilistic}, Theorem 1.1), we can write $\bbbP$ as
\begin{align}
\label{eq:bP factor e 0}
\bbbP(d\mathrm{x}, d\theta) = k(\rmx, d\theta) \nu(d\mathrm{x}),
\end{align}
where $k: \Omega_X \times \mathcal{B}([0,T])  \to [0,1]$ is a regular kernel, i.e. $k(\rmx, \cdot)$ is a probability on $\mathcal{B}([0,T])$ for all $\mathrm{x}\in\Omega_X$ and $k(\cdot, B)$ is $\mathcal{B}(\Omega_X)$--measurable for all $B \in \mathcal{B}([0,T])$.

\noindent From  Proposition 1.10 in \cite{carmona2018probabilistic}, we deduce the following representation of the set $\mathcal{P}_{\text{R}}$:
\begin{align}\label{rep}
\mathcal{P}_{\text{R}} = \left\{
\bbbP \in \mathcal{P}(\bar{\Omega}_X):\, \bbbP(d\mathrm{x}, d\theta) = {k}(\mathrm{x}, d\theta) \nu(d\mathrm{x}) \text{ for some } {k} \in \mathcal{K}_c^{\mathbf{X}} \right\},
\end{align}
where
\begin{align} \label{def:A01def} 
& \mathcal{K}_c^{\mathbf{X}} := \left\{
\begin{aligned}
& {k}: \Omega_X \times \mathcal{B}([0,T]) \to [0,1] \, \text{ is a regular probability kernel, and }\\
&\text{for all }0\leq t \leq T, \text{ and } B\in \mathcal{B}([0,t]),\, {k}(\cdot,B) \text{ is } \mathcal{F}^{X,\nu}_t \text{--measurable}
\end{aligned}
\right\}.
\end{align}
For $\bbbP \in \mathcal{P}(\bar{\Omega}_X)$, define the occupation measures
\begin{align}
\label{eq:Occupation m A01}
& m_t^{\bbbP}(B) := \Ee^{\bbbP}[\textbf{1}_{B}(\bar{{X}}_t) \mathbf 1_{(t,T]}({\bar{\tau}})], \quad B\in \mathcal{B}(\mathbb R), \, t \in [0,T], \notag
\\
& \mu^{\bbbP}(C) := \Ee^{\bbbP}[\textbf{1}_{C}({\bar{\tau}}, \bar{{X}}_{{\bar{\tau}}}) ], \quad C \in \mathcal{B}([0,T]) \otimes \mathcal{B}(\mathbb R),
\end{align}
and introduce the set 
\begin{align*}
\mathcal{R}_1 := \left\{ (m^{\bbbP},\mu^{\bbbP}):\, \bbbP \in \mathcal{P}_{\text{R}}\right\}.
\end{align*}

We now define the following set of probability measures.
\begin{definition}[Law of process with a strict stopping time] Let $\mathcal{P}_{\text{S}}$ be the set of probability measures $\bar{\mathbb{P}}$ on $(\bar{\Omega}_X, \mathcal{B}(\bar{\Omega}_X))$ such that:
\begin{itemize}
\item[(i)] The marginal of $\bar{\mathbb{P}}$ on $\Omega_X$ is $\nu$;
\item[(ii)] ${\bar{\tau}}$ is a stopping time with respect to $\mathbb{F}^{\bar{X}, \bar{\mathbb{P}}}$.
\end{itemize}
\end{definition}
We introduce the following definition.
\begin{definition} The set $\mathcal{T}$ represents the set of measurable functions $\tau: \Omega_X \mapsto [0,T]$ such that $\tau$ is a $\mathbb{F}^{X,\nu}$-stopping time. 
\end{definition}
We have the following characterization of the set $\mathcal{P}_{\text{S}}$.
\begin{theorem}[Characterization of the set $\mathcal{P}_{\text{S}}$]
We have the following representation:
\begin{align}
\label{def:A00}
& \mathcal{P}_{\text{S}} = \left\{\bdelta_{\tau(\mathrm{x})}(d\theta) \nu(d\mathrm{x})  \text{ for some } \tau \in \mathcal{T}\right\}.
\end{align}
\end{theorem}
\begin{proof}
Let $\bar{\mathbb{P}} \in \mathcal{P}_{\text{S}}$. Since $\bar{\tau}$ is $\mathcal{F}_T^{\bar{X},\bar{\mathbb{P}}}$-measurable, there exists a measurable function $\tau: \Omega_X \to [0,T]$ such that $\bar{\mathbb{P}}(\bar{\tau}=\tau(\bar{X}))=1$. Using also the fact that, by definition of the set $\mathcal{P}_{\text{S}}$, the first marginal of $\bar{\mathbb{P}}$ is $\nu$, we can disintegrate $\bar{\mathbb{P}}$ as $\bar{\mathbb{P}}(dx,d\theta)=\nu(dx)\bdelta_{\tau(x)}(d\theta)$. Now, since $\bar{\tau}$ is a $\mathbb{F}^{\bar{X}, \bar{\mathbb{P}}}$ stopping time, $\bar{\tau}=\tau(\bar{X})$ $\mathbb{P}$-a.s.~and the filtration $\mathbb{F}^{\bar{X}, \bar{\mathbb{P}}}$ is complete, it follows that $\tau(\bar{X})$ is a $\mathbb{F}^{\bar{X}, \bar{\mathbb{P}}}$ stopping time. Furthermore, for any $t \in [0,T]$ and $C \in \mathcal{F}^{\bar{X}, \bar{\mathbb{P}}}_t$, we can easily verify that $(X, \tau(X))^{-1}(C) \in \mathcal{F}_t^{X,\nu}$. By taking $C=\{\tau(\bar{X}) \leq t \}$ we obtain that $\{\tau(X) \leq t\} \in \mathcal{F}_{t}^{X,\nu}$, which implies that $\tau \in \mathcal{T}$.

We now take $\tau \in \mathcal{T}$. We define $\bar{\mathbb{P}}(dx,d\theta):=\bdelta_{\tau(x)}(d\theta)\nu(dx)$. Since $\bar{\mathbb{P}} \circ \bar{X}^{-1}= \nu$, the first condition is satisfied. Define now the set $A:=\{\bar{\tau}=\tau(\bar{X})\}$ which satisfies $\bar{\mathbb{P}}(A)=1$. We have
\begin{align}
\{\tau(\bar{X}) \leq t\}&=\left(\{x \in \Omega_X:\,\, \tau(x) \leq t\} \times [0,T]\right) \nonumber \\
&= (A_1 \times [0,T]) \cup (A_2 \times [0,T]),
\end{align}
with $A_1 \in \mathcal{F}_t^X$ and $A_2 \in \mathcal{N}_{\nu}(\mathcal{F}_T^\nu)$ such that
$\{x \in \Omega_X:\,\, \tau(x) \leq t\}=A_1 \cup A_2$. Therefore, we have $A_1 \times [0,T] \in \mathcal{F}_t^{\bar{X}}$ and $A_2 \times [0,T] \in \mathcal{N}_{\bar{\mathbb{P}}}(\bar{\mathcal{F}}_T)$, such that $\{\tau(\bar{X}) \leq t\} \in \mathcal{F}_t^{\bar{X}, \bar{\mathbb{P}}}$. We conclude that
\begin{align*}
\{\bar{\tau} \leq t\}=\left(\{\tau(\bar{X}) \leq t\} \cap A \right) \cup \left(\{\tau(\bar{X}) \leq t\} \cap A^{\text{c}} \right) \in \mathcal{F}_t^{\bar{X},\bar{\mathbb{P}}}.
\end{align*}

\end{proof}

\begin{proposition}[Relation between the sets $\mathcal{P}_{\text{S}}$ and $\mathcal{P}_{\text{R}}$]
\label{pro:A00inA10}
We have the inclusion $\mathcal{P}_{\text{S}} \subset \mathcal{P}_{\text{R}}$.
\end{proposition}
\begin{proof}
Let $\bar{\mathbb{P}} \in \mathcal{P}_{\text{S}}$. Since $\tau(x)$ is $\mathcal{B}(\Omega_X)$-measurable and takes values in $[0,T]$, it follows that $\delta_{\tau(\cdot) }(B)$ is $\mathcal{B}(\Omega_X)$-measurable for all $B \in \mathcal{B}([0,T])$ and that $\delta_{\tau(x) }(\cdot)$ is a probability measure on $\mathcal{B}([0,T])$ for each $x \in \Omega_X$. Furthermore, as $\tau$ is a $\mathbb{F}^{X,\nu}$-stopping time, it follows that, for all $t \in [0,T]$ and $B \in \mathcal{B}([0,t])$, we have that $\delta_{\tau(x)  }(B)$ is $\mathcal{F}_t^{X,\nu}$-measurable. We conclude that $\delta_{\tau(x)}(d\theta) \in \mathcal{K}_c^{\mathbf{X}}$, which implies that $\bar{\mathbb{P}} \in \mathcal{P}_{\text{R}}$.
\end{proof}
Let 
$$
\mathcal{R}_0 := \{ (m^{\bbbP},\mu^{\bbbP}):\, \bbbP \in \mathcal{P}_{\text{S}}\}.
$$
\begin{theorem}
\label{coro:inclusions e 0}
Suppose that either Poisson equation regime (Assumptions \ref{hyp:Dynamics_Poisson_X_Y} and \ref{hyp:Cost Poisson OS}) holds or Explicit OU regime (Assumptions \ref{hyp:Dynamics_Discret_X_Y} and \ref{hyp:Cost discretization OS}) with additional Assumption  \ref{hyp:EJP unbounded equality} holds.  Then the following inclusion holds true: \begin{align}
\label{eq:inclusions e 0}
\overline{\text{conv}}(\mathcal{R}_0) \subset \mathcal{R}_1 \subset \mathcal{R}.
\end{align}
\end{theorem}
\begin{proof}
The proof is divided into several steps.

\vspace{3mm}

\textbf{\textit{Step 1: Compactness and convexity of $\mathcal{P}_{\text{R}}$.}} We first prove the compactness of the set $\mathcal{P}_{\text{R}}$. Since the marginal on $\Omega_X$ of any element of $\mathcal{P}_{\text{R}}$ is $\nu$, and since $[0,T]$ is compact, it follows that $\mathcal{P}_{\text{R}}$ is tight. To show that $\mathcal{P}_{\text{R}}$ is closed, let $\mathbb{P}_n \to \mathbb{P}$, with $\mathbb{P}_n \in \mathcal{P}_{\text{R}}$. Let $\phi_t$, $\psi_T$ and $\psi_t$ be, respectively, $\mathcal{F}_t^{\tau}$, $\mathcal{F}_T^{X}$ and $\mathcal{F}_t^{X}$ measurable and continuous bounded mappings. We have
\begin{align*}
\mathbb{E}^{\mathbb{P}}\left[\phi_t(\tau)\psi_T(X)\psi_t(X)\right]&=\lim_{n \to \infty} \mathbb{E}^{\mathbb{P}_n}\left[\phi_t(\tau)\psi_T(X)\psi_t(X)\right] \nonumber \\
&=\lim_{n \to \infty} \mathbb{E}^{\mathbb{P}_n}\left[ \phi_t(\tau)\mathbb{E}^{\mathbb{P}_n}[\psi_T(X)|\mathcal{F}_t^X]\psi_t(X)\right] \nonumber \\
& = \mathbb{E}^{\mathbb{P}}\left[\phi_t(\tau)\mathbb{E}^{\mathbb{P}}[\psi_T(X)|\mathcal{F}_t^X]\psi_t(X)\right].
\end{align*}
Since the continuous bounded functions generate $\mathcal{F}_t^{X}$ and $\mathcal{F}_t^{\tau}$,
it follows that $\mathcal{F}_t^{\tau}$ is independent on $\mathcal{F}_T^{X}$ conditional with respect to $\mathcal{F}_t^{X}$. We derive that $\mathbb{P} \in \mathcal{P}_{\text{R}}$.

Let us now show the convexity of the set $\mathcal{P}_{\text{R}}$. Observe that the set $\mathcal{P}_{\text{R}}$ represents the set of probability measures $\mathbb{P}$ with first marginal $\nu$ such that 
\begin{align*}
\mathbb{E}^{\mathbb{P}}\left[\phi_t(\tau)\psi_T(X)\psi_t(X)\right] = \mathbb{E}^{\mathbb{P}}\left[\mathbb{E}^{\mathbb{P}}[\phi_t(\tau)|\mathcal{F}_t^X]\psi_T(X)\psi_t(X)\right].
\end{align*}
We therefore have
\begin{align*}
\int_0^T \int_{\Omega_X} \phi_t(u) \psi_T(x) \psi_t(x) \mathbb{P}(dx,du)=\int_{\Omega_X}  \psi_T(x) \psi_t(x) \left(\int_0^T \phi_t(u)k(x,du)\right) \nu(dx).
\end{align*}
The above constraint is convex with respect to $\mathbb{P}$, and we therefore conclude that $\mathcal{P}_{\text{R}}$ is convex.\\

\textbf{\textit{Step 2: Compactness and convexity of the set $\mathcal{R}_1$.}} In view of the previous step, to show the compactness of the set $\mathcal{R}_1$, we need to prove the continuity of the application $\mathbb{P} \mapsto (m^\mathbb{P}, \mu^{\mathbb{P}})$; \,\  $\mathcal{P}_{\text{R}} \mapsto \mathcal{R}_1$ (where we use the topology of weak convergence for $\mu$ and the topology of weak convergence of associated measures for $m$). Let $(\mathbb{P}^n)_{n \geq 1} \subset \mathcal{P}_{\text{R}}$ be a sequence of probability measures weakly converging to a probability measure $\mathbb{P} \in \mathcal{P}_{\text{R}}$. Let $\varphi \in C_b([0,T] \times \mathbb{R})$. We have: 
\begin{align}\label{equality}
\int_0^T \int_{\mathbb{R}} \varphi(t,x)m_t^{\mathbb{P}^n}(dx)dt=\int_0^T \int_{\Omega_X} \int_0^T \varphi(t,\pi_t(\mathbf{x}))\mathbf 1_{(t,T]}(\theta)\mathbb{P}^n(d\mathbf{x},d\theta)dt=\int_{\Omega_X} \int_0^T \psi(\theta,\mathbf{x})\mathbb{P}^n(d\mathbf{x},d\theta),
\end{align}
where $\pi_t(\mathbf x)$ is the projection operator $\mathbf x\mapsto x_t$ and the last equality follows by Fubini and the function $\psi$ is given by $\psi(\mathbf{x}, \theta):=\int_0^\theta \varphi(t,\pi_t(\mathbf{x}))dt$.
It can be easily observed that, taking into account the assumptions on the function $\varphi$, the map $(\mathbf{x}, \theta) \mapsto \psi(\mathbf{x}, \theta)$ is continuous and bounded. Therefore, by passing to the limit in \eqref{equality}, we obtain:
\begin{align*}
\lim_{n \rightarrow \infty} \int_0^T \int_\mathbb{R} \varphi(t,x) m_t^{\mathbb{P}_n}(dx)dt=\int_0^T \int_\mathbb{R} \varphi(t,x) m_t^{\mathbb{P}}(dx)dt.
\end{align*}
The convexity of the set $\mathcal{R}_1$ follows immediately from the one of the set $\mathcal{P}_{\text{R}}$.\\

\textbf{\textit{Step 3: $\overline{\text{conv}}(\mathcal{R}_0) \subset \mathcal{R}_1$.}} Since $\mathcal{P}_{\text{S}} \subset \mathcal{P}_{\text{R}}$, we deduce that $\mathcal{R}_{0} \subset \mathcal{R}_1$. Furthermore, by convexity and compactness of $\mathcal{R}_1$ proved in the previous step, the result follows.\\

\textbf{\textit{Step 4: $\mathcal{R}_1 \subset \mathcal{R}$.}} Given $\varphi \in C^{1,2}_b([0,T]\times \Rr ; \Rr)$, let $(\bar{M}_t[\varphi])_{0 \leq t \leq T}$ be defined by 
\begin{align}
\label{eq:Martingale barX 0}   
\bar{M}_t[\varphi](\bar{\omega}) := \varphi(t, \bar{\mathrm{X}}_t(\bar{\omega})) - \int_0^t \left((\partial_t + \bar{\mathcal{L}})\varphi\right)(s, \bar{\mathrm{X}}_s(\bar{\omega})) ds.
\end{align}
By definition of the set $\mathcal{P}_{\text{R}}$, it follows that $\bar{M}_\cdot[\varphi]$ is a $(\bar{\mathbb{F}}, \bar{\mathbb{P}})$-continuous bounded martingale for any $\bar{\mathbb{P}} \in \mathcal{P}_{\text{R}}$ and for any $\varphi \in C_{b}^{1,2}([0,T] \times \mathbb{R}; \mathbb{R})$ . Fix $\mathbb{P} \in \mathcal{P}_{\text{R}}$. Thus,
by optional sampling,
\begin{align*}
\Ee^\bbbP \left[\bar{M}_{\bar{\tau}}[\varphi]\right] = \Ee^\bbbP \left[\varphi(0, \bar{\mathrm{X}}_0)\right], 
\end{align*}
yielding
\begin{align*}
\mathbb{E}^{\bar{\mathbb{P}}}\left[\varphi(\bar{\tau}, \bar{\mathrm{X}}_{\bar{\tau}})\right]=\Ee^\bbbP \left[\varphi(0, \bar{\mathrm{X}}_0)+\int_0^{\bar{\tau}}\left((\partial_t + \bar{\mathcal{L}})\varphi\right)(s, \bar{\mathrm{X}}_s) ds\right],
\end{align*}
 for any $\varphi \in C_{b}^{1,2}([0,T] \times \mathbb R; \mathbb{R})$.
By using the above equality and the definition of $(m^{\bar{\mathbb{P}}}, \mu^{\bar{\mathbb{P}}})$ given by $\eqref{eq:Occupation m A01}$, we conclude that  $(m^{\bar{\mathbb{P}}}, \mu^{\bar{\mathbb{P}}}) \in \mathcal{R}$.

\end{proof}
It remains to show that $\mathcal{R} \subset \mathcal{R}_1$. To obtain such inclusion, we prove that $\mathcal{R} \subset \overline{\text{conv}}(\mathcal{R}_0)$ holds. 

\begin{theorem}
\label{pro:inclusion e 0}
Suppose that either Poisson equation regime (Assumptions \ref{hyp:Dynamics_Poisson_X_Y} and \ref{hyp:Cost Poisson OS}) holds or Explicit OU regime (Assumptions \ref{hyp:Dynamics_Discret_X_Y} and \ref{hyp:Cost discretization OS}) with additional Assumption  \ref{hyp:EJP unbounded equality} holds. The following inclusion holds:
\begin{align}
\mathcal{R} \subset \overline{\text{conv}}(\mathcal{R}_0).
\end{align}
\end{theorem}
\begin{proof}
Arguing by contradiction, let $(\tilde{m},\tilde{\mu}) \in \mathcal{R} \setminus \overline{\text{conv}}(\mathcal{R}_0)$. By Theorem \ref{coro:inclusions e 0}, $\overline{\text{conv}}(\mathcal{R}_0)$ is a closed subset of $\mathcal{R}_1$, which is  compact. Thus,  $\overline{\text{conv}}(\mathcal{R}_0)$ is convex and compact. The singleton set $\{(\tilde{m}, \tilde{\mu})\}$ is convex and closed. Therefore, by Theorem 3.4 in \cite{rudin1991functional}, there exist $F,G \in C_b([0,T]\times\Rr)$ and $C_1,C_2 \in \Rr$ such that 
\begin{align}
\label{eq:aux inequality separation 1}
& \int_{[0,T]\times \Rr } F(t,x) m_t(dx)dt + \int_{[0,T]\times \Rr } G(t,x) \mu(dt,dx) \notag
\\
& < C_1 < C_2 < \int_{[0,T]\times \Rr } F(t,x) \tilde{m}_t(dx)dt + \int_{[0,T]\times \Rr } G(t,x) \tilde{\mu}(dt,dx)
\end{align}
for all $(m,\mu) \in \overline{\text{conv}}(\mathcal{R}_0)$.

For these functions, define
\begin{align*}
V^{\mathrm{LP}}(F,G)
&:= \sup_{(m,\mu)\in\mathcal R}
\left\{\int_{[0,T]\times\Rr}F(t,x)m_t(dx)dt+\int_{[0,T]\times\Rr}G(t,x)\mu(dt,dx)\right\},
\\
V^{\mathrm{S}}(F,G)
&:= \sup_{(m,\mu)\in\mathcal R_0}
\left\{\int_{[0,T]\times\Rr}F(t,x)m_t(dx)dt+\int_{[0,T]\times\Rr}G(t,x)\mu(dt,dx)\right\}.
\end{align*}
Assumption \ref{hyp:EJP unbounded equality} and the bounded continuity of $F$ and $G$ allow us to apply Theorem 2.24 in \cite{4LPApproach} to the single-agent optimal stopping problem with singleton control set, running payoff $F$, terminal payoff $G$, drift $\bar b$, and volatility $\sigma$. Hence
\[
V^{\mathrm{LP}}(F,G)=V^{\mathrm{S}}(F,G).
\]
On the other hand, taking the supremum over $(m,\mu) \in \mathcal{R}_0$ in \eqref{eq:aux inequality separation 1} and recalling that $(\tilde{m},\tilde{\mu}) \in \mathcal{R}$, we obtain 
\begin{align*}
& V^{\mathrm{S}}(F,G)  \leq C_1 < C_2  \leq V^{\mathrm{LP}}(F,G), 
\end{align*}
which contradicts the equality of the two values. We conclude that $\mathcal{R} \subset \overline{\text{conv}}(\mathcal{R}_0)$.
\end{proof}

Using the above results, we get the two following representation results of an admissible couple of measures $(m,\mu) \in \mathcal{R}$.

\begin{theorem}[\textit{Representation of occupation measures in terms of randomized stopping}]
\label{thm:representation2}
Suppose that either Poisson equation regime (Assumptions \ref{hyp:Dynamics_Poisson_X_Y} and \ref{hyp:Cost Poisson OS}) holds or Explicit OU regime (Assumptions \ref{hyp:Dynamics_Discret_X_Y} and \ref{hyp:Cost discretization OS}) with additional Assumption  \ref{hyp:EJP unbounded equality} holds. Then, $(m,\mu)\in \mathcal{R}$ if and only if there exists $k \in \mathcal{K}_c^{\mathbf{X}}$ such that:
    \begin{align}\label{repp}
    & m_t(B) = \Ee^{\nu}\left[ \textbf{1}_B(X_t) k(X,(t,T]) \right], \quad B \in \mathcal{B}(\mathbb R),\quad t\in [0,T], \nonumber 
    \\
    & \mu(B) = \Ee^{\nu}\left[\int_{[0,T]}\textbf{1}_B(\theta,X_\theta)k(X,d\theta) \right], \quad B \in \mathcal{B}([0,T]\times \mathbb R).
    \end{align}
\end{theorem}
\begin{proof}
By Theorems \ref{coro:inclusions e 0} and \ref{pro:inclusion e 0}, we get
\begin{align*}
\overline{\text{conv}}(\mathcal{R}_0) = \mathcal{R}_1 = \mathcal{R}.
\end{align*}    
Therefore, given $(m,\mu)\in \mathcal{R}$, we have $(m,\mu) = (m^{\bbbP},\mu^{\bbbP})$ for some $\bbbP \in \mathcal{P}_{\text{R}}$. In view of the representation result of the set $\mathcal{P}_{\text{R}}$ (see \eqref{rep}), we get the existence of a kernel $k \in \mathcal{K}_c^{\mathbf{X}}$ such that $\eqref{repp}$ holds.

Conversely, let $k \in \mathcal{K}_c^{\mathbf{X}}$. We can then construct a probability measure $\bar{\mathbb{P}} \in \mathcal{P}_{\text{R}}$, and in view of the above relation, we get that $(m^{\bbbP},\mu^{\bbbP}) \in \mathcal{R}$.
\end{proof}

We now conclude the proof of Theorem \ref{thm:exist_effective}.
\begin{proof}
By Theorem \ref{existLP}, there exists a linear programming MFG equilibrium $(m^\star,\mu^\star)\in\mathcal R$. By Theorem \ref{thm:representation2}, there exists a kernel $k_c^\star\in \mathcal K_c^{\mathbf X}$ such that
\begin{align}
\label{eq:alt-exist-effective-rep-star}
& m_t^\star(B)=\Ee^\nu\big[\mathbf 1_B(X_t)\,k_c^\star(X,(t,T])\big], \quad B\in\mathcal B(\mathbb R),\ t\in[0,T], \notag\\
& \mu^\star(B)=\Ee^\nu\left[\int_{[0,T]}\mathbf 1_B(\theta,X_\theta)\,k_c^\star(X,d\theta)\right], \quad B\in\mathcal B([0,T]\times\mathbb R).
\end{align}

We now transfer this kernel to the original probability space. Since $X^0$ solves
\[
dX_t^0=\bar b(t,X_t^0)\,dt+\sigma(t,X_t^0)\,dW_t^{\mathrm x},
\]
with $\sigma$ bounded away from zero, we can recover the Brownian motion from $X^0$ through
\[
W_t^{\mathrm x}=\int_0^t \sigma^{-1}(s,X_s^0)\,dX_s^0-\int_0^t \sigma^{-1}(s,X_s^0)\bar b(s,X_s^0)\,ds.
\]
Hence the completed filtrations generated by $(X_0,W^{\mathrm x})$ and by $X^0$ coincide:
\[
\mathbb F^{\mathrm x}=\mathbb F^{X^0}.
\]
Define
\[
k^\star(\omega,B):=k_c^\star(X^0(\omega),B), \qquad B\in\mathcal B([0,T]).
\]
Because $k_c^\star(\cdot,B)$ is $\mathcal F_t^{X,\nu}$-measurable for every $B\in\mathcal B([0,t])$, it follows that $k^\star(\cdot,B)$ is $\mathcal F_t^{X^0}=\mathcal F_t^{\mathrm x}$-measurable. Therefore $k^\star\in\mathcal K^{\mathrm x}$. Moreover, since $\nu$ is the law of $X^0$, \eqref{eq:alt-exist-effective-rep-star} becomes
\begin{align}
\label{eq:alt-exist-effective-rep-original}
& m_t^\star(B)=\Ee\big[\mathbf 1_B(X_t^0)\,k^\star(\cdot,(t,T])\big], \quad B\in\mathcal B(\mathbb R),\ t\in[0,T], \notag\\
& \mu^\star(B)=\Ee\left[\int_{[0,T]}\mathbf 1_B(\theta,X_\theta^0)\,k^\star(\cdot,d\theta)\right], \quad B\in\mathcal B([0,T]\times\mathbb R).
\end{align}
Thus $(k^\star,m^\star,\mu^\star)$ satisfies the consistency part of Definition \ref{def:exist_effective}.

It remains to prove optimality. Let $k\in\mathcal K^{\mathrm x}$ be arbitrary. We construct a path-space kernel $k_c\in\mathcal K_c^{\mathbf X}$ such that
\[
k(\omega,B)=k_c(X^0(\omega),B)\quad\mathbb P\text{-a.s. for all }B\in\mathcal B([0,T]).
\]
Indeed, for every rational $r\in[0,T]\cap\mathbb Q$, the random variable $k(\cdot,[0,r])$ is $\mathcal F_r^{\mathrm x}=\mathcal F_r^{X^0}$-measurable, so by the Doob--Dynkin lemma there exists a Borel map $F(r,\cdot):\Omega_X\to[0,1]$ such that
\[
k(\omega,[0,r])=F(r,X^0(\omega))\quad\mathbb P\text{-a.s.}
\]
One can then choose a full $\nu$-measure set on which $r\mapsto F(r,\mathrm x)$ is increasing on the rationals, extend it to all $t\in[0,T]$ by right-continuity, and obtain a regular probability kernel $k_c\in\mathcal K_c^{\mathbf X}$ satisfying the above identity.

Define the pair $(m^k,\mu^k)$ by
\begin{align*}
& m_t^k(B):=\Ee\big[\mathbf 1_B(X_t^0)\,k(\cdot,(t,T])\big], \quad B\in\mathcal B(\mathbb R),\ t\in[0,T],\\
& \mu^k(B):=\Ee\left[\int_{[0,T]}\mathbf 1_B(\theta,X_\theta^0)\,k(\cdot,d\theta)\right], \quad B\in\mathcal B([0,T]\times\mathbb R).
\end{align*}
Since $k(\cdot,B)=k_c(X^0(\cdot),B)$ a.s. and $\nu$ is the law of $X^0$, we also have
\begin{align*}
& m_t^k(B)=\Ee^\nu\big[\mathbf 1_B(X_t)\,k_c(X,(t,T])\big],\\
& \mu^k(B)=\Ee^\nu\left[\int_{[0,T]}\mathbf 1_B(\theta,X_\theta)\,k_c(X,d\theta)\right].
\end{align*}
Hence, by Theorem \ref{thm:representation2}, $(m^k,\mu^k)\in\mathcal R$.

Now, by the defining optimality of the LP equilibrium $(m^\star,\mu^\star)$,
\begin{align*}
&\int_0^T\!\!\int_{\mathbb R}\bar f(t,x,m_t^\star)m_t^\star(dx)\,dt
+\int_{[0,T]\times\mathbb R} g(t,x,\mu^\star)\mu^\star(dt,dx)
\\ &\geq
\int_0^T\!\!\int_{\mathbb R}\bar f(t,x,m_t^\star)m_t^k(dx)\,dt
+\int_{[0,T]\times\mathbb R} g(t,x,\mu^\star)\mu^k(dt,dx).
\end{align*}
Using \eqref{eq:alt-exist-effective-rep-original} for $(m^\star,\mu^\star)$ and the definitions of $(m^k,\mu^k)$, this inequality is exactly
\[
J^0(k^\star,m^\star,\mu^\star)\geq J^0(k,m^\star,\mu^\star).
\]
Since $k\in\mathcal K^{\mathrm x}$ was arbitrary, $(k^\star,m^\star,\mu^\star)$ is an effective MFG equilibrium with randomized stopping.
\end{proof}

\section{Two-scale Mean-field games of controlled diffusion}
\label{Sec:Control}

\subsection{Notation, assumptions, and definitions}
\label{Sec:Control intro}
Given a filtration $\mathbb{F}$, we denote by $\mathcal{A}(\mathbb{F})$ the set of $\mathbb{F}$--progressively measurable ${A}$--valued processes, where $A \subset \mathbb{R}$ is a convex compact set. Let $\delta >0$. Given $\alpha \in \mathcal{A}(\mathbb{F}^{\bm{W}})$, consider the following system:
\begin{align}
\label{eq:dynamics 2S--c}
\begin{cases}
dX^{\delta,\alpha}_t = ( b_1(t,X^{\delta,\alpha}_t,Y^\delta_t) + b_2(X^{\delta,\alpha}_t,\alpha_t)) dt + \sigma (t,X^{\delta,\alpha}_t) dW^{\text{x}}_t, \quad X^{\delta,\alpha}_0 = X_0 \sim m_0^X,
\\
dY^\delta_t = \frac{1}{\delta} B(Y^\delta_t) dt + \frac{1}{\sqrt{\delta}}\Sigma (Y^\delta_t) dW^{\text{y}}_t + \frac{1}{\sqrt{\delta}}\hat{\Sigma} (Y^\delta_t) dW^{\text{c}}_t, \quad Y^\delta_0 = Y_0 \sim m^Y_0,
\end{cases} \, 0 \leq t \leq T,
\end{align}
with $m_0^X \in \mathcal P(\mathbb R)$ and $m_0^Y \in \mathcal P(\mathbb R)$. 
We impose the following assumptions on the coefficients describing the dynamics, which corresponds to Poisson equation regime in the controlled case. The explicit OU regime could be treated along the same lines, but we omit it to save space. 
\begin{assumption}
\label{hyp:Dynamics_Poisson_X_Y--c} 
\leavevmode 
\begin{enumerate}[label=(\roman*)] 
    \item \label{hyp:RocknerAG--c2} There exists a constant $\lambda>0$ such that for all $x\in \mathbb R$, and $t\in [0,T]$, $\tfrac{1}{\lambda} \leq \sigma^2(t,x)\leq \lambda$

    \item \label{hyp:RocknerR1--1--c2} 
    $m_0^X \in \mathcal{P}_2(\mathbb R)$.     $b_1(\cdot,\cdot,y) \in C^{1/2,1}_b([0,T]\times\mathbb R)$ for all $y\in\Rr$ and $\sigma \in C^{1/2,1}_b([0,T]\times\mathbb R)$. 
    Moreover, the coefficient $b_2$ is continuous and bounded, and there exists $C>0$ such that for all $(t,x,x',y,y',a) \in [0,T]\times \mathbb R^2 \times \Rr^2 \times A$,
    \begin{align*}
    & |b_1(t,x,y)| \leq C(1+|y|), 
    \quad |b_1(t,x,y)-b_1(t,x',y')| \leq C(|x-x'|+|y-y'|), 
    \\
    & |b_2(x,a)-b_2(x',a)| \leq C|x-x'|, 
    \quad |\sigma(t,x)-\sigma(t,x')| \leq C|x-x'|.
    \end{align*}
 
    
    \item \label{hyp:RocknerAs--c2} There exists a constant $\lambda>0$ such that for all $y\in \Rr$, $\tfrac{1}{\lambda} \leq \Sigma^2(y) + \hat{\Sigma}^2(y)\leq \lambda$;

    \item \label{hyp:RocknerAb--c2} $\lim_{|y| \to \infty} yB(y) = -\infty$
  
    \item \label{hyp:RocknerR1--2--c2} $m^Y_0 \in \mathcal{P}_2(\Rr)$, and $B,\Sigma,\hat{\Sigma} \in C^1_b(\Rr)$.
\end{enumerate}
\end{assumption}
Under Assumption \ref{hyp:Dynamics_Poisson_X_Y--c}, for any $\alpha \in \mathcal{A}(\mathbb{F}^{\bm{W}})$ the system \eqref{eq:dynamics 2S--c} admits  a unique strong solution (Theorem 3.1.1 Chapter 3 in \cite{liu2015stochastic}, also Theorem 6.16, Chapter 6 in \cite{yong2012stochastic} and \cite{rozovsky2005note}
) which we denote by $(X^{\delta; \alpha},Y^{\delta})$. 

We introduce the following notation
\begin{align*}
& J^\delta(\alpha,m):=\mathbb{E}\left[\int_0^T f(t,X_t^{\delta,\alpha},Y_t^\delta,\alpha_t,m^{(\text{x},\text{y})}_t)dt+ g(X_{T}^{\delta,\alpha},m^{(\text{x})}_T)\right],
\end{align*}
where $(m_t)_{0\leq t\leq T} \in \mathbb V(\mathbb R \times \mathbb R\times A)$. 
We assume the following on the reward functions $f$ and $g$:
\begin{assumption}
\label{hyp:f and g--c}
The functions $f$ and $g$ have the representation 
\begin{align*}
f(t,x,y,a,\nu) & = F^1(t,x,y) F^2\left(t,x,\int_{\mathbb R}\hat{f}_1(x')\nu^{(\text{x})}(dx')\right)+F^3(t,x)\int_{\mathbb R\times \mathbb{R}}\hat{f}_2(x',y')\nu(dx',dy') 
\\
& \quad + F\left(t,x,a,\int_{\mathbb R}\hat{f}(x')\nu^{(\text{x})}(dx') \right),\qquad \nu \in \mathcal P(\mathbb R\times \mathbb R), 
\\
g(x,\nu) & = G\left(x,\int_{\mathbb R} \hat{g}(x')\nu(dx')\right),\qquad \nu \in \mathcal P(\mathbb R). 
\end{align*}
where $F^1:[0,T]\times \mathbb R\times \mathbb{R} \to \mathbb{R}$, $F^2:[0,T]\times \mathbb R\times \mathbb{R} \to \mathbb{R}$, $F^3:[0,T]\times \mathbb R  \to \mathbb{R}$, $F:[0,T] \times \mathbb R \times A \times \mathbb{R} \to \mathbb{R}$, $\hat{g}: \mathbb R \to \mathbb{R}$, $\hat{f}_1\in C^2_b(\mathbb R) $, $\hat{f}: \mathbb R \to \mathbb{R}$, and $\hat{f}_2:\mathbb R\times \mathbb{R} \to \mathbb{R}$, $G:\mathbb R\times \mathbb{R} \to \mathbb{R}$ are continuous and bounded, $F^2\in C_b^{1,2,1}([0,T]\times \mathbb R\times \mathbb{R})$, $F^3$ is $\gamma^3$-H\"{o}lder continuous uniformly on $x$, with $\gamma^3 \geq 1/2$, and there exits $C>0$ such that for all $(t,x,x',y, y', \beta,\beta',a)\in [0,T]\times \mathbb R^2 \times \mathbb{R}^4 \times A$
\begin{align*}
&|F(t,x,a,\beta)-F(t,x',a,\beta')|, \leq C \left(|x - x'|+|\beta - \beta'|\right), 
\\
&|F^1(t,x,y)-F^1(t,x',y')|,\, |\hat{f}_2(x,y)-\hat{f}_2(x',y')|\leq C \left(|x - x'|+|y - y'|\right), 
\\
& | F^2(t,x,\beta) - F^2(t,x',\beta')|,\,  | G(x,\beta) - G(x',\beta')| \leq C\left(|x - x'|+ |\beta - \beta'| \right),
\\
&| F^3(t,x) - F^3(t,x')|, \, |\hat{f}(x)-\hat{f}(x')|,  \, |\hat{g}(x)-\hat{g}(x')| \leq C |x-x'|, 
\end{align*}
and $F^1(\cdot,\cdot,y) \in C^{1/2,1}_b([0,T]\times \mathbb R)$, $\hat{f}_2(\cdot,y) \in C^{1}_b(\mathbb R)$, for all $y\in \Rr$.
\end{assumption}

We introduce the following definition of a \textit{two-scale $\varepsilon$-MFG equilibrium}.

\begin{definition}[Two-scale $\varepsilon$-MFG and MFG equilibrium with \textit{strict} control]\label{def:Strong Nash2s--strict c}
Let $\varepsilon\geq 0$ and $\delta>0$. We say that $(\alpha^{\varepsilon,\delta},m^{\varepsilon,\delta})$ is a \textit{two-scale $\varepsilon$-MFG equilibrium with strict control} if:
\begin{itemize}
\item[(i)] $\alpha^{\varepsilon,\delta} \in \mathcal{A}(\mathbb{F}^{\bm{W}})$ and for all $\alpha' \in \mathcal{A}(\mathbb{F}^{\bm{W}})$,
\begin{align*}
& J^\delta(\alpha^{\varepsilon,\delta},m^{\varepsilon,\delta}) \geq J^\delta(\alpha',m^{\varepsilon,\delta})-\varepsilon
\end{align*}
\item[(ii)] $(m^{\varepsilon,\delta}_t)_{0 \leq t \leq T}\in \mathbb V(\mathbb R\times \mathbb R \times A)$ such that 
\begin{align*}
& m^{\varepsilon,\delta}_t(B) = \Ee\left[ \textbf{1}_{B}(X^{\delta,\alpha^{\varepsilon,\delta}}_t,Y^{\delta}_t,\alpha^{\varepsilon,\delta}_t) \bigg| \mathcal{F}^{\text{c}}_t\right]\ \text{a.s.} \quad B \in \mathcal{B}(\mathbb R \times \mathbb R \times A),\,\, \quad t \in [0,T]. \notag
\end{align*}
\end{itemize}
When $\varepsilon=0$, we drop the superscript and say that $(\alpha^{\delta},m^{\delta})$ is a \textit{two-scale  MFG equilibrium with strict control}.
\end{definition}
We define MFG equilibria with \textit{relaxed control}. To this end, we first introduce the set of admissible relaxed controls. For a filtration $\mathbb{F}$, let $\mathcal{A}^{\mathcal{R}}(\mathbb{F})$ denote the set of $\mathbb{F}$-progressively measurable processes taking values in $\mathcal{P}(A)$. Under Assumption \ref{hyp:Dynamics_Poisson_X_Y--c}, Theorem 3.1.1 in \cite{liu2015stochastic}, applied to the random coefficient 
\[
(t,x,\omega) \mapsto \int_{A} b_2(x,a)\,\Lambda_t(da)(\omega),
\] 
ensures the existence and uniqueness of a strong solution to the stochastic differential equation with relaxed control
\begin{align}
\label{eq:relaxed}
\begin{cases}
dX^{\delta,\Lambda}_t = ( b_1(t,X^{\delta,\Lambda}_t,Y^\delta_t) + \int_{A} b_2(X^{\delta,\Lambda}_t,a)\,\Lambda_t(da)) dt + \sigma (t,X^{\delta,\Lambda}_t) dW^{\text{x}}_t, \quad X^{\delta,\Lambda}_0 = X_0 \sim m_0^X,
\\
dY^\delta_t = \frac{1}{\delta} B(Y^\delta_t) dt + \frac{1}{\sqrt{\delta}}\Sigma (Y^\delta_t) dW^{\text{y}}_t + \frac{1}{\sqrt{\delta}}\hat{\Sigma} (Y^\delta_t) dW^{\text{c}}_t, \quad Y^\delta_0 = Y_0 \sim m^Y_0,
\end{cases}
\end{align}
 for any $\Lambda \in \mathcal{A}^{\mathcal{R}}(\mathbb{F}^{\bm{W}})$. Let
\begin{align*}
& J^\delta(\Lambda,m):=\mathbb{E}\left[\int_0^T \int_{A}f(t,X_t^{\delta,\Lambda},Y_t^\delta,a,m^{(\text{x},\text{y})}_t)\Lambda_t(da)dt+ g(X_{T}^{\delta,\Lambda},m^{(\text{x})}_T)\right],
\end{align*}
with $ (m_t)_{0\leq t\leq T} \in \mathbb V(\mathbb R \times \mathbb R\times A)$. 
We introduce the following definitions of a \textit{two-scale $\varepsilon$-MFG equilibrium with relaxed control} and \textit{two-scale MFG equilibrium with relaxed control}.
\begin{definition}[Two-scale $\varepsilon$-MFG and MFG equilibrium with \textit{relaxed} control]\label{def:2S e--MFG with randomized control}
Let $\varepsilon\geq 0$ and $\delta>0$. We say that $(\Lambda^{\varepsilon,\delta},m^{\varepsilon,\delta})$ is a \textit{two-scale $\varepsilon$-MFG equilibrium with relaxed control} if:
\begin{itemize}
\item[(i)] $\Lambda^{\varepsilon,\delta} \in \mathcal{A}^\mathcal{R}(\mathbb{F}^{\bm{W}})$ and for all $\Lambda' \in \mathcal{A}^{\mathcal{R}}(\mathbb{F}^{\bm{W}})$,
\begin{align*}
& J^\delta(\Lambda^{\varepsilon,\delta},m^{\varepsilon,\delta}) \geq J^\delta(\Lambda',m^{\varepsilon,\delta}) - \varepsilon
\end{align*}
\item[(ii)] $(m^{\varepsilon,\delta}_t)_{0 \leq t \leq T}\in \mathbb V(\mathbb R\times \mathbb R \times A)$ such that 
\begin{align*}
& m^{\varepsilon,\delta}_t(B \times C) = \Ee\left[ \textbf{1}_{B}(X^{\delta,\Lambda^{\varepsilon,\delta}}_t\!\!,Y^{\delta}_t) \Lambda^{\varepsilon,\delta}_t(C) \bigg| \mathcal{F}^{\text{c}}_t\right] \text{a.s.} \quad B \in \mathcal{B}(\mathbb R\times \mathbb R),\, C \in \mathcal{B}(A),\,\, \quad t \in [0,T]. \notag
\end{align*}
\end{itemize}
When $\varepsilon=0$, we drop the superscript and say that $(\Lambda^{\delta},m^{\delta})$ is a \textit{two-scale equilibrium with relaxed control}.
\end{definition}
We aim to prove the existence of a two-scale $\varepsilon$-MFG equilibrium. To do so, we introduce the following (averaged) dynamics. Under Assumption \ref{hyp:Dynamics_Poisson_X_Y--c}--\ref{hyp:RocknerAs--c2}--\ref{hyp:RocknerAb--c2}, let $\bar{\mu}$ be the invariant measure of \eqref{eq:Fast system}, and let us recall the notation \eqref{eq:averaging}. 

Given $\Lambda \in \mathcal{A}^{\mathcal{R}}(\mathbb{F}^{\bm W})$, we consider the SDE
\begin{align}
\label{eq:dynamics effective--cR}
& dX^{0,\Lambda}_t \!=\! (\bar{b}_1(t,X^{0,\Lambda}_t) + \!\int_{A} \! b_2(X^{0,\Lambda}_t,a)\Lambda_t(da) ) dt + \sigma (t,X^{0,\Lambda}_t) dW^{\text{x}}_t, \quad X^{0,\Lambda}_0 = X_0 \sim m_0^X, \, 0\leq t \leq T.
\end{align}
Similar to \eqref{eq:dynamics 2S--c}, Assumption \ref{hyp:Dynamics_Poisson_X_Y--c} for the slow component, together with the Lipschitz continuity of $\bar{b}_1$ established in \eqref{eq:Lipschitz eff os}, implies that for any $\Lambda \in \mathcal{A}^{\mathcal{R}}(\mathbb{F}^{\bm{W}})$, Theorem 3.1.1 in \cite{liu2015stochastic}, applied with the random coefficient $(t,x,\omega)\mapsto \int_{A} b_2(x,a)\Lambda_t(da)(\omega)$, ensures the existence and uniqueness of a strong solution to \eqref{eq:dynamics effective--cR}. For $\alpha \in \mathcal{A}(\mathbb{F}^{\bm{W}})$, we denote by $X^{0,\alpha}$ the solution corresponding to $\Lambda_t=\delta_{\alpha_t}$.

The cost functional for the effective MFG problem is given by
$$J^{0}(\Lambda,m):=\mathbb{E}\left[\int_0^T \int_{A} \bar{f}(t,X_t^{0,\Lambda},a,m^{(\text{x})}_t)\Lambda_t(da)dt+ g(X_{T}^{0,\Lambda},m^{(\text{x})}_T)\right], \qquad m\in \mathcal V(\mathbb R\times A), $$
where
\begin{align*} 
\bar{f}(t,x,a,\nu)& :=\bar{F}^1(t,x)F^2\left(t,x,\int_{\mathbb R}\hat{f}_1(x')\nu(dx'))\right) + F^3(t,x)\int_{\mathbb R}\bar{\hat{f}}_2(x')\nu(dx')
\\
& \quad + F\left(t,x,a,\int_{\mathbb R}\hat{f}(x')\nu(dx') \right),\qquad \nu \in \mathcal P(\mathbb R). 
\end{align*}
We write $J^{0}(\alpha,m)$ for the same expression with $\Lambda_t$ replaced by $\delta_{\alpha_t}$. 

We now introduce the definitions of \textit{effective MFG equilibrium} with \textit{strict} and \textit{relaxed controls}.
\begin{definition}[Effective MFG equilibrium with \textit{strict} control]\label{def:strict}
We say that $(\alpha,m)$ is an \textit{effective MFG equilibrium with strict control} if
\begin{itemize}
\item[(i)] $\alpha \in \mathcal{A}(\mathbb{F}^{\text{x}})$ and for all $\alpha' \in \mathcal{A}(\mathbb{F}^{\text{x}})$,
\begin{align*}
& J^0(\alpha,m) \geq J^0(\alpha',m)
\end{align*}
\item[(ii)] $(m_t)_{0 \leq t \leq T} \in \mathcal V(\mathbb R\times A)$  satisfies the consistency condition:
\begin{align*}
& m_t(B) = \Ee\left[ \textbf{1}_{B}(X^{0,\alpha}_t,\alpha_t)\right]  \quad B \in \mathcal{B}(\mathbb R \times A),\,\, \quad t \in [0,T]. \notag
\end{align*}
\end{itemize}
\end{definition}
The definition of an MFG equilibrium with \textit{relaxed control} can be written as follows:

\begin{definition}[Effective MFG equilibrium with \textit{relaxed} control]\label{defeffective}
We say that $(\Lambda,m)$ is an \textit{effective MFG equilibrium with relaxed control} if
\begin{itemize}
\item[(i)] $\Lambda \in \mathcal{A}^{\mathcal{R}}(\mathbb{F}^{\text{x}})$ and for all $\Lambda' \in \mathcal{A}^{\mathcal{R}}(\mathbb{F}^{\text{x}})$,
\begin{align*}
& J^0(\Lambda,m) \geq J^0(\Lambda',m)
\end{align*}
\item[(ii)] $(m_t)_{0 \leq t \leq T}$  satisfies the consistency condition:
\begin{align*}
& m_t(B\times C) = \Ee\left[ \textbf{1}_{B}(X^{0,\Lambda}_t)\Lambda_t(C)\right]  \quad B \in \mathcal{B}(\mathbb R),\, C \in \mathcal{B}(A),\,\, \quad t \in [0,T]. \notag
\end{align*}
\end{itemize}
\end{definition}

\subsection{Two-scale $\varepsilon$-MFG equilibrium with \textit{relaxed control}} 
\label{Sec:Control existence e-MFG}
In this section, we present a method for constructing a two-scale $\varepsilon$-MFG equilibrium for the MFG problem with common noise, based on a corresponding equilibrium of the effective MFG problem. The existence of an equilibrium for the effective MFG problem will be addressed in the subsequent section.

To this end, we first establish an averaging result for controlled diffusion processes in a two-scale setting, where the slow component is controlled. This result extends classical averaging principles by deriving effective dynamics for the slow component under control. The derivation relies on the separability assumption on $b_1$ and $b_2$ in Assumption \ref{hyp:Dynamics_Poisson_X_Y--c} --\ref{hyp:RocknerR1--1--c2}. A detailed proof is provided in the Appendix.
\begin{proposition}
\label{pro:Strong rate diffusion--c}
Suppose that Assumption \ref{hyp:Dynamics_Poisson_X_Y--c} holds. Let $0 < \delta <1/2$. Then, for all $T>0$, there exists a constant $C>0$ independent of $\delta$ such that for all $\Lambda \in \mathcal{A}^{\mathcal{R}}({\mathbb{F}^{\bm{W}}})$,
\[
\sup\limits_{0\leq t \leq T} \Ee|X^{\delta,\Lambda}_t - X^{0,\Lambda}_t|^2 \leq C \delta.
\]
\end{proposition}
Using the above result, we obtain the following construction result of a $\varepsilon$-MFG equilibrium with relaxed control.
\begin{theorem}[Construction of a two-scale $\varepsilon$-MFG equilibrium  with \textit{relaxed control}]
\label{relaxed1}
Suppose that Assumptions \ref{hyp:Dynamics_Poisson_X_Y--c} and \ref{hyp:f and g--c} hold. Let $(\Lambda^\star, m^\star)$ be an effective MFG equilibrium with relaxed control. Define 
\begin{align*}
& m^{\delta , \star}_t(B \times C) := \Ee \left[ \textbf{1}_{B}(X^{\delta,\Lambda^{\star}}_t,Y^{\delta}_t)\Lambda^{\star}_t(C)\bigg| \mathcal{F}^{\text{c}}_t \right], \quad B \in \mathcal{B}(\mathbb R\times \mathbb R),\, C \in \mathcal{B}(A),\,\, \quad t \in [0,T],
\end{align*}
where $(X^{\delta,\Lambda^{\star}}, Y^\delta)$ is the strong solution of \eqref{eq:relaxed}. Then, $(\Lambda^\star, m^{\delta,\star})$ is a two--scale $\varepsilon$--MFG equilibrium with relaxed control for $\varepsilon = C \delta^{1/6}$, for some $C>0$ independent of $\delta$. 
\end{theorem}
\begin{proof}
Let $\Lambda \in \mathcal{A}^{\mathcal{R}}(\mathbb{F}^{\bm{W}})$. We have
\begin{align}
\label{ineq1-relax}
& J^\delta(\Lambda,m^{\delta,\star})-J^\delta(\Lambda^\star,m^{\delta,\star}) \notag
\\
& \leq \left|J^\delta(\Lambda,m^{\delta,\star})-J^{0}(\Lambda,m^{\star})\right|  + J^{0}(\Lambda,m^{\star})-J^{0}(\Lambda^\star,m^{\star}) + \left|J^{0}(\Lambda^\star,m^{\star})-J^{\delta}(\Lambda^\star,m^{\delta,\star})\right|.
\end{align}
We write
\begin{align}
\label{eq:Aux Prop 2 1--c}
| J^\delta(\Lambda,m^{\delta,\star}) - J^{0}(\Lambda,m^{\star})| 
& \!=\! \left| \Ee \left[ \int_0^T ( \tilde{A}_t + A_t + \breve{A}_t )dt + B \right] \right|,
\end{align}
where
\begin{align*}
\tilde{A}_t & = F^1(t,X^{\delta,\Lambda}_t,Y^{\delta}_t) F^2(t,X^{\delta,\Lambda}_t, \langle \hat{f}_1, m^{\delta ,\star(\text{x},\text{y})}_t \rangle ) - \bar{F}^1(t,X^{0,\Lambda}_t)F^2(t,X^{0,\Lambda}_t, \langle \hat{f}_1, m_t^{\star(\text{x})} \rangle ),
\\
A_t & = F^3(t,X^{\delta,\Lambda}_t)\langle \hat{f}_2, m^{\delta ,\star(\text{x,y})}_t \rangle  - F^3(t,X^{0,\Lambda}_t)\langle \bar{\hat{f}}_2, m_t^{\star(\text{x})} \rangle,
\\
\breve{A}_t & = \int_{A} \left( F(t,X^{\delta,\Lambda}_t,a, \langle \hat{f}, m^{\delta ,\star(\text{x})}_t \rangle ) - F(t,X^{0,\Lambda}_t,a,\langle \hat{f}, m^{\star(\text{x})}_t \rangle  ) \right) \Lambda_t(da),
\\
B & = G(X^{\delta,\Lambda}_T , \langle \hat{g}, m^{\delta ,\star(\text{x})}_T \rangle ) - G(X^{0,\Lambda}_T ,\langle \hat{g}, m^{\star(\text{x})}_T\rangle  ).
\end{align*}
Throughout the following calculations, the constant $C$ may vary from line to line but is independent of $\delta$. By Proposition \ref{pro:Strong rate diffusion--c}, we follow the same strategy as in the proof of Theorem \ref{thm:Existence 2S epsilon MFG eq} under Assumptions \ref{hyp:Dynamics_Poisson_X_Y} and \ref{hyp:Cost Poisson OS}. Many of the estimates carry over directly when the stopping time is set to $T$. We highlight below the points where the control plays a role.
\\
{\noindent\bf \textit{Step 1: Bound on $\tilde{A}_t$.}} We write
\begin{align}
\label{eq:tildeA expand_control}
\tilde{A}_t & = \tilde{A}^{1}_t  + \tilde{A}^{2}_t + \tilde{A}^3_t,
\end{align}
where
\begin{align*}
\tilde{A}^{1}_t & = F^2(t,X^{0,\Lambda}_t, \langle \hat{f}_1, m_t^{\star(\text{x})} \rangle ) \left( F^1(t,X^{\delta,\Lambda}_t,Y^{\delta}_t) - \bar{F}^1(t,X^{\delta,\Lambda}_t) \right) \\
\tilde{A}^{2}_t & = F^2(t,X^{0,\Lambda}_t, \langle \hat{f}_1, m_t^{\star(\text{x})} \rangle ) \left( \bar{F}^1(t,X^{\delta,\Lambda}_t) -\bar{F}^1(t,X^{0,\Lambda}_t)\right) 
\\
\tilde{A}^{3}_t & = F^1(t,X^{\delta,\Lambda}_t,Y^{\delta}_t) \left( F^2(t,X^{\delta,\Lambda}_t, \langle \hat{f}_1, m^{\delta ,\star(\text{x})}_t \rangle ) - F^2(t,X^{0,\Lambda}_t, \langle \hat{f}_1, m_t^{\star(\text{x})} \rangle ) \right).
\end{align*}
The three terms are bounded in the same way as in the proof of Theorem \ref{thm:Existence 2S epsilon MFG eq} using the second part of Lemma \ref{lemmaM}, leading to 
\begin{align}
\label{eq:tildeA bound final_control}
\left| \Ee \left[ \int_0^T \tilde{A}_t dt \right] \right| & \leq C \sqrt{\delta}.
\end{align}
{\bf \textit{Step 2: Bound on $A_t$.}} We write
\begin{align}
\label{eq:A expand_control}
A_t & = A^{1}_t + A^{2}_t +  A^{3}_t,
\end{align}
where
\begin{align*}
A^{1}_t & = \langle \hat{f}_2, m^{\delta,\star(\text{x},\text{y})}_t \rangle ( F^3(t,X^{\delta, \Lambda}_t) - F^3(t,X^{0, \Lambda}_t) ),
\\
A^{2}_t & = F^3(t,X^{0,\Lambda}_t)  \left( \langle \bar{\hat{f}}_2, m^{\delta,\star(\text{x})}_t \rangle - \langle \bar{\hat{f}}_2, m^{\star(\text{x})}_t \rangle \right),\\
A^{3}_t & = F^3(t,X^{0,\Lambda}_t)  \left( \langle \hat{f}_2, m^{\delta,\star(\text{x},\text{y})}_t \rangle - \langle \bar{\hat{f}}_2, m^{\delta,\star(\text{x})}_t \rangle \right).
\end{align*}
Similarly to \eqref{eq:A1 bound 1} and \eqref{eq:A22 bound}, by Proposition \ref{pro:Strong rate diffusion--c}, we obtain 
\begin{align}
\label{eq:A1 bound 1_control}
&\left| \Ee \left[ \int_0^T (A^{1}_t + A^{2}_t) dt \right] \right| \leq C \sqrt{\delta}.
\end{align}
The bound on $A^3$ is obtained by the same discretization procedure as in the proof of Theorem \ref{thm:Existence 2S epsilon MFG eq}. In particular, using the regularity of the coefficients in \eqref{eq:dynamics effective--cR}, the boundedness of $b_2$, and the fact that $\Lambda_t$ takes values in $\mathcal P(A)$, there exists a constant $C>0$ independent of $\delta$ and $\Lambda$ such that
\[
\Ee [ | X^{0,\Lambda}_t - X^{0,\Lambda}_s |^2 ] \leq C |t-s|.
\]
Indeed, since $\Lambda_t \in \mathcal P(A)$, we have
\[
\left| \int_A b_2(x,a)\Lambda_t(da)\right| \leq \sup_{a \in A}|b_2(x,a)| \leq C,
\]
so the controlled drift in \eqref{eq:dynamics effective--cR} is uniformly bounded in $\Lambda$. Using the discretization procedure as above, we then get \begin{align}
\label{eq:A3 bound final_control}
\left| \Ee \left[ \int_0^T A^{3}_t dt \right] \right|  \leq C \delta^{\tfrac{1}{6}}.
\end{align}
Combining this estimate with \eqref{eq:A1 bound 1_control}, we conclude:
\begin{align}
\label{eq:A bound final_control}
\left| \Ee \left[ \int_0^T A_t dt \right] \right| \leq C \delta^{\tfrac{1}{6}}.
\end{align}

{\bf \noindent\textit{Step 3: Bound on \eqref{eq:Aux Prop 2 1--c}}.} By Lipschitz continuity of $F$ and $\hat{f}$, and Proposition \ref{pro:Strong rate diffusion--c}, we have
\begin{align}
\Ee \left[ \left|\breve{A}_t\right| \right] & \leq C \Ee \left[ \left| X^{\delta,\Lambda}_t - X^{0,\Lambda}_t \right| + \Ee \left[ \left| X^{\delta,\Lambda^\star}_t - X^{0,\Lambda^\star}_t \right| \bigg| \mathcal{F}^{\text{c}}_t \right] \right] \leq C \sqrt{\delta}.
\end{align}
A similar bound follows for $B$, and the previous steps provide the bound
\begin{align}
\label{eq:Aux Prop 2 1--c_bound}
\left| J^\delta(\Lambda,m^{\delta,\star}) - J^{0}(\Lambda,m^{\star})\right| & \leq C \delta^{1/6}.
\end{align}
By using similar computations as for \eqref{eq:Aux Prop 2 1--c_bound}, we derive that
\begin{align}
\label{eq:Strong Third term bound_control}
\left|J^{0}(\Lambda^\star,m^{\star})-J^{\delta}(\Lambda^\star,m^{\delta,\star})\right| \leq C \delta^{1/6}.
\end{align}
For a given $\Lambda \in \mathcal{A}^{\mathcal{R}}(\mathbb{F}^\textbf{W})$, let
\[
m_t^\Lambda(B\times C):=\Ee\left[\textbf{1}_B(X_t^{0,\Lambda})\Lambda_t(C)\right], \qquad
\mu^\Lambda(B):=\Ee\left[\textbf{1}_B(X_T^{0,\Lambda})\right].
\]
By the disintegration theorem (\cite{carmona2018probabilistic}, Theorem 1.1), there exists a measurable kernel $\hat{\Lambda}:[0,T]\times \mathbb R \to \mathcal P(A)$ such that
\[
m_t^\Lambda(dx,da)=\hat{\Lambda}(t,x)(da)\,m_t^{\Lambda,(\text{x})}(dx).
\]
Moreover, by Theorem C.6 in \cite{4LPApproach}, there exists a weak solution $\hat X$ of the effective dynamics driven by the feedback relaxed control $\hat{\Lambda}(t,\hat X_t)$ such that
\[
\Ee\left[\textbf{1}_B(\hat X_t)\hat{\Lambda}(t,\hat X_t)(C)\right]=m_t^\Lambda(B\times C), \qquad
\Ee\left[\textbf{1}_B(\hat X_T)\right]=\mu^\Lambda(B).
\]
Since the effective payoff is linear in the occupation measure and the terminal law, this implies
\begin{align}\label{ineq2-relax}
J^{0}(\Lambda,m^\star) = J^{0}(\hat{\Lambda}(t,\hat X_t), m^\star).
\end{align}
Writing
\begin{align}
J^0(\Lambda,m^\star)-J^{0}(\Lambda^\star,m^\star) = J^0(\Lambda,m^\star)-J^{0}(\hat{\Lambda}(t,\hat X_t), m^\star)+J^{0}(\hat{\Lambda}(t,\hat X_t), m^\star)-J^{0}(\Lambda^\star,m^\star),
\end{align}
from the identity \eqref{ineq2-relax} and since $\hat{\Lambda}(\cdot,\hat X_\cdot) \in \mathcal{A}^{\mathcal{R}}(\mathbb{F}^{\text{x}})$ and $(\Lambda^\star,m^\star)$ is a MFG equilibrium for the effective problem, we obtain that 
\begin{align}
\label{eq:Strong Second term bound_control}
J^0(\Lambda,m^\star)-J^{0}(\Lambda^\star,m^\star) \leq 0.
\end{align}
By combining \eqref{ineq1-relax}, \eqref{eq:Aux Prop 2 1--c_bound}, \eqref{eq:Strong Third term bound_control}, and \eqref{eq:Strong Second term bound_control}, we obtain that
\begin{align*}
J^\delta(\Lambda, m^{\delta,\star})-J^\delta(\Lambda^\star, m^{\delta,\star}) \leq C \delta^{1/6}.
\end{align*}
Because $\Lambda \in \mathcal{A}^{\mathcal{R}}(\mathbb{F}^{\bm{W}})$ is arbitrary, we conclude that $(\Lambda^\star, m^{\delta,\star})$ is a two--scale $\varepsilon$--MFG equilibrium with relaxed control for $\varepsilon = C \delta^{1/6}$, where $C>0$ is independent of $\delta$.

\end{proof}

\begin{remark}
\label{rem:orders_control}
Similar to Remarks \ref{rem:Poisson orders} and \ref{rem:Common noise and orders}, in Theorem \ref{relaxed1}, the dependence of $\hat{f}_2$ on the fast scale variable $y$ lowers the approximation order to $1/6$, since in {\bf \textit{Step 2}} the discretization of the term $A^{3}_t$ forces a step size whose sharpest bound yields this rate. Without such dependence, the higher order $1/2$ from {\bf \textit{Step 1}} would be recovered. This loss of order is tied to the presence of common noise: it introduces conditional expectations that require discretization. In contrast, without common noise, the corresponding terms could be estimated directly, yielding the sharper convergence rate $1/2$.
\end{remark}
To ensure the existence of equilibria with strict controls, we require the admissible set of coefficients and costs to satisfy the following convexity assumption.
\begin{assumption}
\label{hyp:Convex}
For all $(t,x,m)$, 
\[
K(t,x):= \left\{ (\bar{b}_1(t,x)+b_2(x,a),\sigma(t,x),z):\, a\in A,\, z\leq \bar{f}(t,x,a,m_t) \right\} \subset \Rr \times \Rr_+ \times \Rr
\]
is convex.
\end{assumption}
Under the above Assumption, we will show in the next Section that we have existence of an \textit{effective MFG equilibrium} with \textit{strict control}. In this case, Theorem \ref{relaxed1} can be written as follows. 
\begin{theorem}[Construction of a two-scale $\varepsilon$-MFG equilibrium  with \textit{strict control}]
\label{pro:epsilon Nash strong--c}
Suppose that Assumptions \ref{hyp:Dynamics_Poisson_X_Y--c}, \ref{hyp:f and g--c}, and \ref{hyp:Convex} hold. Let $(m^\star, \alpha^\star)$ be an effective MFG equilibrium. Define
\begin{align*}
& m^{\delta , \star}_t(B) := \Ee \left[ \textbf{1}_{B}(X^{\delta,\alpha^{\star}}_t, Y^{\delta}_t, \alpha^\star_t)\bigg| \mathcal{F}^{\text{c}}_t \right], \quad B \in \mathcal{B}(\mathbb R \times \mathbb R \times A),
\end{align*}
where $(X^{\delta,\alpha^{\star}}, Y^\delta)$ is the strong solution of \eqref{eq:dynamics 2S--c}. Then, $(\alpha^\star, m^{\delta,\star})$ is a two--scale $\varepsilon$--MFG equilibrium  for $\varepsilon = C \delta^{1/6}$, for some $C>0$ independent of $\delta$. 
\end{theorem}

\subsection{Existence of an equilibrium for the  \textit{effective} MFG problem}
\label{Sec:Control existence effective}

In this section, we prove that there exists an equilibrium for the \textit{effective MFG problem}. This result is based on the the existence result of a MFG equilibrium in the linear programming (LP) sense developed in \cite{4LPApproach}, together with a probabilistic representation of the occupation measures involved in the LP formulation.

\begin{proposition}\label{admissibility} Let $(\Omega, \mathcal{F},\mathbb{F}, \mathbb{P})$ be a filtered probability space,  $\nu$ an $\mathbb{F}$-progressively measurable process with values in $\mathcal{P}(A)$, $W$ an $\mathbb{F}$-Brownian motion and $X$ an $\mathbb{F}$-adapted process such that
\begin{align}
dX_t=\bar b_1(t,X_t)dt + \int_A b_2(X_t,a)\nu_t(da)dt+\sigma(t,X_t)dW_t,\,\, 0 \leq t \leq T, \,\, X_0 \sim m_0^X.
\end{align}
Define the associated measures
\begin{align}\label{mesmu}
\mu:=\mathbb{P} \circ X_T^{-1}
\end{align}
and
\begin{align}\label{mesm}
m_t(B \times C):= \mathbb{E}^{\mathbb{P}}[\textbf{1}_{B}(X_t)\nu_t(C)],\,\, B \in \mathcal{B}(\mathbb R),\,\, C \in \mathcal{B}(A).
\end{align}
 Then $(m,\mu) \in \mathcal{R}$, where
$\mathcal{R}$ represents the set of pairs $(m,\mu)\in \mathcal{P}(\mathbb R)\times \mathcal V(\mathbb R\times A)$, such that for all $u\in C_b^{1, 2}([0, T]\times \mathbb R)$,
\begin{align*}
  \int_{\mathbb R} \! u(T, x)\mu(dx)&= \int_{\mathbb R} u(0, x)m_0^X(dx) + \int_0^T \int_{\mathbb R\times A} (\partial_t u + \bar{\mathcal{L}} u ) (t, x, a)m_t(dx, da)dt,\\
  \bar{\mathcal{L}} u(t,x,a) &= (\bar b_1(t,x) + b_2(x,a))\frac{\partial u}{\partial x}(t,x) + \frac{\sigma^2(t,x)}{2} \frac{\partial^2 u}{\partial x^2}(t,x). 
\end{align*}  
 \end{proposition}
 \begin{proof}
The result follows by an application of Itô's formula (see Proposition 2.7 in \cite{4LPApproach}).
\end{proof}
Based on this observation, the following definition of  linear programming MFG equilibrium for the effective MFG problem is given in \cite{4LPApproach}.
\begin{definition}[Linear programming MFG equilibrium] 
\label{def:LP Nash effective--c}
We say that $(m^{\star},\mu^{\star}) \in \mathcal{R}$ is \textit{a linear programming MFG equilibrium} if
\begin{align}
&\int_{0}^T\int_{\mathbb R \times A} \bar{f}(t,x,a,(m^\star_t)^x)m^\star_t(dx,da)dt+\int_{\mathbb R} g(x,\mu^\star)\mu^\star(dx) \nonumber
\\ 
& \geq   \int_{0}^T\int_{\mathbb R \times A} \bar{f}(t,x,a,(m^\star_t)^x)m_t(dx,da)dt+ \int_{\mathbb R} 
 g(x,\mu^\star)\mu(dx), 
\end{align}
for all $(m,\mu) \in \mathcal{R}$.
\end{definition}





\begin{theorem}[Existence of an \textit{effective} MFG equilibrium] 
\label{existLP--c} 
Suppose that Assumptions \ref{hyp:Dynamics_Poisson_X_Y--c} and \ref{hyp:f and g--c}
hold.  Then there exists an \textit{effective} MFG equilibrium with \textit{relaxed control}. Moreover, if Assumption \ref{hyp:Convex} is satisfied, there exists an effective MFG equilibrium with strict control. 
\end{theorem}
\begin{proof}
Let $(m,\mu) \in \mathcal{R}$. By the disintegration theorem (\cite{carmona2018probabilistic}, Theorem 1.1), there exists a regular kernel $\nu: [0,T] \times \mathbb R \times \mathcal{B}(A) \to [0,1]$ such that 
\begin{align}
\label{eq:control decomposition}
m_t(dx,da)dt = \nu_{t,x}(da) m_t(dx,A)dt.
\end{align}
By Theorem C.6. in \cite{4LPApproach}, there exist a filtered probability space $(\Omega,\mathcal{F},\mathbb{F},\mathbb{P})$, a Brownian motion $W^{\text{x}}$, and an $\mathbb{F}$--adapted process $X$ satisfying
\begin{align}
\label{eq:effective-rep}
& dX_t =\left(\bar{b}_1(t,X_t) + \int_{A} b_2(X_t,a) \nu_{t,X_t}(da) \right)dt + \sigma(t,X_t)dW^{\text{x}}_t, \quad 0\leq t \leq T, \quad X_0 \sim m_0^X, \notag
\\
& m_t(B \times C)  = \Ee^{\mathbb{P}}\left[ \textbf{1}_B(X_t)\nu_{t,X_t}(C) \right], \quad B \in \mathcal{B}(\mathbb R), \, C \in \mathcal{B}(A), \quad t \in [0,T], \notag
\\
& \mu(B) = \Ee^{\mathbb{P}}\left[\textbf{1}_B(X_T ) \right], \quad B \in \mathcal{B}(\mathbb R).
\end{align}
Since the stochastic differential equation in \eqref{eq:effective-rep} admits an unique strong solution (see e.g. Theorem 4 in \cite{veretennikov1980strong}, Theorem 4-2 in \cite{zvonkin1974transformation} and Theorem 3 in \cite{zvonkin1975strong}), we conclude that the representation \eqref{eq:effective-rep} for $(m,\mu)$ holds on the initial probability space.
We therefore have 
\begin{align}\label{Rbelongs}
\mathcal{R} \subset \{(m^\Lambda, \mu^\Lambda),\,\,  \Lambda \in \mathcal{A}^{\mathcal{R}}(\mathbb{F}^{\text{x}})\}.
\end{align}

Conversely, by Proposition \ref{admissibility}, if we consider the initial probability space $(\Omega,\mathcal{F}, \mathbb{F}^{\text{x}},\mathbb{P})$, $\Lambda$ a $\mathbb{F}^{\text{x}}$-progressively measurable process with values in $\mathcal{P}(A)$ and the process $X$ given by \eqref{eq:dynamics effective--cR}, then  the couple of measures $(m,\mu)$ defined by \eqref{mesmu}-\eqref{mesm} belongs to $\mathcal{R}$, i.e.
\begin{align}\label{belongstoR}
 \{(m^\Lambda, \mu^\Lambda),\,\,  \Lambda \in \mathcal{A}^{\mathcal{R}}(\mathbb{F}^{\text{x}})\} \subset \mathcal{R}.
\end{align}
In view of \eqref{Rbelongs} and \eqref{belongstoR} there is a one-to-one correspondence between $\mathcal R$ and the set $\{(m^\Lambda, \mu^\Lambda),\,\,  \Lambda \in \mathcal{A}^{\mathcal{R}}(\mathbb{F}^{\text{x}})\}$. Therefore, from Theorem 3.11 in \cite{4LPApproach}, which establishes the existence of an LP MFG equilibrium, we derive the existence of an effective MFG equilibrium with \textit{relaxed control} in the sense of Definition \ref{defeffective}. Under Assumption \ref{hyp:Convex}, in view of Proposition 3.22 in \cite{4LPApproach}, there exists an effective MFG equilibrium with strict control in the sense of Definition \ref{def:strict}. 
\end{proof}

\section{Appendix}
\label{Sec:Appendix}

\subsection{Equality of supremum under independent enlargement of filtration}
\noindent \textbf{Proof of Proposition \ref{pro:sup equality} .}
\begin{proof}
We focus on the inequality 
\begin{align}
\label{eq:Sups ineq}
\sup_{\tau \in \mathcal{T}(\mathbb{F}^{\bm{W}})} J^0(\tau,m,\mu) \leq \sup_{\tau \in \mathcal{T}(\mathbb{F}^{\text{x}})} J^0(\tau,m,\mu).
\end{align}
Consider the Snell envelope (\cite{Peskir}, Chapter I, Section 2) $\overline{Y}^1 = (\overline{Y}^1_t)_{0 \leq t \leq T}$ of the process $Y_t := \int_0^t \bar{f}(s,X^0_s,m_s) ds + g(t,X^0_t,\mu)$ w.r.t. $\mathbb{F}^{\text{x}}$; that is, $\bar{Y}^1_t = \text{esssup}_{\tau \geq t} \Ee[Y_\tau|\mathcal{F}^{\text{x}}_t]$. Because $Y$ is $\mathbb{F}^{\text{x}}$--adapted, it is also $\mathbb{F}^{\bm{W}}$--adapted, and we consider its Snell envelope $\overline{Y}^2$ w.r.t. $\mathbb{F}^{\bm{W}}$. From the Doob--Meyer decomposition of the super--martingale $\overline{Y}^1$, we have $\overline{Y}^1 = M - A$, where $M$ is a $\mathbb{F}^{\text{x}}$--martingale and $A$ is an adapted non-decreasing process. Because $X_0$ and $W^{\text{x}}$ are independent from $Y_0$, $W^{\text{y}}$, and $W^{\text{c}}$, the filtrations they generate are independent. By Proposition 1.21 in \cite{aksamit2017enlargement}, this implies that $M$ is also a $\mathbb{F}^{\bm{W}}$--martingale. Hence, $\overline{Y}^1$ is a $\mathbb{F}^{\bm{W}}$–super--martingale dominating $Y$. Since $\overline{Y}^2$ is the minimal such super--martingale, we obtain $\overline{Y}^2 \leq \overline{Y}^1$, and in particular $\overline{Y}^2_0 \leq \overline{Y}^1_0$, which yields \eqref{eq:Sups ineq}. The opposite inequality follows by taking the super--martingale $(\Ee[\overline{Y}^2_t|\mathcal{F}^{\text{x}}_t])_{0\leq t \leq T}$. The claim follows.
\end{proof}

\subsection{Moment estimates on the invariant measure}

\noindent \textbf{Proof of Proposition \ref{pro:E invariant m estimate}.}
\begin{proof}
Under Assumptions \ref{hyp:Dynamics_Discret_X_Y}-\ref{hyp:Dynamics_Discret_X_Y A1}, the invariant distribution is Gaussian and therefore has finite absolute moments of every order. Suppose now that Assumptions \ref{hyp:Dynamics_Poisson_X_Y}--\ref{hyp:RocknerAs2}--\ref{hyp:RocknerAb2}--\ref{hyp:fast coeff} hold. Set
\[
a(y):=\Sigma^2(y)+\hat{\Sigma}^2(y).
\]
By Assumption \ref{hyp:Dynamics_Poisson_X_Y}--\ref{hyp:RocknerAs2}, there exists $\lambda>0$ such that $\lambda^{-1}\le a(y)\le \lambda$ for all $y\in\Rr$.

By the standard one-dimensional invariant-density formula for the diffusion \eqref{eq:Fast system} (see, e.g., \cite{metafune2005global}, Theorem 2.1, and \cite{bianca2017existence}, Proposition 5.3), $\bar{\mu}$ has a density of the form
\begin{align}
\label{eq:Aux Invm0}
f_{\bar{\mu}}(y)
=\frac{C_0}{a(y)}
\exp\left( \int_0^y \frac{2B(z)}{a(z)}\,dz \right),
\end{align}
where $C_0>0$ is a normalizing constant. 

Fix $p'\geq 1$. Since $yB(y)\to -\infty$ as $|y|\to\infty$, we may choose $K>0$ such that $2K/\lambda>p'+1$, and then choose $R>1$ so that
\[
yB(y)\le -K,\qquad |y|\ge R.
\]
For $y\ge R$, this implies $B(y)\le -K/y$. Hence, using $a(y)\le \lambda$ and the fact that $B(y)<0$ on $[R,\infty)$,
\begin{align*}
\int_0^y \frac{2B(z)}{a(z)}\,dz
&\le C_R+\int_R^y \frac{2B(z)}{\lambda}\,dz
\le C_R-\frac{2K}{\lambda}\int_R^y \frac{dz}{z}
\le C_R-\frac{2K}{\lambda}\log y,
\end{align*}
where $C_R$ denotes a finite constant depending only on $R$. Therefore,
\[
f_{\bar{\mu}}(y)\le C y^{-2K/\lambda},\qquad y\ge R.
\]
Similarly, for $y\le -R$, the inequality $yB(y)\le -K$ gives $B(y)\ge K/|y|$. Consequently,
\begin{align*}
\int_0^y \frac{2B(z)}{a(z)}\,dz
&=C_R-\int_y^{-R}\frac{2B(z)}{a(z)}\,dz
\le C_R-\frac{2K}{\lambda}\int_y^{-R}\frac{dz}{|z|}
\le C_R-\frac{2K}{\lambda}\log |y|,
\end{align*}
and hence
\[
f_{\bar{\mu}}(y)\le C |y|^{-2K/\lambda},\qquad y\le -R.
\]
Since $2K/\lambda>p'+1$, the two tail estimates imply
\[
\int_{\Rr}|y|^{p'}\bar{\mu}(dy)<\infty.
\]
\end{proof}

\subsection{Existence of the projection with kernels}

\begin{lemma}
Fix $k \in \mathcal{K}^{\bm{W}}$. There exists a measure flow $(m_t)_{0 \leq t \leq T}$ which is a version of the conditional expectation flow $\mathbb{E}\left[\textbf{1}_\cdot(X_t)k(\cdot, (t,T])|\mathcal{F}_t^{\text{c}} \right]$. 
\end{lemma}

\begin{proof}
(i) We show first that there exists a regular probability kernel $m:\Omega \times \mathcal{B}(C([0,T];\Rr)) \times \mathcal{B}([0,T]) \mapsto [0,1]$ such that $m(A \times B)=\mathbb{E}\left[\textbf{1}_A(X) k(\cdot,B)|\mathcal{F}_t^{\text{c}}\right]$ a.s. Define $F_\omega(t):=k(\omega,[0,t])$. For each $\omega$, let $\bar{F}_\omega^{-1}:[0,1] \mapsto [0,T]$ denote the right-continuous inverse of $t \mapsto F_\omega(t)$. Consider the enlarged space $(\Omega \times \bar{\Omega},\bar{\mathcal{F}},\bar{\mathbb{F}},\bar{\mathbb{P}})$, where $\bar{\Omega}$ supports a random variable $U$ uniformly distributed between $0$ and $1$ independent of $W^{\text{c}}$, the $\sigma$-algebra $\mathcal{F}$ is defined by $\bar{\mathcal{F}}=\mathcal{F} \otimes \sigma(U)$, the filtration $\bar{\mathbb{F}}$ is given by $\bar{\mathbb{F}}:=\{\mathcal{F}_t^{\boldsymbol{W}}\otimes \sigma(U),\,\, 0 \leq t \leq T \}$ and $\bar{\mathbb{P}}(dw,dw')=\mathbb{P}(dw)\mathbb{P}^1(dw')$.\\
We have 
\begin{align}
\left\{ (\omega, \bar{\omega}) \in \Omega \times \bar{\Omega}: \bar{F}^{-1}_\omega(U(\bar{\omega})) \leq t \right\}=\left\{ (\omega, \bar{\omega}) \in \Omega \times \bar{\Omega}: U(\bar{\omega}) \leq F_\omega(t) \right\} \in \bar{\mathcal F}_t,
\end{align}
which implies that $\tau:=\bar{F}^{-1}_\omega(U(\bar{\omega}))$ is a $\bar{\mathbb{F}}$-stopping time. We have 
$$\bar{\mathbb{P}}(\bar{F}^{-1}_\omega(U(\bar{\omega})) \leq t |\mathcal{F}^{\bm{W}})=k(\cdot,[0,t]) \text{\,\, a.s.}$$

We therefore have 
\begin{align}
\mathbb{E} [\textbf{1}_{A}(X) k(\cdot,[0,t])|\mathcal{F}^{\text{c}}_T]=\bar{\mathbb{E}} [\textbf{1}_{A}(X) \textbf{1}_{[0,t]}(\tau)|\mathcal{F}^{\text{c}}_T] \text{\,\, a.s.}
\end{align}
By a monotone class argument, the above equality extends to all $B \in \mathcal{B}([0,T])$.

Using standard results, since $(X,\tau)$ takes values in the Polish space $C([0,T];\Rr)\times[0,T]$, there exists a regular conditional probability version of the above conditional probability.\\

(ii) Observe that $m_t(A)=m(\pi_t^{-1}(A) \times (t,T])=\mathbb{E}[\textbf{1}_A(X_t)k(\cdot,(t,T])|\mathcal{F}_t^{\text{c}}]=\mathbb{E}[\textbf{1}_A(X_t)k(\cdot,(t,T])|\mathcal{F}_t^{\text{c}}]$ a.s., since the processes $(X_t)$ and $(k(\cdot,(t,T]))$ are $\mathbb{F}$-adapted and $W^{\text{c}}$ is a $\mathbb{F}$-Brownian motion. Furthermore, since the filtration is complete, the flow $(m_t)$ is $\mathbb{F}^{\text{c}}$-adapted.
\end{proof}

\subsection{Proof of Proposition \ref{pro:Strong rate sup}: Assumption \ref{hyp:Dynamics_Poisson_X_Y}}
\begin{proof}
We start by recalling the following fixed-time estimate which follows directly from Theorem 2.1 in \cite{rockner2021averaging} (by selecting $\alpha_\delta = \sqrt{\delta}$, $\gamma_\delta, \beta_\delta=1$, $H,c\equiv 0$, $\vartheta =1$, and $q=2$ in their framework, corresponding to their Regime 1 in (1.9)). 
Let $0 < \delta <1/2$. For all $T>0$, there exists a constant $C>0$ independent of $\delta$ such that 
\begin{align}
\sup\limits_{0\leq t \leq T} \Ee|X^\delta_t - X^0_t|^2 \leq C \delta.\label{strong rate part 1}
\end{align}

We now proceed with the proof of Proposition \ref{pro:Strong rate sup}. Using the integral representation of \eqref{eq:dynamics 2S} and \eqref{eq:dynamics effective}, we have
\begin{align}
\label{Aux Integral rep}
X^\delta_t - X^0_t = \int_0^t \!( b(s,X^\delta_s,Y^\delta_s) - \bar{b}(s,X^0_s) ) ds + \int_0^t \! ( \sigma(s, X^\delta_s)-\sigma(s, X^0_s) ) dW^{\text{x}}_s =: A_t + M_t.
\end{align}
By the BDG inequality with exponent $p=2$, the Lipschitz continuity of $\sigma$, \eqref{strong rate part 1}, and Jensen’s inequality, we obtain
\begin{align}
\label{eq:M 1}
\Ee \left[ \sup_{0 \leq t \leq T} |M_t| \right] \leq 
C \Ee \left[ \int_0^T \left(\sigma(s,X^\delta_s)-\sigma(s,X^0_s) \right)^2 ds \right]^{1/2} \leq C \sqrt{\delta}.
\end{align}
Writing
\begin{align}
\label{Aux At def}
A_t = \int_0^t \left( b(s,X^\delta_s,Y^\delta_s) - \bar{b}(s,X^\delta_s) \right) ds + \int_0^t \left( \bar{b}(s,X^\delta_s) - \bar{b}(s,X^0_s) \right) ds =: A^1_t + A^2_t,
\end{align}
the Lipschitz continuity of $\bar{b}$ (see \eqref{eq:Lipschitz eff os}) and \eqref{strong rate part 1} give
\begin{align}
\label{Aux A2t bound}
\Ee  \left[ \sup_{0\leq t \leq T} \left| A^2_t \right| \right] \leq C \int_0^T \Ee \left[ |X^\delta_s - X^0_s| \right] ds \leq C \sqrt{\delta}. 
\end{align}
The estimate for $A^1$ follows from Lemma \ref{lem:supL2_H}. 
\end{proof}

\subsection{Proof of Proposition \ref{pro:Strong rate sup}: Assumption \ref{hyp:Dynamics_Discret_X_Y}}
\begin{proof}
We start by recalling the fixed-time estimate from Theorem 2.3 in \cite{rockner2021strong}.
  Let $0 < \delta <1/2$. For all $T>0$, there exists a constant $C>0$ independent of $\delta$ such that 
\begin{align}
\sup_{t\in[0,T]} \Ee \left[ | X^\delta_t - X^0_t |^2 \right] \leq C \delta^{2/3}.\label{strong rate part 2}
\end{align}

We now proceed with the proof of Proposition  \ref{pro:Strong rate sup}. Using the integral forms of \eqref{eq:dynamics 2S} and \eqref{eq:dynamics effective}, write
\[
X_t^\delta-X_t^0
=R_t^\delta+A_t^\delta+M_t^\delta,
\]
where
\[
R_t^\delta:=\int_0^t \Phi(s,X_s^\delta,Y_s^\delta)\,ds,\qquad
A_t^\delta:=\int_0^t(\bar b(s,X_s^\delta)-\bar b(s,X_s^0))\,ds,
\]
and
\[
M_t^\delta:=\int_0^t(\sigma(s,X_s^\delta)-\sigma(s,X_s^0))\,dW_s^{\mathrm x}.
\]
The terms $A^\delta$ and $M^\delta$ are controlled as in the preceding proof, using \eqref{strong rate part 2} in place of the fixed-time estimate under Assumption \ref{hyp:Dynamics_Poisson_X_Y}. More precisely,
\begin{equation}
\label{eq:OU A M bound direct}
\Ee\left[\sup_{0\leq t\leq T}|A_t^\delta|\right]
+\Ee\left[\sup_{0\leq t\leq T}|M_t^\delta|\right]
\leq C\delta^{1/3}.
\end{equation}

It remains to estimate $R^\delta$. Let $\Delta\in(0,1)$, to be chosen later, and set $t_k=k\Delta$ for $k=0,\ldots,N$, where $N=\lceil T/\Delta\rceil$. Put $t_N=T$ if necessary. For $s\in[t_k,t_{k+1})$, write
\[
\Phi(s,X_s^\delta,Y_s^\delta)
=\Phi(t_k,X_{t_k}^\delta,Y_s^\delta)
+\big(\Phi(s,X_s^\delta,Y_s^\delta)-\Phi(t_k,X_{t_k}^\delta,Y_s^\delta)\big).
\]
By the Lipschitz continuity of $b$ in Assumption \ref{hyp:Dynamics_Discret_X_Y} and standard moment estimates for $X^\delta$, the freezing error satisfies
\begin{equation}
\label{eq:OU freezing error sup}
\Ee\left[
\sup_{0\leq t\leq T}
\left|\int_0^t\big(\Phi(s,X_s^\delta,Y_s^\delta)-\Phi(\theta_s^\Delta,X_{\theta_s^\Delta}^\delta,Y_s^\delta)\big)\,ds\right|
\right]
\leq C\Delta^{1/2},
\end{equation}
where $\theta_s^\Delta:=t_k$ for $s\in[t_k,t_{k+1})$.

Define the block variables
\[
I_k:=\int_{t_k}^{t_{k+1}}\Phi(t_k,X_{t_k}^\delta,Y_s^\delta)\,ds,
\qquad
\mathcal G_k:=\mathcal F_{t_k}^{\mathbf W}.
\]
The explicit Ornstein--Uhlenbeck transition formula gives, for $s\geq t_k$,
\[
Y_s^\delta=\theta+(Y_{t_k}^\delta-\theta)e^{-\kappa(s-t_k)/\delta}
+\frac{\sqrt{\sigma_1^2+\sigma_2^2}}{\sqrt{\delta}}
\int_{t_k}^s e^{-\kappa(s-r)/\delta}\,d\widetilde W_r,
\]
where $\widetilde W$ is a Brownian motion obtained from $(W^{\mathrm y},W^{\mathrm c})$. Since $y\mapsto b(t,x,y)$ is Lipschitz and $\bar b(t,x)$ is the average with respect to the invariant Gaussian law, this formula implies
\begin{equation}
\label{eq:OU block mean}
\left|\Ee[I_k\mid\mathcal G_k]\right|
\leq C\delta(1+|Y_{t_k}^\delta|),
\end{equation}
and
\begin{equation}
\label{eq:OU block variance}
\Ee\left[\left|I_k-\Ee[I_k\mid\mathcal G_k]\right|^2\mid\mathcal G_k\right]
\leq C\delta\Delta(1+|Y_{t_k}^\delta|^2).
\end{equation}
Indeed, \eqref{eq:OU block mean} follows by integrating the exponential memory term over $[t_k,t_{k+1}]$, and \eqref{eq:OU block variance} follows from the covariance formula of the Gaussian OU process and the Lipschitz bound on $\Phi$.

The moment bound implied by Assumption \ref{hyp:Dynamics_Discret_X_Y} gives
\[
\sup_{0<\delta<1}\sup_{0\leq k\leq N}\Ee[1+|Y_{t_k}^\delta|^2]\leq C.
\]
Hence, using \eqref{eq:OU block mean},
\[
\Ee\left[\sum_{k=0}^{N-1}|\Ee[I_k\mid\mathcal G_k]|\right]
\leq C\frac{\delta}{\Delta}.
\]
Moreover, the sequence
\[
\sum_{k=0}^{n-1}\big(I_k-\Ee[I_k\mid\mathcal G_k]\big),\qquad n=0,\ldots,N,
\]
is a discrete-time martingale. By Doob's inequality and \eqref{eq:OU block variance},
\[
\Ee\left[\max_{0\leq n\leq N}\left|
\sum_{k=0}^{n-1}\big(I_k-\Ee[I_k\mid\mathcal G_k]\big)\right|\right]
\leq C\left(\sum_{k=0}^{N-1}\Ee|I_k-\Ee[I_k\mid\mathcal G_k]|^2\right)^{1/2}
\leq C\sqrt{\delta}.
\]
It remains only to pass from the discrete maximum to the continuous-time supremum. For $u\in[t_k,t_{k+1}]$, set
\[
L_k:=\sup_{t_k\leq u\leq t_{k+1}}
\left|\int_{t_k}^u\Phi(t_k,X_{t_k}^\delta,Y_s^\delta)\,ds\right|.
\]
Then, by Cauchy's inequality and the uniform moment bounds for $X^\delta$ and $Y^\delta$,
\[
\Ee[L_k^2]
\leq \Delta \int_{t_k}^{t_{k+1}}
\Ee\left[|\Phi(t_k,X_{t_k}^\delta,Y_s^\delta)|^2\right]\,ds
\leq C\Delta^2.
\]
Consequently,
\[
\Ee\left[\max_{0\leq k<N}L_k\right]
\leq \Ee\left[\left(\sum_{k=0}^{N-1}L_k^2\right)^{1/2}\right]
\leq \left(\sum_{k=0}^{N-1}\Ee[L_k^2]\right)^{1/2}
\leq C\Delta^{1/2}.
\]
Together with \eqref{eq:OU freezing error sup}, this yields
\[
\Ee\left[\sup_{0\leq t\leq T}|R_t^\delta|\right]
\leq C\left(\Delta^{1/2}+\frac{\delta}{\Delta}+\sqrt{\delta}\right).
\]
Choosing $\Delta=\delta^{2/3}$ gives
\[
\Ee\left[\sup_{0\leq t\leq T}|R_t^\delta|\right]\leq C\delta^{1/3}.
\]
Combining this estimate with \eqref{eq:OU A M bound direct}, we obtain
\[
\Ee\left[\sup_{0\leq t\leq T}|X_t^\delta-X_t^0|\right]\leq C\delta^{1/3}.
\]
\end{proof}

\subsection{Proof of Proposition \ref{pro:Strong rate diffusion--c}.}
\begin{proof}
We adapt the proof of Theorem 2.1 in \cite{rockner2021averaging}. Let $\Lambda \in \mathcal{A}^{\mathcal{R}}(\mathbb{F}^{\bm{W}})$, and denote by $X^{\delta}$ and $X^{0}$ the unique strong solutions of \eqref{eq:dynamics 2S--c} and \eqref{eq:dynamics effective--cR}, respectively, where the superscript $\Lambda$ is omitted for clarity. Let $\hat{\mathcal{L}}(t,x):= \bar{b}_1(t,x) \partial_x + \frac{1}{2}\sigma^2(t,x) \partial_{xx}$, and consider on $[0,T]\times \Rr$  the backward equation
\begin{align}
\label{eq:bPDE}    
\partial_t v(t,x) + \hat{\mathcal{L}}v(t,x) + \bar{b}_1(t,x) =0, \quad v(T,x) =0.
\end{align}
By Lemma 3.3 in \cite{rockner2021averaging}, $\bar{b}_1 \in C^{1/2,1}_b([0,T] \times \Rr)$, and by Assumption \ref{hyp:Dynamics_Poisson_X_Y--c}--\ref{hyp:RocknerAG--c2}, there exists a solution $v \in L^{\infty}([0,T];C^3_b(\Rr)) \cap C^{3/2}_b([0,T];L^{\infty}(\Rr))$ (see equation (5.2) and Theorem 5.1 in Chapter IV, Section 5 of \cite{ladyzhenskaia1968linear} and Theorem 2 in \cite{fedrizzi2011pathwise}) such that for $T$ small, 
\begin{align}
\label{eq:Aux Control 0}
|\partial_x v(t,x)| \leq \tfrac{1}{2}
\end{align} 
for all $(t,x) \in [0,T]\times \Rr$. In the following, we apply the arguments for such $T$. The result for general $T$ is obtained by iteratively extending this result. Set $\Gamma(t,x):= x + v(t,x)$ and define
\begin{align}\label{def1}
V^\delta_t:= \Gamma(t,X^\delta_t),\,\,V^0_t:= \Gamma(t,X^0_t).
\end{align}
Let $\tilde{b}(t,x,y):=b_1(t,x,y)-\bar{b}_1(t,x)$. Then, the map $x\mapsto \Gamma(t,x)$ is a $C^1$--diffeomorphism for every $t \in [0,T]$, $\tfrac{1}{2} \leq |\partial_x \Gamma(t,x)| \leq \tfrac{3}{2}$, $V^\delta_0 = V^0_0$, and using \eqref{eq:bPDE}, we have
\begin{align*}
dV^\delta_t & = \left(\! \partial_t v(t,X^\delta_t) + (\partial_x v(t,X^\delta_t) + 1) \!\left(\! b_1(t,X^\delta_t,Y^\delta_t) +\!\int_{A} \!\!b_2(X^\delta_t,a) \Lambda_t(da) \!\right) + \tfrac{1}{2}\partial_{xx}v(t,X^\delta_t)\sigma^2(t,X^\delta_t) \!\right)\! dt 
\nonumber \\
& \quad +(1+\partial_x v(t,X_t^\delta))\sigma(t,X_t^\delta)dW_t^{\text{x}} \nonumber \\
& = \partial_x \Gamma(t,X^\delta_t) \!\left(\! (\tilde{b}(t,X^\delta_t,Y^\delta_t) + \!\int_{A} \!\! b_2(X^\delta_t,a)\Lambda_t(da) ) dt + \sigma(t,X^\delta_t) dW^{\text{x}}_t \!\right),
\\
dV^0_t & = \left(\! \partial_t v(t,X^0_t) + (\partial_x v(t,X^0_t) + 1) \left(\! \bar{b}_1(t,X^0_t) + \!\int_{A} \!\! b_2(X^0_t,a)\Lambda_t(da) \!\right) + \tfrac{1}{2}\partial_{xx}v(t,X^0_t)\sigma^2(t,X^0_t) \!\right)\! dt 
\\
& \quad + (1+\partial_x v(t,X^0_t) ) \sigma(t,X^0_t) dW^{\text{x}}_t 
= \partial_x \Gamma(t,X^0_t) \left( \!\int_{A} \!\! b_2(X^0_t,a)\Lambda_t(da) dt + \sigma(t,X^0_t) dW^{\text{x}}_t \!\right),
\end{align*}
which shows that
\begin{align}
\label{eq:Aux Control 1}
V^\delta_t - V^0_t & = \!\int_0^t \!\!\left(\! \partial_x \Gamma(s,X^\delta_s) ( \tilde{b}(s,X^\delta_s,Y^\delta_s) + \!\int_{A}\!\! b_2(X^\delta_s,a)\Lambda_s(da) ) - \partial_x \Gamma(s,X^0_s) \!\int_{A}\!\! b_2(X^0_s,a)\Lambda_s(da) \!\right) \!ds \notag
\\
& \quad + \!\int_0^t\!\! ( \partial_x \Gamma(s,X^\delta_s) \sigma(s,X^\delta_s) - \partial_x \Gamma(s,X^0_s) \sigma(s,X^0_s) ) dW^{\text{x}}_s.
\end{align}
Using the regularity of $v$, the Lipschitz assumption on $b_1$ and $b_2$, The Lipschitz continuity of $\partial_x \Gamma$, and boundedness of $b_2$ and $\partial_x \Gamma$, we have
\begin{align}
\label{eq:Aux Control 11}
& \left| \int_0^t \!\left(\! \partial_x \Gamma(s,X^\delta_s) ( \tilde{b}(s,X^\delta_s,Y^\delta_s) + \!\int_{A}\!\! b_2(X^\delta_s,a)\Lambda_s(da) ) - \partial_x \Gamma(s,X^0_s) \!\int_{A}\!\! b_2(X^0_s,a)\Lambda_s(da) \! \right)\! ds \right|^2 \notag
\\
& \leq C \!\left( \left| \int_0^t\!\! (\partial_x \Gamma \, \tilde{b})(s,X^\delta_s,Y^\delta_s) ds \right|^2 \!+ \left| \int_0^t\!\int_{A} \! ( (\partial_x \Gamma \, b_2)(s,X^\delta_s,a) - (\partial_x \Gamma \, b_2)(s,X^0_s,a) ) \Lambda_s(da) ds \right|^2 \right) \notag
\\
& \leq C \!\left( \left| \int_0^t \!(\partial_x \Gamma\, \tilde{b})(s,X^\delta_s,Y^\delta_s) ds \right|^2 \!+ \bigg( \!\int_0^t \!\! \int_{A} \!\left( \left| \partial_x \Gamma(s,X^\delta_s)\right| \left| b_2(X^\delta_s,a) - b_2(X^0_s,a) \right|  \right. \right.  \notag
\\
& \hskip6.2cm \left.\left. + \left| \partial_x \Gamma(s,X^\delta_s) - \partial_x \Gamma(s,X^0_s)\right| |b_2(X^0_s,a)|\right) \Lambda_s(da) ds \!\bigg)^2 \right) \notag
\\
& \leq C \!\left( \left| \int_0^t \! (\partial_x \Gamma \, \tilde{b}) (s,X^\delta_s,Y^\delta_s) ds \right|^2 \!+ \left( \int_0^t \!C |X^\delta_s-X^0_s| ds \!\right)^2 \right).
\end{align}
Because $\langle (\partial_x \Gamma\, \tilde{b})(t,x,y), \bar{\mu}\rangle =0$, $(\partial_x \Gamma\, \tilde{b})(\cdot,\cdot,y)\in C^{1/2,1}_b([0,T] \times \Rr)$ for all $y\in \Rr$, and $(\partial_x \Gamma\, \tilde{b})(t,x,\cdot)$ is Lipschitz on $y$ uniformly on $(t,x)$, by Assumption \ref{hyp:Dynamics_Poisson_X_Y--c}--\ref{hyp:RocknerAs--c2}--\ref{hyp:RocknerAb--c2}, Lemma 4.1 in \cite{rockner2021averaging} gives 
\begin{align}
\label{eq:Aux Control 2}
\Ee \left[ \left| \int_0^t \partial_x \Gamma(s,X^\delta_s) \tilde{b}(s,X^\delta_s,Y^\delta_s) ds \right|^2 \right] \leq C (\sqrt{\delta} + \delta)^2.  
\end{align}
For the integral w.r.t. $W^{\text{x}}$ in \eqref{eq:Aux Control 1}, since $x\mapsto (\partial_x \Gamma\, \sigma)(t,x) \in C^{1}_b(\Rr)$, It\^{o}'s isometry gives
\begin{align}
\label{eq:Aux Control 3}
\Ee \left[\int_0^t \!\! ( (\partial_x \Gamma\, \sigma) (s,X^\delta_s) - (\partial_x \Gamma\, \sigma) (s,X^0_s) )^2 ds\right] \leq C \Ee \left[ \int_0^t \!\!|X^\delta_s - X^0_s|^2 ds \right].
\end{align}
From the definitions of $V^\delta$ and $V^0$, and recalling that $x\mapsto \Gamma(t,x)$ is a $C^1$--diffeomorphism for every $t\! \in\! [0,T]$, the mean–value theorem, together with \eqref{eq:Aux Control 0}, \eqref{eq:Aux Control 1}, \eqref{eq:Aux Control 11}, \eqref{eq:Aux Control 2}, and \eqref{eq:Aux Control 3}, give
\begin{align*}
& \Ee \left[ |X^\delta_t - X^0_t|^2 \right] \leq C \Ee \left[ |V^\delta_t - V^0_t|^2 \right] \leq C \left( (\sqrt{\delta} + \delta)^2 +  \int_0^t \Ee \left[|X^\delta_s-X^0_s|^2\right] ds  \right).
\end{align*}
Applying Gr\"{o}nwall inequality yields the result for $T$ satisfying \eqref{eq:Aux Control 0}. The general case follows by iteratively extending this argument.
\end{proof}

\subsection{Technical lemmas}
\begin{lemma}\label{lemmaM}${}$
\begin{enumerate}
\item  Let Assumptions \ref{hyp:Dynamics_Poisson_X_Y} and \ref{hyp:Cost Poisson OS} hold and let $\tau \in \mathcal T(\mathbb F^{\mathbf W})$. Then, there exists $C<\infty$ such that 
  \begin{equation}
  \label{eq:tildeA1 bound final}
\left|\mathbb E\left[\int_0^\tau   F^2(t,X^0_t,\langle \hat f_1,m^\star_t\rangle)(F^1(t,X^\delta_t,Y^\delta_t) - \bar F^1(t,X^\delta_t)) dt\right]\right| \leq C \sqrt{\delta}. 
  \end{equation}
\item Let  Assumptions \ref{hyp:Dynamics_Poisson_X_Y--c} and \ref{hyp:f and g--c} hold and let $\Lambda \in \mathcal{A}^{\mathcal{R}}(\mathbb{F}^{\bm{W}})$. Then there exists $C<\infty$ such that 
$$
\left|\mathbb E\left[\int_0^T F^2(t,X^{0,\Lambda}_t, \langle \hat{f}_1, m_t^{\star(\text{x})} \rangle ) \left( F^1(t,X^{\delta,\Lambda}_t,Y^{\delta}_t) - \bar{F}^1(t,X^{\delta,\Lambda}_t)\right)dt\right]\right|\leq C\sqrt{\delta}.
$$
\end{enumerate}  
\end{lemma}  
\begin{proof}${}$\\
\emph{Part 1.}\quad  Let $\hat F_t: =\langle \hat f_1,m^\star_t\rangle = \mathbb E[\hat f_1(X^0_t) \mathbf 1_{(t,T]}(\tau^\star)]$. Similarly to \cite{1Bouveret} (see Lemma 3.3), it can be proved that $\hat{F}$ is a function of bounded variation on $[0,T]$. Indeed, let $\psi \in C^1([0,T])$. Then, by integration by parts, 
\begin{align*}
  &\int_0^T \hat F_t \psi'(t) dt =   \mathbb E\left[\int_0^{\tau^\star} \hat f_1(X^0_t) \psi'(t) dt\right] = \mathbb E\left[\hat f_1(X^0_{\tau^\star})\psi(\tau^\star) - \hat f_1(X^0_0)\psi(0)-\int_0^{\tau^\star} \psi(t) d\hat f_1(X^0_t) \right] \\
  & = \mathbb E\left[\hat f_1(X^0_{\tau^\star})\psi(\tau^\star) - \hat f_1(X^0_0)\psi(0)-\int_0^{\tau^\star} \psi(t) \hat f'_1(X^0_t) \bar b(t,X^0_t) dt -\frac{1}{2}\int_0^{\tau^\star} \psi(t) \hat f^{\prime\prime}_1(X^0_t) \sigma^2(t,X^0_t) dt\right]\\ &\leq C \|\psi\|_\infty
\end{align*}
by our assumptions on $\hat f_1$. 

Using the notation of Lemma \ref{lem:supL2_H} and performing integration by parts, we have:
\begin{align*}
  \int_0^\tau F^2(t,X^0_t,\hat F_t)dH^\delta_t &=F^2(\tau,X^0_\tau,\hat F_\tau) H^\delta_\tau  - \int_0^\tau H^\delta_t F^2_t(t,X^0_t,\hat F_t) dt \\&- \int_0^\tau H^\delta_t F^2_x(t,X^0_t,\hat F_t)\bar b(t,X^0_t) dt -\int_0^\tau H^\delta_t F^2_x(t,X^0_t,\hat F_t)\sigma(t,X^0_t) dW^{\mathrm x}_t  \\&- \frac{1}{2}\int_0^\tau H^\delta_t F^2_{xx}(t,X^0_t,\hat F_t) \sigma^2(t,X^0_t) dt - \int_0^\tau H^\delta_t F^2_f(t,X^0_t,\hat F_t) d\hat F^c_t\\
   &- \sum_{0\leq t\leq \tau: \Delta \hat F_t \neq 0} H^\delta_t \{F^2(t,X^0_t,\hat F_t)-F^2(t,X^0_t,\hat F_{t-})\},
\end{align*}
where $F^2_f$ denotes the derivative with respect to the third argument, and $\hat F^c$ is the continuous component of $F^c$.  Since by Assumption  \ref{hyp:Cost Poisson OS}, $F^2 \in C^{1,2,1}_b$, the terms in the first line are bounded by $C\sqrt{\delta}$ for some $C<\infty$ by Lemma \ref{lem:supL2_H}. For the first term in the second line, we have, also using Lemma \ref{lem:supL2_H}:
\begin{align*}
  \mathbb E\left[\left| \int_0^\tau H^\delta_t F^2_x(t,X^0_t,\hat F_t)\bar b(t,X^0_t) dt\right|\right]&\leq C \mathbb E\left[ \int_0^T \left| H^\delta_t \right|^2 dt\right]^{\frac{1}{2}} \mathbb E\left[ \int_0^T \left| \bar b(t,X^0_t) \right|^2dt\right]^{\frac{1}{2}}\\
  &\leq C \mathbb E\left[ \int_0^T \left| \bar b(t,X^0_t) \right|^2dt\right]^{\frac{1}{2}} \sqrt{\delta},
\end{align*}
which can be further estimated using uniform Lipschitz bounds on $\bar b$. For the second term in the second line, we write:
$$
\mathbb E\left[\left| \int_0^\tau H^\delta_t F^2_x(t,X^0_t,\hat F_t) \sigma(t,X^0_t) dW^{\mathrm x}_t\right|\right]\leq \mathbb E\left[ \int_0^T |H^\delta_t|^2 |F^2_x(t,X^0_t,\hat F_t)|^2 |\sigma(t,X^0_t)|^2 dt\right]^{\frac{1}{2}},
$$
which we estimate as above, using the uniform bound on $\sigma$. The first term in the third line is treated similarly. For the second term in the third line, we write
$$
\mathbb E\left[\left|\int_0^\tau H^\delta_t F^2_f(t,X^0_t,\hat F_t) d\hat F^c_t\right|\right] \leq \mathbb E\left[ \|H^\delta_t\|_{\infty} \|F^2_f\|_{\infty} TV_{[0,T]}(\hat F^c)\right]<C \sqrt{\delta},
$$
using the properties of $F^2$, the bounded variation of $\hat F$ and Lemma  \ref{lem:supL2_H}. The sum in the 4th line is estimated similarly. \\

\noindent\emph{Part 2.}\quad  Now let $\hat F_t: =\langle \hat{f}_1, m_t^{\star(\text{x})} \rangle = \mathbb E[\hat f_1(X^{0,\Lambda^*}_t) ]$. By integration by parts, for $\psi \in C^1([0,T])$, 
\begin{align*}
  &\int_0^T \hat F_t \psi'(t) dt =   \mathbb E\left[\int_0^T \hat f_1(X^{0,\Lambda^*}_t) \psi'(t) dt\right] = \mathbb E\left[\hat f_1(X^{0,\Lambda^*}_{T})\psi(T) - \hat f_1(X^{0,\Lambda^*}_0)\psi(0)-\int_0^{T} \psi(t) d\hat f_1(X^{0,\Lambda^*}_t) \right] \\
  & = \mathbb E\left[\hat f_1(X^{0,\Lambda^*}_{T})\psi(T) - \hat f_1(X^{0,\Lambda^*}_0)\psi(0)-\int_0^{T} \psi(t) \hat f'_1(X^{0,\Lambda^*}_t) \bar b_1(t,X^{0,\Lambda^*}_t) dt\right]\\ &+\mathbb E\left[-\int_0^{T} \psi(t) \hat f'_1(X^{0,\Lambda^*}_t) \int_A b_2(X^{0,\Lambda^*}_t,a)\Lambda^*_t(da) dt -\frac{1}{2}\int_0^{T} \psi(t) \hat f^{\prime\prime}_1(X^{0,\Lambda^*}_t) \sigma^2(t,X^{0,\Lambda^*}_t) dt\right]\\ &\leq C \|\psi\|_\infty
\end{align*}
by our assumptions on $\hat f_1$ and the coefficients of $X$. Moreover, the dominated convergence theorem implies that $\hat F$ is continuous. 

Performing integration by parts as in the first part of the proof, we then have:
\begin{align*}
  \int_0^T F^2(t,X^{0,\Lambda^*}_t,\hat F_t)dH^\delta_t &=F^2(T,X^{0,\Lambda^*}_T,\hat F_T) H^\delta_T  - \int_0^T H^\delta_t F^2_t(t,X^{0,\Lambda^*}_t,\hat F_t) dt \\&- \int_0^T H^\delta_t F^2_x(t,X^{0,\Lambda^*}_t,\hat F_t)\bar b_1(t,X^{0,\Lambda^*}_t) dt\\ &- \int_0^T H^\delta_t F^2_x(t,X^{0,\Lambda^*}_t,\hat F_t)\int_A b_2(X^{0,\Lambda^*}_t,a)\Lambda_t^*(da) dt \\ &-\int_0^T H^\delta_t F^2_x(t,X^{0,\Lambda^*}_t,\hat F_t)\sigma(t,X^{0,\Lambda^*}_t) dW^{\mathrm x}_t  \\&- \frac{1}{2}\int_0^T H^\delta_t F^2_{xx}(t,X^{0,\Lambda^*}_t,\hat F_t) \sigma^2(t,X^{0,\Lambda^*}_t) dt - \int_0^T H^\delta_t F^2_f(t,X^{0,\Lambda^*}_t,\hat F_t) d\hat F_t
\end{align*}
The proof now follows as in the first part using Lemma \ref{lem:supL2_H}. 

\end{proof}

\begin{lemma}[Sup--norm $L^2$ bound for centered fast fluctuations]
\label{lem:supL2_H}
Let $\Lambda \in \mathcal{A}^{\mathcal{R}}(\mathbb{F}^{\bm{W}})$ and let $(X_t^{\delta,\Lambda},Y_t^\delta)_{t\in[0,T]}$ follow the controlled dynamics \eqref{eq:relaxed}. 
Let $F^1:[0,T]\times \mathbb R\times\R\to\R$ be such that $F^1(\cdot,\cdot,y) \in C^{1/2,1}_b$ for all $y\in \R$ and
$$
|F^1(t,x,y_1) - F^1(t,x,y_2)|\leq |y_1-y_2| (1+|y_1|^m + |y_2|^m)
$$
for some $m>0$ and all $y_1,y_2\in \mathbb R$. 
Define
\[
H_t^\delta := \int_0^t \Big(F^1(s,X_s^{\delta,\Lambda},Y_s^\delta)-\bar F^1(s,X_s^{\delta,\Lambda})\Big)\,ds,
\qquad t\in[0,T].
\]
Then there exists a constant $C>0$ independent of $\delta\in(0,1]$ such that
\[
\E\Big[\sup_{0\le t\le T}|H_t^\delta|^2\Big]\le C\,\delta.
\]
\end{lemma}
\begin{proof}

Set
\[
\Phi(t,x,y):=F^1(t,x,y)-\bar F^1(t,x), \qquad (t,x,y)\in [0,T]\times \mathbb R\times \mathbb R.
\]
By construction, $\int_{\mathbb R}\Phi(t,x,y)\bar\mu(dy)=0$ for every $(t,x)\in [0,T]\times \mathbb R$. We now follow the Poisson-equation argument used in the proof of Lemma 4.1 in \cite{rockner2021averaging}. More precisely, for every $(t,x)$ there exists a centered solution $u(t,x,\cdot)$ of
\[
\mathcal L_Y u(t,x,\cdot)=\Phi(t,x,\cdot), \qquad \int_{\mathbb R}u(t,x,y)\bar\mu(dy)=0,
\]
and the proof of \cite[Lemma 4.1]{rockner2021averaging} yields polynomial-growth bounds on $u$ and its derivatives in the variables that appear below. In particular, for some integer $\ell\geq 1$ and some constant $C>0$ independent of $(t,x,y)$,
\begin{equation}
\label{eq:altproof-corrector-bound}
|u(t,x,y)|+|\partial_t u(t,x,y)|+|\partial_x u(t,x,y)|+|\partial_{xx}u(t,x,y)|+|\partial_y u(t,x,y)|
\leq C(1+|y|^\ell).
\end{equation}

Next, let
\[
\tilde b^\Lambda(t,x,y,\omega):=b_1(t,x,y)+\int_A b_2(x,a)\Lambda_t(da)(\omega).
\]
Since $b_2$ is bounded by Assumption \ref{hyp:Dynamics_Poisson_X_Y--c} and $\Lambda_t\in\mathcal P(A)$, the control term is uniformly bounded in $(t,\omega)$, and therefore the usual moment estimates for \eqref{eq:relaxed} imply that for every $p\geq 2$,
\begin{equation}
\label{eq:altproof-moment-bound}
\sup_{\delta\in(0,1]}\E\left[\sup_{0\leq s\leq T}|X_s^{\delta,\Lambda}|^p+\sup_{0\leq s\leq T}|Y_s^\delta|^p\right]<\infty.
\end{equation}
In particular, the constants in these estimates are uniform in the admissible relaxed control $\Lambda$.

Applying It\^o's formula to the process $u(s,X_s^{\delta,\Lambda},Y_s^\delta)$ and using the identity $\mathcal L_Yu=\Phi$, we obtain for every $t\in[0,T]$,
\begin{align*}
\int_0^t \Phi(s,X_s^{\delta,\Lambda},Y_s^\delta)\,ds
&= \delta\Big(u(t,X_t^{\delta,\Lambda},Y_t^\delta)-u(0,X_0,Y_0)\Big)
\\
&\quad -\delta\int_0^t \Big(\partial_su+\tilde b^\Lambda(s,X_s^{\delta,\Lambda},Y_s^\delta)\partial_xu+\tfrac12\sigma^2(s,X_s^{\delta,\Lambda})\partial_{xx}u\Big)(s,X_s^{\delta,\Lambda},Y_s^\delta)\,ds
\\
&\quad -\delta\int_0^t \sigma(s,X_s^{\delta,\Lambda})\partial_xu(s,X_s^{\delta,\Lambda},Y_s^\delta)\,dW_s^{\mathrm x}
\\
&\quad -\sqrt{\delta}\int_0^t \Sigma(Y_s^\delta)\partial_yu(s,X_s^{\delta,\Lambda},Y_s^\delta)\,dW_s^{\mathrm y}
-\sqrt{\delta}\int_0^t \hat\Sigma(Y_s^\delta)\partial_yu(s,X_s^{\delta,\Lambda},Y_s^\delta)\,dW_s^{\mathrm c}.
\end{align*}
Recalling the definition of $H_t^\delta$, this can be rewritten as
\[
H_t^\delta=R_t^{\delta,1}+R_t^{\delta,2}+M_t^{\delta,\mathrm x}+M_t^{\delta,\mathrm y}+M_t^{\delta,\mathrm c},
\]
with the obvious notation.

We estimate these terms separately. First, by \eqref{eq:altproof-corrector-bound}, \eqref{eq:altproof-moment-bound}, and the fact that $\delta\leq 1$,
\begin{align*}
\E\Big[\sup_{0\leq t\leq T}|R_t^{\delta,1}|^2\Big]
&\leq C\delta^2 \E\Big[\sup_{0\leq t\leq T}(1+|Y_t^\delta|^{2\ell})\Big]
\leq C\delta^2,
\\
\E\Big[\sup_{0\leq t\leq T}|R_t^{\delta,2}|^2\Big]
&\leq C\delta^2 \E\Big[\Big(\int_0^T (1+|Y_s^\delta|^\ell)\,ds\Big)^2\Big]
\leq C\delta^2.
\end{align*}
The boundedness of the control term enters only through $\tilde b^\Lambda$, hence the same constant $C$ works uniformly in $\Lambda$.

For the martingale driven by $W^{\mathrm x}$, the Burkholder--Davis--Gundy inequality, the boundedness of $\sigma$, and again \eqref{eq:altproof-corrector-bound}--\eqref{eq:altproof-moment-bound} give
\begin{align*}
\E\Big[\sup_{0\leq t\leq T}|M_t^{\delta,\mathrm x}|^2\Big]
&\leq C\delta^2 \E\left[\int_0^T |\sigma(s,X_s^{\delta,\Lambda})|^2 |\partial_xu(s,X_s^{\delta,\Lambda},Y_s^\delta)|^2\,ds\right]
\leq C\delta^2.
\end{align*}
Similarly, for the fast martingale terms,
\begin{align*}
\E\Big[\sup_{0\leq t\leq T}|M_t^{\delta,\mathrm y}|^2\Big]
&\leq C\delta \E\left[\int_0^T |\Sigma(Y_s^\delta)|^2 |\partial_yu(s,X_s^{\delta,\Lambda},Y_s^\delta)|^2\,ds\right]
\leq C\delta,
\\
\E\Big[\sup_{0\leq t\leq T}|M_t^{\delta,\mathrm c}|^2\Big]
&\leq C\delta \E\left[\int_0^T |\hat\Sigma(Y_s^\delta)|^2 |\partial_yu(s,X_s^{\delta,\Lambda},Y_s^\delta)|^2\,ds\right]
\leq C\delta.
\end{align*}
Combining the above bounds and using $(a_1+\cdots+a_5)^2\leq 5(a_1^2+\cdots+a_5^2)$, we conclude that
\[
\E\Big[\sup_{0\leq t\leq T}|H_t^\delta|^2\Big]\leq C\delta.
\]
This proves the claim.

\end{proof}

\section{Acknowledgement}
This work has been carried out at the Energy4Climate Interdisciplinary Center (E4C) of IP Paris, which is in part supported by 3rd Programme d'Investissements d'Avenir [ANR-18-EUR-0006-02], and by the Foundation of Ecole Polytechnique (Chaire ``Décarboner l'économie'' financed by BNP Paribas).

\end{document}